\documentclass[UTF-8,reqno]{amsart}
\usepackage{enumerate}
\usepackage{mhequ}
\usepackage[margin=1.2in]{geometry}
\usepackage{amssymb,url,color, booktabs,nccmath}
\usepackage{mathrsfs}
\usepackage{enumitem}
\usepackage{graphicx}
\usepackage{tikz}
\usepackage{verbatim}
\usepackage{amsfonts}
\usepackage{bbold}
\usetikzlibrary{shapes,snakes}
\usetikzlibrary{calc}
\usetikzlibrary{decorations.shapes}
\usepackage{color}
\usepackage[colorlinks=true]{hyperref}
\hypersetup{
	linkcolor=blue,          
	citecolor=red,        
	filecolor=blue,      
	urlcolor=cyan
}
\newcommand{\blue}{\color{blue}}

\newcommand{\black}{\color{black}}

\definecolor{darkergreen}{rgb}{0.0, 0.5, 0.0}

\numberwithin{equation}{section}

\newcommand{\be}{\begin{eqnarray}}
	\newcommand{\ee}{\end{eqnarray}}
\newcommand{\ce}{\begin{eqnarray*}}
	\newcommand{\de}{\end{eqnarray*}}
\newtheorem{theorem}{Theorem}[section]
\newtheorem{lemma}[theorem]{Lemma}
\newtheorem{remark}[theorem]{Remark}
\newtheorem{definition}[theorem]{Definition}
\newtheorem{proposition}[theorem]{Proposition}
\newtheorem{Examples}[theorem]{Example}
\newtheorem{corollary}[theorem]{Corollary}
\newtheorem{assumption}{Assumption}[section]

\newenvironment{nouppercase}{%
	\renewcommand{\uppercasenonmath}[1]{}}{}

\def\[{{\Big[}}
\def\]{{\Big]}}
\def\<{{\langle}}
\def\>{{\rangle}}
\def\({{\Big(}}
\def\){{\Big)}}

\def\bx{{\mathbf{x}}}

\def\sgn{\mbox{\rm sgn}}

\def\min{{\mathord{{\rm min}}}}

\def\={&\!\!=\!\!&}

\def\1{{\mathbf{1}}}

\def\geq{\geqslant}
\def\leq{\leqslant}
\def\ge{\geqslant}
\def\le{\leqslant}

\def\[{{\Big[}}
\def\]{{\Big]}}
\def\<{{\langle}}
\def\>{{\rangle}}
\def\({{\Big(}}
\def\){{\Big)}}

\def\bx{{\mathbf{x}}}

\def\sgn{\mbox{\rm sgn}}

\def\min{{\mathord{{\rm min}}}}

\def\={&\!\!=\!\!&}
\def\bt{\begin{theorem}}
	\def\et{\end{theorem}}
\def\bl{\begin{lemma}}
	\def\el{\end{lemma}}
\def\br{\begin{remark}}
	\def\er{\end{remark}}
\def\bx{\begin{Examples}}
	\def\ex{\end{Examples}}
\def\bd{\begin{definition}}
	\def\ed{\end{definition}}
\def\bp{\begin{proposition}}
	\def\ep{\end{proposition}}
\def\bc{\begin{corollary}}
	\def\ec{\end{corollary}}
\def\bc{\begin{assumption}}
	\def\ec{\end{assumption}}
\def\bc{\begin{proposition}}
	\def\ec{\end{proposition}}

\def\geq{\geqslant}
\def\leq{\leqslant}
\def\ge{\geqslant}
\def\le{\leqslant}

\def\<{\langle} \def\>{\rangle}

\allowdisplaybreaks

\tikzset{
	dot/.style={circle,fill=black,inner sep=0pt, outer sep=0.7pt, minimum size=1mm},
	Phi/.style={white!40!red,thick,snake=coil,segment amplitude=0.6pt, segment length=2pt},
	Z/.style={black!40!green,thick,snake=coil,segment amplitude=0.6pt, segment length=2pt},
	C/.style={thick,black!20!blue},
	Cr/.style={thick,black!20!red},
	Cg/.style={thick,black!20!green},
}

\begin{document}
	
	\title[Dean-Kawasaki Equation with Singular Interactions]
	{\LARGE {\color{black} Well-posedness of Dean-Kawasaki Equation with Singular Interactions}}

	\author[Likun Wang]{\large Likun Wang}
	\address[L. Wang]{Institute of Automation, Chinese Academy of Sciences, Beijing, 100190, China}
	\email{likun.wang@ia.ac.cn}
	
	\author[Zhengyan Wu]{\large Zhengyan Wu}
\address[Z. Wu]{Department of Mathematics, Technische Universit\"at M\"unchen, Boltzmannstr. 3, 85748 Garching, Germany}
\email{wuzh@cit.tum.de} 
	\author{Rangrang Zhang}
	\address[R. Zhang]{School of Mathematics and Statistics,
		Beijing Institute of Technology, Beijing, 100081, China
	}
	\email{rrzhang@amss.ac.cn}

	\begin{abstract}
		Inspired by [Fehrman, Gess; Invent. Math., 2023] and [Fehrman, Gess; Arch. Ration. Mech. Anal., 2024], we consider the Dean-Kawasaki equation with singular interactions and correlated noise which can be viewed as fluctuating mean-field limits. By imposing the Ladyzhenskaya-Prodi-Serrin condition on the interaction kernel, the existence of probabilistic weak renormalized kinetic solutions is established. Further, under an additional integrability assumption on the divergence of the interaction kernel, a kinetic formulation approach is applied to derive pathwise uniqueness, leading to the strong well-posedness of the equation. As an application, we obtain the well-posedness of a conservative stochastic partial differential equation known as the fluctuating Ising-Kac-Kawasaki dynamics.
			\end{abstract}

	\subjclass[2010]{60H15; 35R60}
	\keywords{Dean-Kawasaki equation, singular interactions, Fluctuating Ising-Kac-Kawasaki equation, well-posedness}
	
	\date{\today}
	
	\begin{nouppercase}
		\maketitle
	\end{nouppercase}
	
	\setcounter{tocdepth}{1}

	\section{Introduction}
	Fluctuating hydrodynamics provides a framework for simulating microscopic fluctuations by combining statistical mechanics and nonequilibrium thermodynamics. This framework leads to various conservative-type stochastic partial differential equations (SPDEs), which represent the fluctuation corrections of hydrodynamic limits of interacting particle systems, characterized by fluctuation-dissipation relations. In general, the fluctuation phenomena of these SPDEs is similar to that of the corresponding interacting particle systems. Taking into account the theory of fluctuating hydrodynamics, we consider a regularized version of Dean-Kawasaki equation, which can be viewed  as fluctuating mean-field limits of the following mean-field systems:
	\begin{equation}\label{IPS}
		\mathrm{d}X_i=-\frac{1}{N}\sum_{j=1}^N\nabla \mathcal{U}(X_i-X_j)\mathrm{d}t+\sqrt{2}\mathrm{d}B_i,\ \ i=1,\dots,N,
	\end{equation}
where $\{B_i\}_{i=1}^N$ represents independent Brownian motions, $\mathcal{U}$ denotes the interaction kernel. Let $\pi_N$ be the empirical measure defined by $\pi_N := \frac{1}{N}\sum_{i=1}^N\delta_{X_i}$. To approximate $\pi_N$, the Dean-Kawasaki equation was proposed by Dean \cite{Dea96} and Kawasaki \cite{Kaw98} based on the covariance structure of the noise. The Dean-Kawasaki equation takes the following form:
	\begin{equation}\label{eq-1.2}
		\partial_t\pi_N=\Delta\pi_N+\nabla\cdot(\pi_N\nabla \mathcal{U}\ast\pi_N)-\sqrt{2}N^{-1/2}\nabla\cdot(\sqrt{\pi_N}\xi),
	\end{equation}
where $\xi$ is a space-time white noise. Note that  (\ref{eq-1.2}) is a conservative supercritical singular stochastic PDE, which is ill-posed in the theories of Hairer's regularity structure \cite{Hai14} and Gubinelli, Imkeller, Perkowski's paracontrolled distribution \cite{GIP15}. The existence and uniqueness (in law) of trivial martingale solution to
	(\ref{eq-1.2}) was shown by Konarovskyi, Lehmann and von Renesse \cite{KLvR19} under the condition that $N$ is a non-negative integer.
	Specifically, the authors proved that the empirical measure of particle system solves the martingale problem, which gives rigorous mathematical meaning to the Dean-Kawasaki equation (\ref{eq-1.2}).
	However, the well-posedness of functional-valued solutions of (\ref{eq-1.2}), in particular, the pathwise uniqueness is challenging, due to the irregularity of the square root function and the nonlocal term even if the noise is sufficiently regular in space.
	In the case of local interaction, the obstacle was resolved by Fehrmann and Gess \cite{FG21}. Concretely,
	the authors considered the following general Dean-Kawasaki equation written by
\begin{equation}\label{eq-1.3}
\partial_t\rho=\Delta\Phi(\rho)-\nabla\cdot\nu(\rho)-\nabla\cdot (\sigma(\rho)\circ\dot{W}^F),
\end{equation}
where $\circ$ represents the Stratonovich integral. The nonlinearity is in the form $\Phi(\xi) = \xi^m$ for all $m \geq 1$ corresponding to the hydrodynamic limit of zero-range process with a mean-local jump rate $\Phi$ (see Kipnis and Landim \cite[Chapters 3 and 5]{KL99}). The nonlinear flux function $\nu(\cdot)$ is of Burgers type. The noise coefficient $\sigma(\cdot)$ is locally 1/2-H\"older continuous and the noise $\dot{W}^F$ is sufficiently regular in space and white in time.

	To the best of our knowledge, the well-posedness of the correlated Dean-Kawasaki equation with nonlocal interactions, especially with singular interaction kernels, still remains open. In this paper, we devote to proving the well-posedness of the following equation:
\begin{equation}\label{eq-1.1}
\partial_t \rho = \Delta \rho - \nabla \cdot (\rho V(t) \ast \rho) - \nabla \cdot (\sqrt{\rho} \circ \dot{W}^F).
\end{equation}
In the above equation, the symbol $\ast$ represents spatial convolution and the interaction kernel $V=-\nabla\mathcal{U}$ is generalized to a singular nonlocal time-dependent function. More precisely, we consider the case of spatial dimension $d\geq2$, and let $T>0$ be a fixed time horizon in the whole context, the singular interaction kernel $V$ in  (\ref{eq-1.1}) is assumed to satisfy that
	\begin{description}
		\item[Assumption (A1)]$V\in L^{p^{*}}\left([0,T];L^p(\mathbb{T}^d;\mathbb{R}^d)\right)$ with  $\frac{d}{p}+\frac{2}{p^{*}}\le1$, $2\le p^*\leq\infty$ and $d<p\le\infty$,
		\item[Assumption (A2)] $\nabla\cdot V\in L^{q^{*}}\left([0,T];L^q(\mathbb{T}^d)\right)$ with $\frac{d}{2q}+\frac{1}{q^{*}}\le1$, $1\le q^*\leq\infty$ and $\frac{d}{2}<q\le\infty$.
	\end{description}

    It should be pointed out that Assumption (A1) cannot be derived from Assumption (A2). We take two-dimensional Biot-Savart kernel as the counterexample. For every $x\in\mathbb{T}^2$, the Biot-Savart kernel is of the form $V(x)=\frac{1}{2\pi}\sum_{k\in\mathbb{Z}^2/{\{0\}}}\frac{k^{\bot}}{|k|^2}e_k(x)$, where $\{e_k\}_{k\in\mathbb{Z}^2}$ is a sequence of Fourier basis of $L^2(\mathbb{T}^2)$. Obviously, $\nabla\cdot V=0$, however $V\in L^p(\mathbb{T}^2;\mathbb{R}^2)$ for any $p<2$, see \cite{BFM}.

	Let ${\rm{Ent}}(\mathbb{T}^d)$ be the space of finite entropy functions defined by (\ref{kk-44}). {\color{black}Let $W^F$ be the driving noise whose precise definition will be provided in (\ref{kk-43}) below.} Based on the above assumptions, the well-posedness of  (\ref{eq-1.1}) can be derived. It reads as follows.
	\begin{theorem}\label{thm-main1} (Uniqueness, cf. Theorem \ref{the-uniq})
		Assume that $V$ satisfies Assumptions (A1) and (A2). Let $\hat{\rho}^1,\hat{\rho}^2\in {\rm{Ent}}\left(\mathbb{T}^{d}\right)$. Let $\rho^1,\rho^2$ be renormalized kinetic solutions of (\ref{eq-1.1}) in the sense of Definition \ref{def-2.4} with initial data $\rho^1(\cdot,0)=\hat{\rho}^1,\rho^2(\cdot,0)=\hat{\rho}^2$. If $\hat{\rho}^1=\hat{\rho}^2\ \text{a.e. in }\mathbb{T}^d$, then
		\begin{equation}
			\mathbb{P}\Big(\sup_{t\in[0,T]}\|\rho^1(\cdot,t)-\rho^2(\cdot,t)\|_{L^1(\mathbb{T}^d)}=0\Big)=1.\notag
		\end{equation}
	\end{theorem}
	
	\begin{theorem}\label{thm-main2}
(Existence, cf. Theorem \ref{the-existence})
Assume that $V$ satisfies Assumption (A1). Let $\hat{\rho}\in\text{{\rm{Ent}}}\left(\mathbb{T}^{d}\right)$. Then there exists a stochastic basis $(\tilde{\Omega},\tilde{\mathcal{F}},\{\tilde{\mathcal{F}}(t)\}_{t\in[0,T]},\tilde{\mathbb{P}})$, {\color{black} a trace-class Brownian motion $\tilde{W}^{F}$ on  $L^2(\mathbb{T}^{d})$,} and a process $\tilde{\rho}\in L^1(\tilde{\Omega};L^1([0,T]\times \mathbb{T}^d))$, which is a renormalized kinetic solution of (\ref{eq-1.1}) in the sense of Definition \ref{def-2.4} with the equation (\ref{eq-2.6}) holds for almost every $t\in [0,T]$ and $\tilde{\rho}(\cdot,0)=\hat{\rho}$. If $V$ is further assumed to satisfy Assumption (A2), then there exists a unique probabilistically strong  renormalized kinetic solution of (\ref{eq-1.1}).
	\end{theorem}
Regarding the assumptions of the singular interaction kernels, we make some comments. Assumption (A1) is the Ladyzhenskaya-Prodi-Serrin (LPS) condition, which was firstly proposed by Prodi \cite{Pro59}, Serrin \cite{Ser62}, and Ladyzhenskaya \cite{Lad67} as a regularity criterion when studying the uniqueness of the 3D Navier-Stokes equations. \cite{CI07} provides a representation of the solution of the Navier-Stokes equations by using the stochastic flow. This leads to the study of stochastic differential equations (SDEs) with a singular drift that satisfies the LPS condition. Various literature in this direction includes \cite{KR05, Zha05,Zha11,XXZZ20, RZ21b, RZ23}.  Later, R\"ockner and Zhang \cite{RZ18} extended
the results of \cite{KR05} to distribution-dependent SDEs, which can be interpreted as large N limits
of interacting particle systems (\ref{IPS}). As fluctuating mean-field limits of the same model (\ref{IPS}), it is
natural to assume the SPDE (\ref{eq-1.1}) to have an LPS-type interaction kernel, paralleling the theory of
singular SDEs and distribution-dependent SDEs. Finally, we point out that the condition on $\nabla\cdot V$
is for technical reasons, which will be explained in detail at the end of the proof of uniqueness (see
Section \ref{sec-4})


Employing a similar approach akin to that used for (\ref{eq-1.1}), we can establish the well-posedness of the following conservative SPDE given by
\begin{equation}\label{isingkackawasaki}
	\partial_t\rho=\Delta\rho+\beta\nabla\cdot((1-\rho^2)\nabla J\ast\rho)-\gamma^{1/2}\nabla\cdot(\sqrt{1-\rho^2}\circ\xi_{\delta}).
\end{equation}
Here, $J$ stands for the Kac potential, {\color{black}$\beta$ is a constant depending on temperature}, {\color{black}the parameter $\gamma>0$ represents the noise intensity,}
 and the noise $\xi_{\delta}$ is white in time and correlated in space with a correlation length $\delta$ representing the simulation grid size.
The equation (\ref{isingkackawasaki}) proposed by \cite{GJE99} is related to the nonlinear fluctuations of Kawasaki dynamical Ising-Kac model (see subsection \ref{subsec-1.1} (2) for details). The well-posedness of (\ref{isingkackawasaki}) is formulated as follows.
 		\begin{theorem}\label{thm-main3}
 (Well-posedness, cf. Theorem \ref{the-7.11})
 		For any spatial dimension $d\ge1$, suppose that $\nabla J\in C^{\infty}(\mathbb{T}^d;\mathbb{R}^d)$. Let the initial data $\hat{\rho}\in\overline{{\rm{Ent}}}\left(\mathbb{T}^{d}\right)$ that is defined by (\ref{eqq3}), then there exists {\color{black}a unique probabilistically strong  renormalized kinetic solution} to (\ref{isingkackawasaki}) with initial data $\hat{\rho}$.
 	\end{theorem}

	\subsection{Applications}\label{subsec-1.1}
In this subsection, two important applications of Theorem \ref{thm-main2} and Theorem \ref{thm-main3} are presented.

{\bf(1) An application to fluctuations of  mean-field systems.} We mention
two landmark works by Wang, Zhao, Zhu \cite{WZZ23} and Chen, Ge \cite{CG22}. They studied the Gaussian fluctuations and large deviations for singular interacting mean-field systems
\begin{equation}\label{eq-1.4}
	\mathrm{d}X_i=\frac{1}{N}\sum_{j\neq i}V(X_i-X_j)\mathrm{d}t+\sqrt{2}\mathrm{d}B_t^i,\ i=1,\cdots,N,
\end{equation}
respectively. Precisely, \cite{WZZ23} proved the Gaussian fluctuation of the empirical measure $\pi^N$ around its mean-field limit
\begin{equation*}
	\partial_t\bar{\rho}=\Delta\bar{\rho}-\nabla\cdot(\bar{\rho}V\ast\bar{\rho}),
\end{equation*}
which reads as
\begin{equation*}
	\sqrt{N}(\pi^N-\bar{\rho})\to\bar{\rho}^1\quad \text{as\ $N\to\infty$}.
\end{equation*}
Here, $\bar{\rho}^1$ satisfies the equation
\begin{equation}\label{eq-1.5}
	\partial_t\bar{\rho}^1=\Delta\bar{\rho}^1-\nabla\cdot(\bar{\rho}V\ast\bar{\rho}^1)-\nabla\cdot(\bar{\rho}^1V\ast\bar{\rho})-\sqrt{2}\nabla\cdot(\sqrt{\bar{\rho}}\xi),
\end{equation}
with $\xi$ being a space-time white noise.
Let $\xi_{K(N)}$ be the ultra-violet noise, which converges to $\xi$ as $K(N)\rightarrow\infty$, $N\rightarrow\infty$. An informal computation shows that the Dean-Kawasaki equation with $\xi_{K(N)}$
\begin{equation}\label{rr-1}
	\partial_t\rho^N+\nabla\cdot(\rho^NV\ast\rho^N)=\Delta\rho^N-\sqrt{\frac{2}{N}}\nabla\cdot(\sqrt{\rho^N}\circ\xi_{K(N)})
\end{equation}
fulfills the same Gaussian fluctuation $\sqrt{N}(\rho^{N}-\bar{\rho})\to\bar{\rho}^1$, where $\bar{\rho}^1$ fulfills (\ref{eq-1.5}).

Regarding to \cite{CG22}, the authors established the large deviations for the empirical measure of the two-dimensional interacting particle systems \eqref{eq-1.4} with the rate function informally given by
\begin{align}\notag
	I^0(\rho)&=\sup_{\psi\in C^{\infty}(\mathbb{T}^2\times[0,T])}\Big(
	\int_{\mathbb{T}^2}\psi(x,T)\rho(x,T)\mathrm{d}x-\int_{\mathbb{T}^2}\psi(x,0)\rho(x,0)\mathrm{d}x-\int^T_0\int_{\mathbb{T}^2}\rho\partial_t\psi \mathrm{d}x\mathrm{d}t\\
	\label{eq-1.7}
	& \quad\quad\quad\quad\quad\quad -\int^T_0\int_{\mathbb{T}^2}\rho \Delta\psi \mathrm{d}x\mathrm{d}t-\int^T_0\int_{\mathbb{T}^2}(\rho V\ast \rho)\nabla \psi \mathrm{d}x\mathrm{d}t-\frac{1}{2}\int^T_0\int_{\mathbb{T}^2}\rho|\nabla \psi|^2\mathrm{d}x\mathrm{d}t
	\Big).
\end{align}
As discussed in \cite[Theorem 39]{FG23}, Fehrman and Gess proved that the rate function for large deviations of the Dean-Kawasaki equation is governed by \eqref{eq-1.7} as well when $V=0$. However, upon closely following Fehrman and Gess's argument line by line, the result remains unchanged in the presence of an interaction kernel (at least when the kernel is ``nice''  ). Consequently, the fluctuations of (\ref{eq-1.4}) can be predicted by (\ref{rr-1}).

In fact, aside from the result that the Dean-Kawasaki equation (\ref{rr-1}) preserves the fluctuations feature of the mean-field system (\ref{eq-1.4}), the regularity of the former is $L^2_tW^{1,1}_x$ (see the entropy estimate in Proposition \ref{pro-5.4}) that is stronger than $L^2_tH^{-\frac{d}{2}-}_x$ for (\ref{eq-1.4}).
Thus, it is a good alternative to study the Dean-Kawasaki equation (\ref{rr-1}) to capture the fluctuations of (\ref{eq-1.4}). Applying our result Theorem \ref{thm-main2} to  (\ref{rr-1}), we obtain its well-posedness which plays a fundamental role in studying fluctuations.

Finally, we mention that some quantitative error analysis has been built between discrete and continuous objects. For example,
\cite{DKP22} established a weak error estimate between the regularized Dean-Kawasaki equation and Brownian particles via a duality argument.  \cite{CF23,CFIR23} provided error estimate results for the discretized Dean-Kawasaki equation, demonstrating that structure-preserving discretizations can closely approximate the density fluctuations of $N$ non-interacting diffusing particles to an arbitrary order in $N^{-1}$ within appropriate weak metrics. Thus, it is feasible to adopte a SPDE framework to mirror the fluctuations of the $d$-dimensional $N$-particle SDE systems (\ref{eq-1.4}), especially in the context of mean-field systems with $L^p$-type interactions (refer to \cite{HRZ22,WZZ22,HHMT20}). This parallel allows the study of SPDE \eqref{rr-1} to serve as a theoretical foundation for understanding the dynamics of the corresponding particle system.

{\bf(2) An application to nonlinear fluctuations of dynamical Ising-Kac model.}
The Ising model, originally proposed by Ising in  \cite{EI25}, serves as a fundamental model in statistical mechanics for exploring ferromagnetism. This model involves spins arranged on a lattice, where each spin interacts with its nearest neighbors. A variant of the Ising model called the Ising-Kac model was introduced to recover the van der Waals theory of phase transition,
in which each spin interacts with all other spin variables in a
large ball around its base point (see \cite{HKU}). Two main dynamics for the Ising-Kac model are Glauber dynamics and Kawasaki dynamics. The elementary events in the former are spin-flips (i.e. changes of sign of a single spin), while, for the latter, two spins exchange their positions (more details can be found in \cite{Gla63, Spo12}). For both dynamics, it is an active field of research to derive the macroscopic behavior from microscopic model by hydrodynamical scaling limit.
In \cite{GJE99}, it was conjectured that both the Glauber and the Kawasaki dynamics have anomalous fluctuating behaviors, i.e. nonlinear fluctuations, in a neighborhood of the critical temperature. Moreover, their nonlinear fluctuations are described by $\Phi^4_d$ and Cahn-Hilliard equation, respectively.
Recently, the conjecture on the Glauber dynamics in
1, 2 and 3 dimensions is completely settled; see \cite{BPRS94,FR95} for $d=1$, \cite{MW17} for $d=2$, and \cite{GMW23} for $d=3$.
However, the conjecture on the nonlinear fluctuations of Kawasaki dynamics remains widely open, and only some progress in 1 dimension is made by \cite{IM18}.
 The present paper paves a step on this conjecture in a certain sense. In the following, we make some explanations.


In light of the theory of fluctuating hydrodynamics, a continuous analogue (a conservative SPDE) is
conjectured to have the same scaling limits as microscopic particle system. A formal analysis from \cite{GJE99} shows that the conservative SPDE related to  the Ising-Kac model in one spatial dimension is the equation (\ref{isingkackawasaki}).
Precisely,
 for any $\gamma>0$ and $a\in\mathbb{R}$, let $\beta=1+a\gamma^{2/3}$ and define
 the rescaled density field
	\begin{equation*}
		\rho_{\gamma}(x,t):=\gamma^{-1/3}\rho(\gamma^{-1/3}x,\gamma^{-4/3}t)
	\end{equation*}
with $\rho$ satisfying (\ref{isingkackawasaki}).  $\rho_{\gamma}$ is conjectured to converge to the conservative stochastic Cahn-Hilliard equation
	\begin{equation}\label{cahnhilliard}
		\partial_tu=-\partial_{xx}^2(\partial_{xx}^2u-u^3+au)-\partial_x\xi, 	
	\end{equation}
	as $\gamma\rightarrow0$. It implies that (\ref{isingkackawasaki}) has the same nonlinear fluctuations phenomenon as Ising-Kac model when the temperature is near criticality. Therefore, (\ref{isingkackawasaki}) can be regarded as a phenomenological model simulating the Kawasaki dynamics for the Ising-Kac model.

However, a rigorous proof of the conjecture still remains open, and we emphasize that any rigorous
study of (\ref{isingkackawasaki}) is challenging due to the lack of well-posedness.   Thanks to the structural resemblance with the Dean-Kawasaki equation, in this paper, we  apply our main results to establish strong
well-posedness for the Ising-Kac-Kawasaki equation (\ref{isingkackawasaki}).

	\subsection{Key argument and technical comment}	
		
	As stated in \cite{FG21} and \cite{FG23}, due to the singularity at zero of the It\^{o} correction term $\frac{1}{4\rho}\nabla\rho$, it is not clear how to define the concept of classical weak solutions to (\ref{eq-1.3}). To restrict the value of the solution away from zero, \cite{FG21} employed the kinetic approach and introduced the concept of renormalized kinetic solutions. This concept was firstly proposed by Lions, Perthame, and Tadmor in \cite{LPT94} when studying general multi-dimensional scalar conservation laws. The key insight is that the kinetic function has three variables $(t,x,\xi)$, which stand for time, spatial and velocity variable, respectively. Then, the test function for the renormalized kinetic solution can have compact support with respect to the velocity variable. As a result, it keeps the kinetic solution away from its singularities.
 This idea has been successfully applied to conservative stochastic PDEs, we refer readers to the works of Gess and Souganidis  \cite{GS17}, Fehrman and Gess \cite{FG19}, and Dareiotis and Gess \cite{DG20} for further details. In the present paper, we also adopt the concept of renormalized kinetic solutions.

 	Moreover, we provide further commentary on technical details involved in studying the well-posedness of the fluctuating Ising-Kac-Kawasaki equation (\ref{isingkackawasaki}).
 As discussed above, the entropy estimate plays a central role. Due to the structural similarity between (\ref{isingkackawasaki}) and (\ref{eq-1.1}), we can anticipate it holds. Indeed, by choosing $\Psi(\zeta)=\int^{\zeta}\log\left(\frac{1+\zeta'}{1-\zeta'}\right)d\zeta'$, the corresponding entropy dissipation estimate for (\ref{isingkackawasaki}) holds as well, see Proposition \ref{pro-7.4}.
 Then, following the approach used in (\ref{eq-1.1}), an analogue of Theorems \ref{the-uniq} and \ref{the-6.4} can be established for (\ref{isingkackawasaki}). Finally, the results of Theorem \ref{thm-main1} and \ref{thm-main2} can be partially extended to (\ref{isingkackawasaki}).
 In addition, we emphasize some differences between (\ref{isingkackawasaki}) and (\ref{eq-1.1}). The first is that the kinetic solution of (\ref{eq-1.1}) is nonnegative, while, the solution of(\ref{isingkackawasaki}) takes values in the interval ${\black (-1,1)}$, it causes the preservation of the $L^1(\mathbb{T}^d)$-norm to be invalid for the latter. The second is that the derivative of the diffusion coefficient $\sqrt{\zeta}$ for (\ref{eq-1.1}) only has one singularity at $\zeta=0$
 , whereas the derivative of the diffusion coefficient $\sqrt{1-\zeta^2}$ for (\ref{isingkackawasaki}) has two singularities at $\zeta=+1$ and $\zeta=-1$. It leads to some technical differences in proving the tightness of their approximating equations. For (\ref{eq-1.1}), a new $L^{1}\left([0,T];L^{1}\left(\mathbb{T}^{d}\right)\right)$-equivalent topology is constructed with the aid of a sequence of truncation functions that keep the solution away from zero.
 However, it seems that this method is difficult to apply to the equation (\ref{eqq1}), even if the truncation functions are changed to force the solution away from $-1$ and $+1$.
 We solve this problem by defining a transformation such that the singular term of the new equation has only one singularity. Thus, we can apply the method in Section \ref{subsec-5.3} to obtain the tightness of the new equation. By virtue of this result and the property of the transformation, we finally conclude the tightness of the original equation. For the details, see Section \ref{sec-r1}.

\subsection{Comparison with \cite{FG21}}
In this paragraph, we outline the primary distinctions between our paper and \cite{FG21} from the following three aspects. (i) The equation. Obviously, our equation has additional nonlocal interaction terms compared with \cite{FG21}. Coming up with suitable conditions for the interaction terms to ensure the existence of a solution is the central challenge.  Then, all the estimations of the kernel term presented in Section \ref{sec-3} are definitely new. The key observation is that the nonlocal interaction term is associated with the convolution type distributional dependence SDEs, hence it is natural to impose an LPS-type condition on the kernel.
(ii)The technique. Compared with \cite{FG21}, all of the technical adjustments are to deal with the difficulties arising from nonlocal interaction terms. Due to the irregularity of the nonlocal interaction kernel, the $L^m$-theory ($m\geq 2$) of (\ref{eq-1.3}) established by \cite{FG21} is no longer applicable to (\ref{eq-1.1}). As an alternative, we make an entropy dissipation estimate (see Proposition \ref{pro-5.4}) for (\ref{eq-1.1}) assuming that the kernel $V$ satisfies Assumption (A1).
Such an estimate not only provides the regularity of the square root of the solution, known as Fisher information, but also suggests a potential link between LPS-type conditions in fluid dynamics and the Boltzmann entropy in the fluctuating hydrodynamics. Moreover, it also plays a pivotal role in proving both the uniqueness and the existence of renormalized kinetic solutions to (\ref{eq-1.1}). For the uniqueness, the kernel terms cannot be expected to converge to zero or even be controlled since cut-off functions cannot bound the nonlocal term, which is different from \cite{FG21} where all terms vanish. To solve this obstacle, we show that the interaction
kernel terms in (\ref{eq-1.1}) can be controlled by the square root of the solution under Assumptions (A1) and (A2). Consequently, with the aid of entropy estimate, the uniqueness of (\ref{eq-1.1}) is achieved by applying a stochastic Gronwall's lemma. A comprehensive proof can be found in Theorem \ref{the-uniq}. Regarding the existence of renormalized kinetic solutions to (\ref{eq-1.1}), we also introduce a sequence of approximating equations with regular coefficients and smooth kernel similarly to \cite{FG21}. To find a limiting kinetic measure, the method used by \cite{FG21}  relies heavily on the uniformly $L^m$ estimate of  the approximating equations on the smoothing parameters to ensure the uniformly boundedness of the corresponding kinetic measures in the space of bounded Borel measures over $[0, T] \times \mathbb{T}^d \times \mathbb{R}$ (with norm defined by the total variation of measures). However, this method is not applicable to us due to the lack of uniformly $L^m$ estimate. 
Instead, we make use of the entropy estimate to show that the kinetic measures are bounded in the space of nonnegative bounded Borel measures over $[0, T] \times \mathbb{T}^d \times [0, M]$ for any $M \geq 1$. Subsequently, the existence of a limiting kinetic measure can be proved by a diagonal argument. For more details, see Theorem \ref{the-existence} below.
(iii)The applications. 
The well-posedness of (\ref{eq-1.1}) allows us to study the fluctuations of mean-field systems with singular interactions, especially the Kawasaki dynamics for the Ising-Kac model. The specific examples have been presented in subsection 1.1.

	\subsection{Comments on the literature}
	The existence of solutions to corrected Dean-Kawasaki equations with smooth interacting kernel has been proved by von Renesse and Sturm \cite{vRS09} by Dirichlet forms techniques, where the nonlocal interacting term is replaced by a nonlinear operator. Later, by correcting the drift term of the Dean-Kawasaki equation, the authors of \cite{AvR10} and \cite{KvR19} constructed a solution. Cornalba, Shardlow and Zimmer \cite{CSZ19,CSZ20} derived a suitably regularized Dean-Kawasaki model of wave equation type in one dimension, which corresponds to second-order Langevin dynamics. In the case of local interaction, Fehrman and Gess \cite{FG21} obtained the well-posedness of functional-valued solutions of the Dean-Kawasaki equation with correlated noise. Building on this framework in  \cite{FG21}, the authors also addressed small noise large deviations in \cite{FG23}. Furthermore, Clini and Fehrman  \cite{CF23a} expanded this research by developing a central limit theorem for the nonlinear Dean-Kawasaki equation with correlated noise.

	In the framework of kinetic solution, the study of well-posedness of stochastic conservative law has attracted significiant interests. Debussche and Vovelle \cite{DV10} studied the Cauchy problem in any dimension and obtained the existence and uniqueness of the kinetic solutions. Later, Gess and Souganidis \cite{GS17}, Fehrman and Gess \cite{FG19}, and Dareiotis and Gess \cite{DG20} extended the notion of kinetic solution to parabolic-hyperbolic stochastic PDE with conservative noise. Recently, the nonlocal conservative stochastic PDE was considered by Fehrman, Gess and Gvalani \cite{FGR} and  the mean-field stochastic PDE has been studied by Gess, Gvalani and Konarovskyi \cite{GGK}. For literatures on stochastic nonlinear diffusion equations, we refer to \cite{BBDPR06,BDPR08a,BDPR08b,BDPR09,BDPR16}.

	The literature contributed to Gaussian fluctuations of discrete interacting particle systems is notably comprehensive. Seminal studies such as \cite{FPV87} investigated the zero-range process, while \cite{Rav92} focused on the symmetric simple exclusion process, and \cite{JM18} examined the weakly asymmetric exclusion process.
The exploration of the mean field limit for singular interaction kernels has also yielded significant results. Key contributions in this area include the study of the vortex model by  \cite{Osa86, FHM14}, and the analysis of more general singular kernels in works such as \cite{JW18, BJW20, Ser20, RS}.

	\subsection{Structure of the paper}
	This paper is organized as follows. In Section \ref{sec-2}, basic notations, assumptions on the noise and kernel, and the definition of renormalized kinetic solutions of \eqref{eq-1.1} are given. Section \ref{sec-3} is dedicated to presenting various estimates associated with the nonlocal kernel. 
	The uniqueness of the renormalized kinetic solutions of equation \eqref{eq-1.1} is rigorously proved in Section \ref{sec-4}. Section \ref{sec-5} shifts focus to a sequence of approximating equations for \eqref{eq-1.1}, where we derive an entropy estimate and establish $L^m(\mathbb{T}^d)-$norm estimates in Section \ref{subsec-5.1}. The existence of renormalized kinetic solutions to these approximating equations is obtained in Section \ref{subsec-5.2}. Furthermore, Section \ref{subsec-5.3} addresses the $L^1([0,T];L^1(\mathbb{T}^d))$-tightness of solutions for this sequence of approximating equations. In Section \ref{sec-6}, we conclusively prove the existence of renormalized kinetic solutions of equation \eqref{eq-1.1}. Eventually, we adjust the proof of \eqref{eq-1.1} to show the well-posedness of the fluctuating Ising-Kac-Kawasaki equation (\ref{isingkackawasaki}) in Section \ref{sec-7}.

\section{Preliminaries}\label{sec-2}
\subsection{Notations}
Throughout the paper, {\black $\mathbb{T}^d=[-1/2,1/2]^d$} denotes a $d-$dimensional torus with volume 1.
Let $\nabla$ represent the derivative operator and $\nabla\cdot$ be the divergence operator with respect to the space variable $x\in\mathbb{T}^d$. In particular,
for any $V:[0,T]\times\mathbb{T}^d\to\mathbb{R}^d$, $\nabla\cdot V$ stands for the spatial divergence of $V$.
Let $\|\cdot\|_{L^p(\mathbb{T}^d)}$ denote the norm of Lebesgue space $L^p(\mathbb{T}^d)$ (or $L^p(\mathbb{T}^d;\mathbb{R}^d)$) for integer $p\in [1,\infty]$. The inner product in $L^2(\mathbb{T}^d)$ will be denoted by
$(\cdot,\cdot)$.
Let $C^{\infty}(\mathbb{T}^d\times (0,\infty))$ denote the space of infinitely differentiable functions on $\mathbb{T}^d\times (0,\infty)$. $C^{\infty}_c(\mathbb{T}^d\times (0,\infty))$ contains all {\color{black}infinitely differentiable functions} with compact supports on $\mathbb{T}^d\times (0,\infty)$.
For a non-negative integer $k$ and $p\in [1,\infty]$, denote by $W^{k,p}(\mathbb{T}^d)$ the usual Sobolev space on $\mathbb{T}^d$. Let
$H^a(\mathbb{T}^d)=W^{a,2}(\mathbb{T}^d)$, and $H^{-a}(\mathbb{T}^d)$ stands for the topological dual of $H^a(\mathbb{T}^d)$.
$C^{k}_{loc}(\mathbb{R})$ denotes {\color{black}$k$ times continuously differentiable functions} on any compact sets in $\mathbb{R}$. Let the bracket $\langle\cdot,\cdot\rangle$ stand for {\color{black}the duality pairing between}  $C^{\infty}(\mathbb{T}^d)$ and the space of distributions over $\mathbb{T}^d$.

Let $X$ be a real Banach space with norm $\|\cdot\|_X$. The space $L^p([0,T];X)$ denotes the standard Lebesgue space, and $W^{1,p}([0,T];X)$ denotes the standard Sobolev space.
In the context, without confusion, for $1\leq p\leq \infty$ we denote  $L^p(\mathbb{T}^d;\mathbb{R})$ or $L^p(\mathbb{T}^d;\mathbb{R}^d)$ by $L^p(\mathbb{T}^d)$,
$L^p([0,T]\times\mathbb{T}^d)=:L^p([0,T];L^p(\mathbb{T}^d;\mathbb{R}))$, $L^p([0,T]\times\mathbb{T}^d;\mathbb{R}^d)=:L^p([0,T];L^p(\mathbb{T}^d;\mathbb{R}^d))$
and $C([0,T])=:C([0,T];\mathbb{R})$.
Also, we will encounter
 integrals on a space $Z$ ($Z$ might be $[0,T]\times \mathbb{T}^d$, $[0,T]\times \mathbb{T}^d\times \mathbb{R}$, $[0,T]\times (\mathbb{T}^d)^2\times \mathbb{R}^2$ and so on ). For simplicity, we abbreviate all integrals $\int_Z f\mathrm{d}z$ to $\int_Z f$. In addition, we always use $\nabla f$ to denote the weak derivative of $f$ with respect to the space variable.

In the sequel, the notation $a\lesssim b$ for $a,b\in \mathbb{R}$  means that $a\leq \mathcal{D}b$ for some constant $\mathcal{D}> 0$ independent of any parameters.
We employ the letter $C$  to denote any constant that can be explicitly computed in terms of known quantities. The exact value denoted by $C$ may change from line to line.



\subsection{Assumptions}

Let $(\Omega,\mathcal{F},\mathbb{P},\{\mathcal{F}_t\}_{t\in
	[0,T]},(\{B^k(t)\}_{t\in[0,T]})_{k\in\mathbb{N}})$ be a stochastic basis. Without loss of generality, here the filtration $\{\mathcal{F}_t\}_{t\in [0,T]}$ is assumed to be complete and $\{B^k(t)\}_{t\in[0,T]},k\in\mathbb{N}$, are independent ($d$-dimensional)  $\{\mathcal{F}_t\}_{t\in [0,T]}-$Wiener processes. We use $\mathbb{E}$ to denote the expectation with respect to $\mathbb{P}$. Let the sequence $\{f_k\}_{k\geq 1}$ consist of functions $f_k \in C^1(\mathbb{T}^d;\mathbb{R})$ such that
  \begin{equation*}
	F_1=\sum_{k=1}^{\infty}|f_k|^2,\quad
	F_2=\frac{1}{2}\sum_{k=1}^{\infty}\nabla f_k^2,\quad
	F_3=\sum_{k=1}^{\infty}|\nabla f_k|^2
\end{equation*}
are continuous functions on $\mathbb{T}^d$, and satisfy $\nabla\cdot F_2=\frac{1}{2}\Delta F_1=0$ on $\mathbb{T}^d$.

Define
\begin{align}\label{kk-43}
W^F=\sum_{k=1}^{\infty}f_k(x)B_t^k.
\end{align}
{\black A representative example satisfying the above assumptions is the so-called ultra-violet divergence noise, which admits the spectral expansion
$$
\xi^a(t,x)=\sum_{k\in\mathbb{Z}^d} a_k\left(\sin(2\pi k\cdot x)\,B_t^k+\cos(2\pi k\cdot x)\,W_t^k\right),
$$
where $\{B^k,W^k\}_{k\in\mathbb{Z}^d}$ are independent Brownian motions and $\{a_k\}_{k\in\mathbb{Z}^d}\subset\mathbb{R}$ are coefficients satisfying $\sum_{k\in\mathbb{Z}^d}|k|^2a_k^2<\infty$. In this case, one obtains
\begin{equation*}
    F_1=\sum_{k\in\mathbb{Z}^d} a_k^2,\quad F_2=0,\quad F_3=\sum_{k\in\mathbb{Z}^d}|k|^2a_k^2.
\end{equation*}
Such ultra-violet divergence noise plays a fundamental role in the analysis of fluctuations under singular scaling limits; see \cite[Remark~2.3]{FG21} and \cite[Section~3]{DFG}.
}

\subsection{Renormalized kinetic solution to the Dean-Kawasaki equation}
In this paper, we consider the Dean-Kawasaki equation
\begin{equation}\label{mainspde}
\mathrm{d}\rho=\Delta\rho \mathrm{d}t-\nabla\cdot(\rho V\ast\rho)\mathrm{d}t-\nabla\cdot(\sqrt{\rho}\circ \mathrm{d}W^F),
\end{equation}
where $W^F$ is defined by (\ref{kk-43}).
Then the Stratonovich equation \eqref{mainspde} is formally equivalent to the  It\^o equation
\begin{equation}\label{eq-2.2}
\mathrm{d}\rho=\Delta\rho \mathrm{d}t-\nabla\cdot(\rho V\ast\rho)\mathrm{d}t-\nabla\cdot(\sqrt\rho \mathrm{d}W^F)+\frac{1}{8}\nabla\cdot(F_1\rho^{-1}\nabla\rho+2F_2)\mathrm{d}t.
\end{equation}

For a given nonnegative solution $\rho$ of \eqref{eq-2.2}, we define the kinetic function $\chi: \mathbb{T}^{d}\times\mathbb{R}\times[0,T]\to\{0,1\}$ of $\rho$ as
\begin{align*}
\chi(x,\xi,t):=\mathbf{1}_{\{0<\xi<\rho(x,t)\}}.
\end{align*}
Formally, we have the following distributional equalities
\begin{align}
\nabla\chi=\delta_{0}(\xi-\rho)\nabla\rho, \quad \partial_{\xi}\chi=\delta_{0}(\xi)-\delta_{0}(\xi-\rho)\quad {\rm{and}}\ \rho=\int_{\mathbb{R}}\chi \mathrm{d}\xi.
\end{align}
 Then, the kinetic function $\chi$ of $\rho$ satisfies the equation
\begin{align}\label{eq-2.3}
\partial_t\chi=&\ \nabla\cdot\left(\delta_{0}(\xi-\rho)\nabla\rho\right)+\frac{1}{8}\nabla\cdot\left(\delta_{0}(\xi-\rho)\left(F_1\xi^{-1}\nabla\rho+2F_2\right)\right)-\frac{1}{4}\partial_{\xi}\left(\delta_{0}(\xi-\rho)\left(\nabla\rho\cdot F_2+2\xi F_3\right)\right)\notag\\
&+\partial_{\xi}q-\delta_{0}(\xi-\rho)\nabla\cdot(\rho V\ast\rho)-\delta_{0}(\xi-\rho)\nabla\cdot(\sqrt\rho \dot{W}^F),
\end{align}
where\  $q=4\delta_0(\xi-\rho)\xi|\nabla\sqrt{\rho}|^2$ is the parabolic defect measure.

To define a renormalized kinetic solution to \eqref{mainspde}, we need a concept of kinetic measure.
\begin{definition}\label{def-2.2}
	Let $(\Omega,\mathcal{F},\mathbb{P})$ be a probability space with a filtration $(\mathcal{F}_{t})_{t\in[0,\infty)}.$ A kinetic measure is a map $q$ from $\Omega$ to the space of nonnegative, locally finite measures on $\mathbb{T}^{d}\times(0,\infty)\times[0,T]$ that satisfies the property that the process
	\begin{align*}
	(\omega,t)\in \Omega\times[0,T]\to
	\int^t_0{\color{black}\int_0^{\infty}}\int_{\mathbb{T}^d}\psi(x,\xi)\mathrm{d}q(x,\xi,r)
	\end{align*}
	is $\mathcal{F}_t-$predictable, for every $\psi\in C^{\infty}_c(\mathbb{T}^d\times (0,\infty))$.
\end{definition}


We will prove the well-posedness of \eqref{mainspde} for initial data with finite entropy.
Define
\begin{align}\label{kk-44}
\text{{\rm{Ent}}}(\mathbb{T}^{d})=\Big\{\rho\in L^{1}(\mathbb{T}^{d}): \rho\geq 0\ {\rm{and}}\ \int_{\mathbb{T}^{d}}\rho\log(\rho)<\infty\Big\}.
\end{align}

\begin{definition}(Renormalized kinetic solution)\label{def-2.4}
	Let $\hat{\rho}\in  \text{{\rm{Ent}}}\left(\mathbb{T}^{d}\right)$ and let $V$ satisfy Assumption (A1). A renormalized kinetic solution of (\ref{mainspde}) with initial datum $\rho(\cdot,0)=\hat{\rho}$ is a nonnegative, almost surely continuous $L^1(\mathbb{T}^d)$-valued $\mathcal{F}_t$-predictable function $\rho\in L^1\left(\Omega\times[0,T];L^1(\mathbb{T}^d)\right)$ that satisfies the following properties.
	\begin{enumerate}
		\item Conservation of mass: almost surely for every $t\in[0,T]$,
		\begin{equation}\label{eq-2.4}
		\|\rho(\cdot,t)\|_{L^{1}\left(\mathbb{T}^{d}\right)}=\left\|\hat{\rho}\right\|_{L^{1}\left(\mathbb{T}^{d}\right)}.
		\end{equation}
		\item Regularity of $\sqrt{\rho}$: there exists a constant $c\in(0,\infty)$ depending on $T,\|\hat{\rho}\|_{L^1(\mathbb{T}^d)}, V$ and $d$ such that
		\begin{equation}\label{eq-2.5}
		\mathbb{E}\int_{0}^{T}\int_{\mathbb{T}^{d}}\left|\nabla \sqrt\rho\right|^{2}\le c(T, d,\|\hat{\rho}\|_{L^1(\mathbb{T}^d)},\|V\|_{L^{p*}([0,T];L^p(\mathbb{T}^d;\mathbb{R}^d))}).
		\end{equation}
		Furthermore, there exists a nonnegative kinetic measure $q$ satisfying the following properties.
		\item Regularity: almost surely
		\begin{align}\label{eq-2.66}
		4\delta_{0}(\xi-\rho)\xi|\nabla \sqrt{\rho}|^{2}\le q\quad {\rm{on}}\ \mathbb{T}^d\times(0,\infty)\times[0,T].
		\end{align}
		\item Vanishing at infinity: we have
		\begin{align}\label{eq-2.7}
		\liminf_{M\to \infty}\mathbb{E}\Big[q(\mathbb{T}^d\times [0,T]\times [M,M+1])\Big]=0.
		\end{align}
		\item The equation: for every $\varphi\in\mathrm{C}_{c}^{\infty}\left(\mathbb{T}^{d}\times(0,\infty)\right)$, almost surely for {\blue every} $t\in[0,T]$,
		\begin{align}\label{eq-2.6}
		&{\color{black}\int_0^{\infty}}\int_{\mathbb{T}^d}\chi(x,\xi,t)\varphi(x,\xi)={\color{black}\int_0^{\infty}}\int_{\mathbb{T}^d}\bar{\chi}(\hat{\rho})\varphi(x,\xi)-\int_0^t\int_{\mathbb{T}^d}\nabla\rho\cdot(\nabla\varphi)(x,\rho)\notag\\
		&-\frac{1}{8}\int_0^t\int_{\mathbb{T}^d}F_1(x)\rho^{-1}\nabla\rho\cdot(\nabla\varphi)(x,\rho)-\frac{1}{4}\int_0^t\int_{\mathbb{T}^d}F_2(x)\cdot(\nabla\varphi)(x,\rho)\notag\\
		&-\int_0^t{\color{black}\int_0^{\infty}}\int_{\mathbb{T}^d}\partial_{\xi}\varphi(x,\xi)dq+\frac{1}{4}\int_0^t\int_{\mathbb{T}^d}\left(\nabla\rho\cdot F_2(x)+2F_3(x)\rho\right)(\partial_{\xi}\varphi)(x,\rho)\notag\\
		&-\int_0^t\int_{\mathbb{T}^d}\varphi(x,\rho)\nabla\cdot(\rho V(r)\ast\rho)-\int_0^t\int_{\mathbb{T}^d}\varphi(x,\rho)\nabla\cdot(\sqrt\rho \mathrm{d}W^F(r)),
		\end{align}
		where $\bar{\chi}(\hat{\rho})(x,\xi):={\color{black}\mathbf{1}_{\{0<\xi<\hat{\rho}(x)\}}}$.
	\end{enumerate}	
\end{definition}
\begin{remark}
	The estimate (\ref{eq-2.5}) implies that for every $K\in \mathbb{N}$,
	\begin{align}\label{kk-48}
	[(\rho\wedge K)\vee (1/K)]\in L^2(\Omega;L^2([0,T];H^1(\mathbb{T}^d))).
	\end{align}
As a result, the term $-\frac{1}{8}\int_0^t\int_{\mathbb{T}^d}F_1(x)\rho^{-1}\nabla\rho\cdot(\nabla\varphi)(x,\rho)$ on the righthand side of \eqref{eq-2.6} is well-defined. In addition, the integrability of the kernel term $\int_0^t\int_{\mathbb{T}^d}\varphi(x,\rho)\nabla\cdot(\rho V(r)\ast\rho)$ is guaranteed by Lemma \ref{lem-3.4} below.
\end{remark}

\begin{remark}\label{rmk-2.5}
It is important to note that under Assumption (A1) alone, we can obtain a probabilistically weak renormalized kinetic solution to (\ref{mainspde}) in the sense of Definition \ref{def-2.4} with the equation (\ref{eq-2.6}) holds for almost every $t\in [0,T]$. Under further Assumption (A2), the solution can be shown to be almost surely continuous on $L^1(\mathbb{T}^d)$ and the equation (\ref{eq-2.6}) holds for every $t\in [0,T]$. Thus, it is a renormalized kinetic solution to (\ref{mainspde}) exactly in the sense of Definition \ref{def-2.4}. Moreover, the solution is proven to be strong in the sense of probability.
%
\end{remark}

Beyond the established vanishing property at infinity  (\ref{eq-2.7}), the kinetic measure exhibits a decay at zero as well. In fact, this property can be derived by similar method as in the proof of \cite[Proposition 4.5]{FG21}, thus we omit the proof.
\begin{lemma}
	Suppose that Assumption (A1) is in force. Let $\hat{\rho}\in\text{{\rm{Ent}}}\left(\mathbb{T}^{d}\right)$. Let $\rho$ be a renormalized kinetic solution of \eqref{eq-1.1} in the sense of Definition \ref{def-2.4} with initial data $\rho(\cdot, 0)=\hat{\rho}$. Then it follows that, almost surely,
	\begin{align}\label{eq-2.8}
		\lim_{\beta\rightarrow0}\left[\beta^{-1}q\left(\mathbb{T}^{d}\times[0,T]\times[\beta/2,\beta]\right)\right]=0.
	\end{align}
\end{lemma}

\section{Estimates for the kernel term}\label{sec-3}
In this section, we make  a series of estimates for the nonlocal kernel term under Assumptions (A1) and (A2) which are definitely new compared to \cite{FG21}. These estimates will  play a pivotal role in the substantiation of our main results.

Note that under Assumptions (A1) and (A2), it follows that $\frac{2p}{p-d}=\frac{2}{1-\frac{d}{p}}\le p^{*}$ and  $\frac{2q}{2q-d}=\frac{1}{1-\frac{d}{2q}}\le q^{*}$.
A simple application of H\"{o}lder's inequality implies the following results.
\begin{lemma}
	Let $f\in L^{\infty}\left([0,T];L^1(\mathbb{T}^d)\right)$ be nonnegative with $\sqrt{f}\in L^2\left([0,T];H^1(\mathbb{T}^d)\right)$.
	\begin{enumerate}
		\item If $V$ satisfies Assumption (A1), then there exists a constant $C$ depending on $T$ such that
		\begin{align}\label{eq-3.1}
		\int_{0}^{T}\|\nabla\sqrt{f(t)}\|^{\frac{d}{p}+1}_{L^2(\mathbb{T}^d)}\|V(t)\|_{L^p(\mathbb{T}^d)}\le C(T)\|\nabla\sqrt{f}\|^{\frac{d+p}{p}}_{L^2([0,T];L^2(\mathbb{T}^d))}\|V\|_{L^{p*}([0,T];L^p(\mathbb{T}^d;\mathbb{R}^d))}.
		\end{align}
		\item If $V$ satisfies Assumption (A2), then there exists a constant $C$ depending on $T$ such that
		\begin{align}\label{eq-3.2}
		\int_{0}^{T}\|\nabla\sqrt{f(t)}\|^{\frac{d}{q}}_{L^2(\mathbb{T}^d)}\|\nabla\cdot V(t)\|_{L^q(\mathbb{T}^d)} \le C(T) \|\nabla\sqrt{f}\|^{\frac{d}{q}}_{L^2([0,T];L^2(\mathbb{T}^d))}\|\nabla\cdot V\|_{L^{q*}([0,T];L^q(\mathbb{T}^d))}.
		\end{align}
	\end{enumerate}
\end{lemma}

We emphasize that the estimations presented in (\ref{eq-3.1}) and (\ref{eq-3.2}) will be employed when applying a stochastic Gronwall's lemma to derive the pathwise uniqueness (see Theorem \ref{the-uniq}). In addition, the inequalities (\ref{eq-3.1}) and (\ref{eq-3.2}) imply the index relation imposed on $V$ and $\nabla\cdot V$, respectively.

{\black
In the following, we introduce the standard convolutional kernel on $\mathbb{R}^d$ and $\mathbb{T}^d$. For any $d\ge1$, we take
\begin{equation*}
    \varphi(x) =
\begin{cases}
C \exp\!\left(-\dfrac{1}{1/2 - |x|^2}\right), & |x| < 1/\sqrt{2}, \\[8pt]
0, & |x| \geq 1/\sqrt{2},
\end{cases}
\quad x \in \mathbb{R}^d,
\end{equation*}
where the constant $C > 0$ is chosen to satisfy
\begin{equation*}
    \int_{\mathbb{R}^d} \varphi(x)\,dx = 1.
\end{equation*}
For $\delta > 0$, we then define
\begin{equation*}
    \kappa_\delta(x) = \frac{1}{\delta^d}\,\varphi\!\left(\frac{x}{\delta}\right),
    \qquad x \in \mathbb{R}^d.
\end{equation*}

For the standard convolution kernel on the torus $\mathbb{T}^d$, we fix the coordinate representation of $\mathbb{T}^d$ as $[-1/2,1/2]^d$. It is noted that the function $\varphi$ can be extended periodically, and thus can be regarded as a smooth kernel on $\mathbb{T}^d$.
}

The following lemma shows a product rule for the weak derivative.
\begin{lemma}\label{lem-2}
Let $f\in L^{\infty}([0,T];L^1(\mathbb{T}^d))$ be nonnegative with $\nabla\sqrt{f}\in L^2([0,T];L^2(\mathbb{T}^d;\mathbb{R}^d))$, then the chain rule for weak derivatives $\nabla f=2\sqrt{f}\nabla\sqrt{f}$ holds for almost every $(x,t)\in \mathbb{T}^d\times [0,T]$. Moreover, we have $\nabla f\in L^2([0,T];L^1(\mathbb{T}^d;\mathbb{R}^d))$.
\end{lemma}
\begin{proof}
Let $F(\zeta) = \zeta^2$ for $\zeta \geq 0$, and let $\{\kappa_{\delta}\}_{\delta > 0}$ be a sequence of standard convolution kernels on $\mathbb{R}$. For any $M>0$ and $\delta>0$, define $F_{M}(\zeta):=(\zeta\wedge M)^2$ and $F_{M,\delta}(\zeta):= \kappa_{\delta}\ast F_M (\zeta)$. Then for every $\varphi\in C^{\infty}(\mathbb{T}^d)$, applying the chain rule (see Evans \cite[Chapter 5, Exercise 17]{Eva10}) to $\langle\nabla F_{M,\delta}(\sqrt{f}),\varphi\rangle$, and then passing to the limits $\delta\rightarrow0$, $M\rightarrow\infty$, this completes the poof.
\end{proof}

\begin{lemma}\label{lem-1}
	Let $g\in L^{\infty}\left([0,T];L^1(\mathbb{T}^d)\right)$ and $f\in L^{\infty}\left([0,T];L^1(\mathbb{T}^d)\right)$ be nonnegative functions. In addition, assume that $f$ satisfies $\sqrt{f}\in L^2\left([0,T];H^1(\mathbb{T}^d)\right)$.
	\begin{enumerate}
		\item If $V$ satisfies Assumption (A1), then
		\begin{align}\label{eq-3.3}
		\int^T_0\|\nabla f\cdot V(t)\ast g\|_{L^1(\mathbb{T}^d)}\lesssim\int^T_0\|\nabla\sqrt{f}\|^{\frac{d}{p}+1}_{L^2(\mathbb{T}^d)}\|f\|^{\frac12-\frac{d}{2p}}_{L^1(\mathbb{T}^d)}\|V(t)\|_{L^p(\mathbb{T}^d)}\|g\|_{L^1(\mathbb{T}^d)}.	
		\end{align}
		\item If $V$ satisfies Assumption (A2), then
		\begin{align}\label{eq-3.4}
		\int^T_0\|f\nabla\cdot V(t)\ast g\|_{L^1(\mathbb{T}^d)}\lesssim\int^T_0\|\nabla\sqrt{f}\|^{\frac{d}{q}}_{L^2(\mathbb{T}^d)}\|f\|^{1-\frac{d}{2q}}_{L^1(\mathbb{T}^d)}\|\nabla\cdot V(t)\|_{L^q(\mathbb{T}^d)}\|g\|_{L^1(\mathbb{T}^d)}.
		\end{align}
	\end{enumerate}
\end{lemma}
	\begin{proof}
Based on Lemma \ref{lem-2}, by using H\"{o}lder's and convolutional Young's inequalities, we get
		\begin{align}\notag
		\int^T_0\|\nabla f\cdot V(t)\ast g\|_{L^1(\mathbb{T}^d)}
=&\int^T_0\|2\sqrt{f}\nabla \sqrt{f}\cdot V(t)\ast g\|_{L^1(\mathbb{T}^d)}\\
\label{eq-3.5}
\le &2\int^T_0\|\nabla\sqrt{f}\|_{L^2(\mathbb{T}^d)}
\|g\|_{L^1(\mathbb{T}^d)}\|V(t)\|_{L^{p}(\mathbb{T}^d)}\|\sqrt{f}\|_{L^{p'}(\mathbb{T}^d)},
		\end{align}
		where $\frac2{p}+\frac2{p'}=1$. Applying Gagliardo-Nirenberg interpolation inequality (\cite{BM18}) to $\|\sqrt{f}\|_{L^{p'}(\mathbb{T}^d)}$, there exists a constant $c\in(0,\infty)$ depending on $d$ such that
		\begin{align}\label{eq-3.6}
		\|\sqrt f\|_{L^{p'}(\mathbb{T}^{d})}\le c(d)\|\nabla\sqrt f\|^{\frac{d(p'-2)}{2p'}}_{L^{2}(\mathbb{T}^{d})}\|\sqrt f\|^{1-\frac{d(p'-2)}{2p'}}_{L^{2}(\mathbb{T}^{d})}
=c(d)\|f\|^{\frac{1}{2}-\frac{d}{2p}}_{L^{1}(\mathbb{T}^{d})}\|\nabla\sqrt f\|^{\frac{d}{p}}_{L^{2}(\mathbb{T}^{d})}.
		\end{align}
Substituting \eqref{eq-3.6} into \eqref{eq-3.5}, we get \eqref{eq-3.3}.
		
Using  H\"{o}lder's and convolutional Young's inequalities again to see that
		\begin{align}\label{eq-3.7}
		\int^T_0\|f\nabla\cdot V(t)\ast g\|_{L^1(\mathbb{T}^d)}\le\int^T_0\|f\|_{L^{q'}(\mathbb{T}^d)}\|\nabla\cdot V(t)\|_{L^q(\mathbb{T}^d)}\|g\|_{L^1(\mathbb{T}^d)},
		\end{align}
		where $\frac{1}{q'}+\frac{1}{q}=1$. Applying Gagliardo-Nirenberg interpolation inequality to $\| f\|_{L^{q'}(\mathbb{T}^d)}$, there exists a constant $c\in(0,\infty)$ depending on $d$ such that
		\begin{align}\label{eq-3.8}
		\| f\|_{L^{q'}(\mathbb{T}^d)}\le c(d)\|\nabla\sqrt f\|^{\frac{d(q'-1)}{q'}}_{L^{2}(\mathbb{T}^{d})}\|\sqrt f\|^{2-\frac{d(q'-1)}{q'}}_{L^{2}(\mathbb{T}^{d})}
=c(d)\|f\|^{1-\frac{d}{2q}}_{L^{1}(\mathbb{T}^{d})}\|\nabla\sqrt f\|^{\frac{d}{q}}_{L^{2}(\mathbb{T}^{d})}.
		\end{align}
		By substituting \eqref{eq-3.8} into \eqref{eq-3.7}, we get the desired result \eqref{eq-3.4}.
	\end{proof}

Since the interaction kernel $V$ is irregular, a product rule for the weak derivatives is needed as well.
\begin{lemma} \label{lem-3}
The following two properties hold.
\begin{enumerate}
		\item Let $f\in L^{\infty}([0,T];L^1(\mathbb{T}^d))$ and $g\in L^{\infty}\left([0,T];L^1(\mathbb{T}^d)\right)$ be nonnegative functions with $\sqrt{f}\in L^2\left([0,T];H^1(\mathbb{T}^d)\right)$ and $\sqrt{g}\in L^2\left([0,T];H^1(\mathbb{T}^d)\right)$. Assume that $V$ satisfies Assumption (A1), then the product rule for weak derivatives $\nabla\cdot(fV\ast g)=\nabla f\cdot V\ast g+fV\ast(\nabla g)$ holds for almost every $(x,t)\in \mathbb{T}^d\times [0,T]$, where $V\ast(\nabla g):=\int_{\mathbb{T}^d}V(y)\cdot\nabla_x g(x-y)\mathrm{d}y$.
  \item Let $f\in L^{\infty}([0,T];L^1(\mathbb{T}^d))$ and $g\in L^{\infty}\left([0,T];L^1(\mathbb{T}^d)\right)$ be nonnegative functions with $\nabla\sqrt{f}\in L^2([0,T];L^2(\mathbb{T}^d;\mathbb{R}^d))$. Assume that $V$ satisfies Assumptions (A1) and (A2), then the product rule for weak derivatives $\nabla\cdot(fV\ast g)=\nabla f\cdot V\ast g+f (\nabla \cdot V)\ast g$ holds for almost every $(x,t)\in \mathbb{T}^d\times [0,T]$, where $(\nabla \cdot V)\ast g:=\int_{\mathbb{T}^d}(\nabla\cdot V(x-y)) g(y)\mathrm{d}y$.
\end{enumerate}
\end{lemma}
The proof of Lemma \ref{lem-3} follows from considering a regularization of the kernel, and then passing to the limit, thus we omit the proof.

%
%

\begin{lemma}\label{lem-3.4}
	Let $f,g\in L^{\infty}\left([0,T];L^1(\mathbb{T}^d)\right)$ be nonnegative functions with $\sqrt{f}\in L^2\left([0,T];H^1(\mathbb{T}^d)\right)$ and $\sqrt{g}\in L^2\left([0,T];H^1(\mathbb{T}^d)\right)$. Suppose that  Assumption (A1) holds, then there exists a constant $C<\infty$ such that
	\begin{align*}
	\int_{0}^{T}\int_{\mathbb{T}^{d}}|\nabla\cdot\left(fV(t)\ast g\right)|\leq C.
	\end{align*}
\end{lemma}
\begin{proof}
Referring to (1) in Lemma \ref{lem-3}, we have
\begin{align*}
	\int_{0}^{T}\int_{\mathbb{T}^{d}}|\nabla\cdot\left(fV(t)\ast g\right)|\le\int_{0}^{T}\|\nabla f\cdot V(t)\ast g\|_{L^1(\mathbb{T}^d)}+\int_{0}^{T}\| f (V(t)\ast \nabla g)\|_{L^1(\mathbb{T}^d)}.
	\end{align*}
  \eqref{eq-3.3} and \eqref{eq-3.1} together yield
	\begin{align}\label{eq-3.10}
	\int_{0}^{T}\|\nabla f\cdot V(t)\ast g\|_{L^1(\mathbb{T}^d)}\lesssim&C(T)\|f\|^{\frac12-\frac{d}{2p}}_{L^{\infty}([0,T];L^1(\mathbb{T}^d))}\|g\|_{L^{\infty}([0,T];L^1(\mathbb{T}^d))}\notag\\
&\cdot\|\nabla\sqrt{f}\|^{\frac{d+p}{p}}_{L^2([0,T];L^2(\mathbb{T}^d))}\|V\|_{L^{p*}([0,T];L^p(\mathbb{T}^d;\mathbb{R}^d))}.
	\end{align}
For $\frac{1}{p'}+\frac{1}{p}=1$,
using Lemma \ref{lem-2} and H\"{o}lder's inequality to see that
	\begin{align}
	\int_{0}^{T}\| f (V(t)\ast \nabla g)\|_{L^1(\mathbb{T}^d)}\leq& 2\int_{0}^{T}\|f\|_{L^{p'}(\mathbb{T}^d)}\|V(t)\ast(\sqrt{g}\nabla\sqrt{g})\|_{L^p(\mathbb{T}^d)}\notag\\
\le&2\int_{0}^{T}\|f\|_{L^{p'}(\mathbb{T}^d)}\|V(t)\|_{L^p(\mathbb{T}^d)}\|\sqrt{g}\|_{L^2(\mathbb{T}^d)}\|\nabla\sqrt{g}\|_{L^2(\mathbb{T}^d)}\notag\\
 \lesssim&\|f\|^{1-\frac{d}{2p}}_{L^{\infty}\left([0,T];L^1(\mathbb{T}^d)\right)}\|g\|^{\frac12}_{L^{\infty}\left([0,T];L^1(\mathbb{T}^d)\right)}\notag\\
&\cdot\int_{0}^{T}\|\nabla\sqrt{f}\|^{\frac{d}{p}}_{L^2(\mathbb{T}^d)}\|V(t)\|_{L^p(\mathbb{T}^d)}\|\nabla\sqrt{g}\|_{L^2(\mathbb{T}^d)},
	\end{align}
where \eqref{eq-3.8} has been used for the last inequality.
%
Similar to \eqref{eq-3.1}, we can derive that
	\begin{align}\label{eq-3.12}
	&\int_{0}^{T}\|\nabla\sqrt{f}\|^{\frac{d}{p}}_{L^2(\mathbb{T}^d)}\|V(t)\|_{L^p(\mathbb{T}^d)}\|\nabla\sqrt{g}\|_{L^2(\mathbb{T}^d)}\notag\\
	\lesssim& \Big(\int_{0}^{T}\|\nabla\sqrt{f(t)}\|^{2}_{L^2(\mathbb{T}^d)} \Big)^{\frac{d}{2p}}\Big(\int_{0}^{T}\|\nabla\sqrt{g(t)}\|^{2}_{L^2(\mathbb{T}^d)} \Big)^{\frac{1}{2}}\|V\|_{L^{p^*}([0,T];L^p(\mathbb{T}^d;\mathbb{R}^d))}.
	\end{align}
	Combining (\ref{eq-3.10})-\eqref{eq-3.12}, we conclude the desired result.
\end{proof}

\section{Uniqueness of renormalized kinetic solutions}\label{sec-4}
In this section, we aim to show the pathwise uniqueness of renormalized kinetic solutions of \eqref{eq-2.2}.  As stated in the Introduction part, unlike \cite{FG21}, we have to employ the entropy estimate \eqref{eq-2.5} and a stochastic Gronwall's lemma to handle the kernel term. The latter is formulated as follows, whose proof  can be found in \cite[Lemma 5.3]{GZ}.
\begin{lemma}\label{lem-4.1}
	Let $T>0$. Assume that $X,Y,Z,R:[0,T)\times\Omega\to\mathbb{R}$ are real-valued, nonnegative stochastic processes. Let $\tau<T$, $\mathbb{P}-a.s.$ be a stopping time such that
	\begin{equation}\label{eq-4.1}
	\mathbb{E}\int_0^{\tau}(RX+Z)\mathrm{d}s<\infty.
	\end{equation}
	Assume that for some constant $M<\infty$,
	\begin{equation}\label{eq-4.2}
	\int_0^{\tau}R\mathrm{d}s<M,\ \ \mathbb{P}-a.s.
	\end{equation}
	Suppose that for all stopping times $0\leq\tau_a\leq\tau_b\leq\tau$,
	\begin{equation}\label{eq-4.3}
	\mathbb{E}\Big(\sup_{t\in[\tau_a,\tau_b]}X+\int_{\tau_a}^{\tau_b}Y\mathrm{d}s\Big)\leq C_{0}\mathbb{E}\Big(X(\tau_a)+\int_{\tau_a}^{\tau_b}(RX+Z)\mathrm{d}s\Big),
	\end{equation}
	where $C_{0}$ is a constant independent of $\tau_a$ and $\tau_b$. Then there exists a constant $C$ depending on $C_0, T$ and $M$ such that
	\begin{equation*}
	\mathbb{E}\Big(\sup_{t\in[0,\tau]}X+\int_0^{\tau}Y\mathrm{d}s\Big)\leq C(C_0,T,M)\mathbb{E}\Big(X(0)+\int_0^{\tau}Z\mathrm{d}s\Big).
	\end{equation*}
\end{lemma}

In order to restrict the values of kinetic solutions away from infinity and zero, we introduce cutoff functions by the same way as \cite{FG21}.
For every $\beta\in(0,1)$, let $\varphi_{\beta}:\mathbb{R}\to[0,1]$ be the unique nondecreasing piecewise linear function that satisfies
\begin{equation}\label{kk-46}
\varphi_{\beta}(\xi)=1\  {\rm{if}}\  \xi\ge\beta, \ \varphi_{\beta}(\xi)=0\ {\rm{if}}\  \xi\le\frac{\beta}{2},\  {\rm{and}}\  \varphi'_{\beta}=\frac{2}{\beta}\mathbf{1}_{\{\frac{\beta}{2}<\xi<\beta\}}.
\end{equation}
For every $M\in\mathbb{N}$, let $\zeta_M:\mathbb{R}\to[0,1]$ be the unique nonincreasing piecewise linear function satisfying
\begin{equation}\label{kk-47}
\zeta_M(\xi)=0\ {\rm{if}}\ \xi\ge M+1,\ \zeta_M(\xi)=1\ {\rm{if}}\ \xi\le M,\ {\rm{and}}\ \zeta'_M=-\mathbf{1}_{\{M<\xi<M+1\}}.
\end{equation}
For every $\varepsilon,\delta\in(0,1)$, let $\kappa^{\varepsilon}_{d}:\mathbb{T}^d\to\left[0,\infty\right)$ and $\kappa^{\delta}_{1}:\mathbb{R}\to\left[0,\infty\right)$ be standard convolution kernels of scales $\varepsilon$ and $\delta$ on $\mathbb{T}^d$ and $\mathbb{R}$, respectively {\color{black}(defined in the same way as in the proof of Lemma 3.2).} Let $\kappa^{\varepsilon,\delta}$ be defined by
\begin{align}\label{kk-45}
\kappa^{\varepsilon,\delta}(x,y,\xi,\eta)=\kappa^{\varepsilon}_{d}(x-y)\kappa^{\delta}_{1}(\xi-\eta)\ \text{for\ every}\ (x,y,\xi,\eta)\in(\mathbb{T}^d)^2\times\mathbb{R}^2.
\end{align}

Now, we are ready to prove the uniqueness of renormalized kinetic solutions of \eqref{eq-2.2} by doubling variables method.
\begin{theorem}\label{the-uniq}
	Suppose that Assumptions (A1)-(A2) are in force. Let $\hat{\rho}^1,\hat{\rho}^2\in\text{{\rm{Ent}}}\left(\mathbb{T}^{d}\right)$. Let $\rho^1,\rho^2$ be renormalized kinetic solutions of (\ref{eq-1.1}) in the sense of Definition \ref{def-2.4} with the corresponding initial data $\rho^1(\cdot,0)=\hat{\rho}^1,\rho^2(\cdot,0)=\hat{\rho}^2$. If $\hat{\rho}^1(x)=\hat{\rho}^2(x)$ for almost every $x\in\mathbb{T}^d$, then $\mathbb{P}$-almost surely,
	\begin{equation}
	\sup_{t\in[0,T]}\|\rho^1(\cdot,t)-\rho^2(\cdot,t)\|_{L^1(\mathbb{T}^d)}=0.\notag
	\end{equation}
\end{theorem}
\begin{proof}
 Let $\chi^1$ and $\chi^2$ be kinetic functions of $\rho^{1}$ and $\rho^{2}$, respectively. Recall that for every $\varepsilon,\delta\in(0,1)$, $\kappa^{\varepsilon,\delta}$ is the convolution kernel given by (\ref{kk-45}). Then, for every $i\in\{1,2\}$, we define
\begin{align*}
  \chi^{\varepsilon,\delta}_{t,i}(y,\eta)=(\chi^i(\cdot,\cdot,t)\ast\kappa^{\varepsilon,\delta})
(y,\eta).
\end{align*}
According to Definition \ref{def-2.4} and the Kolmogorov's continuity criterion, for every $\varepsilon,\delta\in(0,1)$, there exists a subset of full probability such that, for every $i\in\{1,2\}$, $(y,\eta)\in\mathbb{T}^d\times(\frac{\delta}{2},\infty)$, and $t\in[0,T]$,
	\begin{align}\label{eq-4.4}
	\left.\chi^{\varepsilon,\delta}_{r,i}(y,\eta)\right|^{r=t}_{r=0}=&\nabla_{y}\cdot\Big(\int^t_0\int_{\mathbb{T}^d}\nabla\rho^i\kappa^{\varepsilon,\delta}(x,y,\rho^i,\eta)\Big)+\partial_{\eta}\Big(\int^t_0{\color{black}\int_0^{\infty}}\int_{\mathbb{T}^d}\kappa^{\varepsilon,\delta}(x,y,\xi,\eta)\mathrm{d}q^i\Big)\notag\\
	&+\nabla_{y}\cdot\Big(\frac18\int^t_0\int_{\mathbb{T}^d}(F_1(x)(\rho^i)^{-1}\nabla\rho^i+2F_2(x))\kappa^{\varepsilon,\delta}(x,y,\rho^i,\eta)\Big)\notag\\
	&-\partial_{\eta}\Big(\frac14\int^t_0\int_{\mathbb{T}^d}(2F_3(x)\rho^i+\nabla\rho^i\cdot F_2(x))\kappa^{\varepsilon,\delta}(x,y,\rho^i,\eta)\Big)\notag\\
	&-\int^t_0\int_{\mathbb{T}^d}\kappa^{\varepsilon,\delta}(x,y,\rho^i,\eta)\nabla\cdot(\rho^iV(r)\ast\rho^i)-\int^t_0\int_{\mathbb{T}^d}\kappa^{\varepsilon,\delta}(x,y,\rho^i,\eta)\nabla\cdot\big(\sqrt{\rho^i}\mathrm{d}W^F\big).
	\end{align}

In view of \eqref{eq-4.4} and  \cite[Lemma 4.3]{FG21}, by the approach similar to the derivation of (4.14) in \cite{FG21}, we have
 almost surely, for every $\varepsilon,\beta\in(0,1),\ M\in\mathbb{N}$, $\delta\in(0,\frac{\beta}{4})$ and $t\in[0,T]$,
	\begin{align}\label{eq-4.8}
	\left.{\color{black}\int_0^{\infty}}\int_{\mathbb{T}^d}\Big|\chi_{r,1}^{\varepsilon,\delta}(y,\eta)
	-\chi_{r,2}^{\varepsilon,\delta}(y,\eta)\Big|^{2}\varphi_{\beta}(\eta)\zeta_{M}(\eta)\right|^{r=t}_{r=0}
	=-2I_t^{\mathrm{err}}-2I_t^{\mathrm{meas}}
	+I_t^{\mathrm{mart}}+I_t^{\mathrm{cut}}+I_t^{\mathrm{ker}}.
	\end{align}
	 The terms $I_t^{\mathrm{err}}, I_t^{\mathrm{meas}}, I_t^{\mathrm{mart}}$ and $I_t^{\mathrm{cut}}$ correspond respectively to the terms $I_t^{\mathrm{err}}, I_t^{\mathrm{meas}}, I_t^{\mathrm{mart}}$ and $I_t^{\mathrm{cut}}$ in \cite[Theorem 4.6]{FG21}. However, the kernel term $I_t^{\mathrm{ker}}$ is the extra term that needs to be estimated.  For brevity, we directly refer to  \cite[Theorem 4.6]{FG21} to list the following results.
	\begin{equation}\label{eq-4.9}	\limsup_{\delta\to0}\left(\limsup_{\varepsilon\to0}\left|I_t^{\mathrm{err}}\right|\right)=0,
	\end{equation}
	\begin{equation}\label{eq-4.10}
		I_t^{\text{meas}}\geq 0,
	\end{equation}
	\begin{equation}\label{eq-4.11}
		\lim_{M\to\infty}\left(\lim_{\beta\to0}\left(\lim_{\delta\to0}\left(\lim_{\varepsilon\to0}I_t^{\text{mart}}\right)\right)\right)=0,
	\end{equation}
and
	\begin{equation}\label{eq-4.12}
		\lim_{M\to\infty}\left(\lim_{\beta\to0}\left(\lim_{\delta\to0}\left(\lim_{\varepsilon\to0}I_t^{\mathrm{cut}}\right)\right)\right)=0,
	\end{equation}
	hold almost surely for every $t\in[0,T]$. Let
	\begin{align*}
		\bar{\kappa}_{r,1}^{\varepsilon,\delta}(x,y,\eta)=\kappa^{\varepsilon,\delta}\left(x,y,\rho^1(x,r),\eta\right)\text { and }\bar{\kappa}_{r,2}^{\varepsilon,\delta}\left(x^{\prime},y,\eta\right)=\kappa^{\varepsilon,\delta}\left(x^{\prime},y,\rho^2\left(x^{\prime},r\right),\eta\right).
	\end{align*}

We mainly deal with kernel term $I_t^{\text{ker}}$ with the following form:
	\begin{align*}
	I_t^{\text{ker}}
	=&\int_{0}^{t}{\color{black}\int_0^{\infty}}\int_{\left(\mathbb{T}^{d}\right)^{2}}\bar{\kappa}_{r,1}^{\varepsilon,\delta}\nabla\cdot(\rho^1V(r)\ast\rho^1)\left(2\chi_{r,2}^{\varepsilon,\delta}-1\right)\varphi_{\beta}(\eta)\zeta_{M}(\eta)\notag\\
	&+\int_{0}^{t}{\color{black}\int_0^{\infty}}\int_{\left(\mathbb{T}^{d}\right)^{2}}\bar{\kappa}_{r,2}^{\varepsilon,\delta}\nabla\cdot(\rho^2V(r)\ast\rho^2)\left(2\chi_{r,1}^{\varepsilon,\delta}-1\right)\varphi_{\beta}(\eta)\zeta_{M}(\eta).
	\end{align*}
	By using Lemma \ref{lem-3.4}, we deduce that $\nabla\cdot(\rho^iV\ast\rho^i)$ is $L^1(\Omega\times[0,T]\times \mathbb{T}^d)$-integrable. It follows from the definition of $\kappa^{\varepsilon,\delta}$, the boundedness of the kinetic functions and the dominated convergence theorem that, after passing to a subsequence $\varepsilon\to0$, almost surely for every $t\in[0,T]$,
	\begin{align}\label{eq-4.13}
	\lim_{\varepsilon\to0}I_t^{\text{ker}}=&\int_{0}^{t}{\color{black}\int_0^{\infty}}\int_{\mathbb{T}^{d}}\bar{\kappa}_{r,1}^{\delta}\left(2\chi_{r,2}^{\delta}-1\right)\varphi_{\beta}(\eta)\zeta_{M}(\eta)\nabla\cdot(\rho^1V(r)\ast\rho^1) \notag\\
	&+\int_{0}^{t}{\color{black}\int_0^{\infty}}\int_{\mathbb{T}^{d}}\bar{\kappa}_{r,2}^{\delta}\left(2\chi_{r,1}^{\delta}-1\right)\varphi_{\beta}(\eta)\zeta_{M}(\eta)\nabla\cdot(\rho^2V(r)\ast\rho^2) ,
	\end{align}
	where $\chi_{r,i}^{\delta}(y,\eta)=\left(\chi_{r}^{i}(y,\cdot)\ast\kappa_{1}^{\delta}\right)(\eta)$ and $\bar{\kappa}_{r,i}^\delta(y,\eta)=\kappa_1^\delta\left(\rho^i(y, r)-\eta\right)$ for each $i \in\{1,2\}$.
	Let us focus on the first term on the righthand side of (\ref{eq-4.13}).
	In view of $|2\chi_{r,2}^{\delta}-1|\le1$, by the definitions of $\varphi_{\beta}$, $\zeta_{M}$ and $\bar{\kappa}_{r,1}^{\delta}$, and the boundedness of the kinetic function, we deduce that there exists a constant $c\in(0,\infty)$ depending on $\beta$ such that for every $t\in[0,T]$,
	\begin{align}
	&\mathbb{E}\left|\int_{0}^{t}{\color{black}\int_0^{\infty}}\int_{\mathbb{T}^{d}}\bar{\kappa}_{r,1}^{\delta}\left(2\chi_{r,2}^{\delta}-1\right)\left(\varphi_{\beta}(\eta)\zeta_{M}(\eta)-\varphi_{\beta}\left(\rho^{1}\right)\zeta_{M}\left(\rho^{1}\right)\right)\nabla\cdot(\rho^1V(r)\ast\rho^1) \right|\notag\\
	\le&\mathbb{E}\int_{0}^{T}{\color{black}\int_0^{\infty}}\int_{\mathbb{T}^{d}}\left|\bar{\kappa}_{r,1}^{\delta}\left(\varphi_{\beta}(\eta)\zeta_{M}(\eta)-\varphi_{\beta}\left(\rho^{1}\right)\zeta_{M}\left(\rho^{1}\right)\right)\nabla\cdot(\rho^1V(r)\ast\rho^1)\right| \notag\\
	\le&\mathbb{E}\int_{0}^{T}{\color{black}\int_0^{\infty}}\int_{\mathbb{T}^{d}}\bar{\kappa}_{r,1}^{\delta}c(\beta)\delta\mathbf{1}_{\left\{\beta/2-\delta<\rho^{1}<M+1+\delta\right\}}\left|\nabla\cdot(\rho^1V(r)\ast\rho^1)\right| \notag\\
	\le&c(\beta)\delta\mathbb{E}\int_{0}^{T}\int_{\mathbb{T}^{d}}\left|\nabla\cdot(\rho^1V(r)\ast\rho^1)\right|\notag.
	\end{align}
	Since for each $i \in\{1,2\}$, $\nabla\cdot(\rho^iV\ast\rho^i)$ is $L^1(\Omega\times[0,T]\times \mathbb{T}^d)$-integrable, passing to a subsequence $\delta\to0$, almost surely,
	\begin{equation}\label{eq-4.14}
	\lim_{\delta\to0}\left|\int_{0}^{t}{\color{black}\int_0^{\infty}}\int_{\mathbb{T}^{d}}\bar{\kappa}_{r,1}^{\delta}\left(2\chi_{r,2}^{\delta}-1\right)\left(\varphi_{\beta}(\eta)\zeta_{M}(\eta)-\varphi_{\beta}\left(\rho^{1}\right)\zeta_{M}\left(\rho^{1}\right)\right)\nabla\cdot(\rho^1V(r)\ast\rho^1) \right|=0.
	\end{equation}
	Similarly, for the second term on the righthand side of (\ref{eq-4.13}), it holds that
	\begin{equation}\label{kk-49}
	\lim_{\delta\to0}\left|\int_{0}^{t}{\color{black}\int_0^{\infty}}\int_{\mathbb{T}^{d}}\bar{\kappa}_{r,2}^{\delta}
	\left(2\chi_{r,1}^{\delta}-1\right)
	\left(\varphi_{\beta}(\eta)\zeta_{M}(\eta)-\varphi_{\beta}
	\left(\rho^{2}\right)\zeta_{M}\left(\rho^{2}\right)\right)\nabla\cdot(\rho^2V(r)\ast\rho^2) \right|=0.
	\end{equation}
	Moreover, referring to \cite[(4.22)]{FG21}
and in view of $\varphi_{\beta}(0)=0$, we deduce that pointwise
	\begin{equation}\label{eq-4.16}
	\lim_{\delta\to0}\left({\color{black}\int_0^{\infty}}\bar{\kappa}_{r,1}^{\delta}\left(2\chi_{r,2}^{\delta}-1\right)\mathrm{d}\eta\right)\varphi_{\beta}\left(\rho^{1}\right)=\left(\mathbf{1}_{\left\{\rho^{1}=\rho^{2}\right\}}+2\mathbf{1}_{\left\{0\le\rho^{1}<\rho^{2}\right\}}-1\right)\varphi_{\beta}\left(\rho^{1}\right).
	\end{equation}
	Combining \eqref{eq-4.13}-\eqref{eq-4.16}, passing to a subsequence $\delta\to0$, almost surely,
	\begin{align}\notag
	\lim_{\delta\to0}\left(\lim_{\varepsilon\to0}I_t^{\text{ker}}\right)
	=&\int_{0}^{t}\int_{\mathbb{T}^{d}}\left(\mathbf{1}_{\left\{\rho^{1}=\rho^{2}\right\}}
	+2\mathbf{1}_{\left\{\rho^{1}<\rho^{2}\right\}}-1\right)\varphi_{\beta}\left(\rho^{1}\right)
	\zeta_{M}\left(\rho^{1}\right)\nabla\cdot(\rho^1V(r)\ast\rho^1)\\
	\label{kk-50}
	&+\int_{0}^{t}\int_{\mathbb{T}^{d}}\left(\mathbf{1}_{\left\{\rho^{1}=\rho^{2}\right\}}
	+2\mathbf{1}_{\left\{\rho^{2}<\rho^{1}\right\}}-1\right)\varphi_{\beta}
	\left(\rho^{2}\right)\zeta_{M}\left(\rho^{2}\right)\nabla\cdot(\rho^2V(r)\ast\rho^2).
	\end{align}
	Now, we claim that along subsequences $\beta\to0$ and $M\to\infty$, almost surely, for every $i\in\{1,2\}$,
	\begin{align}\label{eq-4.22}
	\lim_{M\to\infty}\left(\lim_{\beta\to0}\varphi_{\beta}\left(\rho^{i}\right)\zeta_{M}\left(\rho^{i}\right)\nabla\cdot(\rho^iV\ast\rho^i)\right)=\nabla\cdot(\rho^iV\ast\rho^i)\ \text{ strongly in }L^{1}\left(\mathbb{T}^{d}\times[0,T]\right).
	\end{align}
	Indeed, from the definitions of $\varphi_{\beta}$ and $\zeta_{M}$, it gives that for every $i\in\{1,2\}$,
	\begin{align}\label{eq-4.17}
	&\mathbb{E}\int_{0}^{T}\int_{\mathbb{T}^{d}}\left|\varphi_{\beta}(\rho^{i})\zeta_{M}(\rho^{i})\nabla\cdot(\rho^iV(r)\ast\rho^i)-\nabla\cdot(\rho^iV(r)\ast\rho^i)\right|\notag\\
	\le&\mathbb{E}\int_{0}^{T}\int_{\mathbb{T}^{d}}\mathbf{1}_{\left\{0\le\rho^{i}<\beta\right\}}\left|\nabla\cdot(\rho^iV(r)\ast\rho^i)\right|+\mathbb{E}\int_{0}^{T}\int_{\mathbb{T}^{d}}\mathbf{1}_{\left\{\rho^{i}\ge M\right\}}\left|\nabla\cdot(\rho^iV(r)\ast\rho^i)\right|.
	\end{align}
	Clearly, the second term on the righthand side of \eqref{eq-4.17} converges to zero as $M\to\infty$. For the first term on the righthand side of \eqref{eq-4.17}, by Lemma \ref{lem-2} and (2) in Lemma \ref{lem-3}, it follows that for every $i\in\{1,2\}$,
	\begin{align}\label{eq-4.18}
	&\mathbb{E}\int_{0}^{T}\int_{\mathbb{T}^{d}}\mathbf{1}_{\left\{0\le\rho^{i}<\beta\right\}}\left|\nabla\cdot(\rho^iV(r)\ast\rho^i)\right|\notag\\
	\le&\mathbb{E}\int_{0}^{T}\int_{\mathbb{T}^{d}}\mathbf{1}_{\left\{0\le\rho^{i}<\beta\right\}}\left|\nabla\rho^i\cdot V(r)\ast\rho^i\right|+\mathbb{E}\int_{0}^{T}\int_{\mathbb{T}^{d}}\mathbf{1}_{\left\{0\le\rho^{i}<\beta\right\}}\left|\rho^i\nabla\cdot V(r)\ast\rho^i\right|\notag\\
	\le&2\beta^{\frac12}\mathbb{E}\int_{0}^{T}\int_{\mathbb{T}^{d}}\mathbf{1}_{\left\{0\le\rho^{i}<\beta\right\}}\left|\nabla\sqrt{\rho^i}\cdot V(r)\ast\rho^i\right|+\beta\mathbb{E}\int_{0}^{T}\int_{\mathbb{T}^{d}}\mathbf{1}_{\left\{0\le\rho^{i}<\beta\right\}}\left|\nabla\cdot V(r)\ast\rho^i\right|.
	\end{align}
	By using H\"{o}lder's inequality, convolutional Young's inequality and \eqref{eq-2.4}, we deduce that for every $i\in\{1,2\}$,
	\begin{align*}
	\mathbb{E}\int_{0}^{T}\|\nabla\sqrt{\rho^i}\cdot V(r)\ast\rho^i\|_{L^1(\mathbb{T}^d)}
	\le&\|\hat{\rho}^i\|_{L^1(\mathbb{T}^d)}\Big(\int_{0}^{T}\|V(r)\|_{L^2(\mathbb{T}^d)}^2\Big)^{\frac12}
	\Big(\mathbb{E}\int_{0}^{T}\|\nabla\sqrt{\rho^i}\|^2_{L^2(\mathbb{T}^d)}\Big)^{\frac12},\\
	\mathbb{E}\int_{0}^{T}\int_{\mathbb{T}^{d}}\mathbf{1}_{\left\{0\le\rho^{i}<\beta\right\}}\left|\nabla\cdot V(r)\ast\rho^i\right|\le& \mathbb{E}\int_{0}^{T}\|\nabla\cdot V(r)\ast\rho^i\|_{L^{1}(\mathbb{T}^d)}\le\|\hat{\rho}^i\|_{L^{1}(\mathbb{T}^d)}\int_{0}^{T}\|\nabla\cdot V(r)\|_{L^{1}(\mathbb{T}^d)},
	\end{align*}
	  where we have used the preservation of mass \eqref{eq-2.4} in the last step. Owing to (\ref{eq-4.18}) and Assumptions (A1) and (A2), we get the desired result (\ref{eq-4.22}).
	%
	
	Combining (\ref{kk-50}) and (\ref{eq-4.22}), we get
	\begin{align}\notag &\lim_{M\to\infty}\left(\lim_{\beta\to0}\left(\lim_{\delta\to0}\left(\lim_{\varepsilon\to0}I_t^{\text{ker}}\right)\right)\right)\\
	\notag 
=&\int_{0}^{t}\int_{\mathbb{T}^{d}}\left(\mathbf{1}_{\left\{\rho^{1}=\rho^{2}\right\}}+2\mathbf{1}_{\left\{\rho^{1}<\rho^{2}\right\}}-1\right)\left(\nabla\cdot(\rho^1V(r)\ast\rho^1)-\nabla\cdot(\rho^2V(r)\ast\rho^2)\right)\\
	\notag &+\int_{0}^{t}\int_{\mathbb{T}^{d}}\left(\mathbf{1}_{\left\{\rho^{1}=\rho^{2}\right\}}+2\mathbf{1}_{\left\{\rho^{1}<\rho^{2}\right\}}-1+\mathbf{1}_{\left\{\rho^{1}=\rho^{2}\right\}}+2\mathbf{1}_{\left\{\rho^{2}<\rho^{1}\right\}}-1\right)
	\nabla\cdot(\rho^2V(r)\ast\rho^2)\\
	\label{eq-4.23}
	=:& \tilde{J}_1+\tilde{J}_2.
	\end{align}
	Since $\mathbf{1}_{\left\{\rho^{2}<\rho^{1}\right\}}
	=1-\mathbf{1}_{\left\{\rho^{1}=\rho^{2}\right\}}-\mathbf{1}_{\left\{\rho^{1}<\rho^{2}\right\}}$, the $L^1(\Omega\times [0,T]\times \mathbb{T}^d)$-integrability of $\nabla\cdot(\rho^iV\ast\rho^i)$ for every $i\in\{1,2\}$ implies that
	\begin{equation}\label{eq-4.24}
	\tilde{J}_2=0.
	\end{equation}
	For the term $\tilde{J}_1$, by chain rule and the identity $\sgn(\rho^2-\rho^1)=\mathbf{1}_{\left\{\rho^{1}=\rho^{2}\right\}}+2\mathbf{1}_{\left\{\rho^{1}<\rho^{2}\right\}}-1$, we have
	\begin{align*}
	\tilde{J}_1=:\tilde{J}_{11}+\tilde{J}_{12},
	\end{align*}
	where
	\begin{align*}		\tilde{J}_{11}&=\int_{0}^{t}\int_{\mathbb{T}^d}\sgn(\rho^2-\rho^1)\nabla\cdot\left((\rho^1-\rho^2)V(r)\ast\rho^1\right),\\	\tilde{J}_{12}&=\int_{0}^{t}\int_{\mathbb{T}^d}\sgn(\rho^2-\rho^1)\nabla\cdot\left(\rho^2V(r)\ast(\rho^1-\rho^2)\right).
	\end{align*}
	Define $\sgn^{\delta}:=(\sgn\ast\kappa_{1}^{\delta})$ for every $\delta\in(0,1)$. By  integration by parts formula, we get that almost surely for every $t\in[0,T]$,
	\begin{align}
	\tilde{J}_{11} =&\lim_{\delta\to0}\int_{0}^{t}\int_{\mathbb{T}^d}\sgn^{\delta}(\rho^2-\rho^1)\nabla\cdot\left((\rho^1-\rho^2)V(r)\ast\rho^1\right)\notag\\
	=&-\lim_{\delta\to0}\int_{0}^{t}\int_{\mathbb{T}^d}(\sgn^{\delta})'(\rho^2-\rho^1)(\rho^1-\rho^2)\nabla(\rho^2-\rho^1)\cdot V(r)\ast\rho^1\notag.
	\end{align}
	It follows from  the uniform boundedness of $(\delta\kappa_{1}^{\delta})$ in $\delta\in(0,\beta/4)$ that there exists $c \in(0,\infty)$ independent of $\delta$ but depending on the convolution kernel such that for all $\delta\in(0,\beta/4)$,
	\begin{align*}
	\left|(\sgn^{\delta})'(\rho^2-\rho^1)(\rho^1-\rho^2)\right|=2|\kappa_{1}^{\delta}(\rho^2-\rho^1)(\rho^1-\rho^2)|
	\le c\mathbf{1}_{\left\{0<|\rho^{1}-\rho^{2}|<\delta\right\}}.
	\end{align*}
	Moreover, by \eqref{eq-3.3}, we get the $L^1(\Omega\times [0,T]\times \mathbb{T}^d)$-integrability of $\nabla(\rho^2-\rho^1)\cdot V\ast\rho^1$. As a result, we deduce that, almost surely for every $t\in[0,T]$,
	\begin{equation}\label{eq-4.27}
	\tilde{J}_{11}=0.
	\end{equation}
	Regarding to the term $\tilde{J}_{12}$,
	by (2) in Lemma \ref{lem-3}, \eqref{eq-3.3} and \eqref{eq-3.4}, we deduce that almost surely for every $t\in[0,T]$,
	\begin{align}\label{eq-4.28}
	\tilde{J}_{12}\le&\int_{0}^{t}\int_{\mathbb{T}^d}\left|\nabla\rho^2\cdot V(r)\ast(\rho^1-\rho^2)\right|+\int_{0}^{t}\int_{\mathbb{T}^d}\left|\rho^2(\nabla\cdot V(r))\ast(\rho^1-\rho^2)\right|\notag\\
	\le&C(\|\hat{\rho}^2\|_{L^1(\mathbb{T}^d)},p,d)\int_{0}^{t}\|V(r)\|_{L^{p}(\mathbb{T}^d)}\|\nabla\sqrt{\rho^2}\|^{1+\frac{d}{p}}_{L^2(\mathbb{T}^d)}\|\rho^1-\rho^2\|_{L^{1}(\mathbb{T}^d)}\notag\\
	&+C(\|\hat{\rho}^2\|_{L^1(\mathbb{T}^d)},q,d)\int_{0}^{t}\|\nabla\cdot V(r)\|_{L^{q}(\mathbb{T}^d)}\|\nabla\sqrt{\rho^2}\|^{\frac{d}{q}}_{L^{2}(\mathbb{T}^{d})}\|\rho^1-\rho^2\|_{L^{1}(\mathbb{T}^d)}.
	\end{align}
	Combining \eqref{eq-4.23}-\eqref{eq-4.28}, we conclude that almost surely for every $t\in[0,T]$,
	\begin{align}\label{eq-4.29}
	\lim_{M\to\infty}\lim_{\beta\to0}\lim_{\delta\to0}\lim_{\varepsilon\to0}I_t^{\text{ker}}\le&C(\|\hat{\rho}^2\|_{L^1(\mathbb{T}^d)},p,d)\int_{0}^{t}\|V(r)\|_{L^{p}(\mathbb{T}^d)}\|\nabla\sqrt{\rho^2}\|^{1+\frac{d}{p}}_{L^2(\mathbb{T}^d)}\|\rho^1-\rho^2\|_{L^{1}(\mathbb{T}^d)}\notag\\
	+&C(\|\hat{\rho}^2\|_{L^1(\mathbb{T}^d)},q,d)\int_{0}^{t}\|\nabla\cdot V(r)\|_{L^{q}(\mathbb{T}^d)}\|\nabla\sqrt{\rho^2}\|^{\frac{d}{q}}_{L^{2}(\mathbb{T}^{d})}\|\rho^1-\rho^2\|_{L^{1}(\mathbb{T}^d)}.
	\end{align}

	$\mathbf{Conclusion.}$  Based on the properties of the kinetic function, \eqref{eq-4.9}-\eqref{eq-4.12}, it follows that almost surely for every $t\in[0,T]$,
	\begin{align}\label{eq-4.30}
	{\color{black}\int_0^{\infty}}\int_{\mathbb{T}^{d}}\Big.\Big|\chi_{r}^{1}-\chi_{r}^{2}\Big|^{2}\Big|_{r=0}^{r=t}&=\lim_{M\to\infty}\Big(\lim_{\beta\to0}\Big(\lim_{\delta\to0}\Big(\lim_{\varepsilon\to0}{\color{black}\int_0^{\infty}}\int_{\mathbb{T}^{d}}\Big.\Big|\chi_{r,1}^{\varepsilon,\delta}-\chi_{r,2}^{\varepsilon,\delta}\Big|^{2}\varphi_{\beta}\zeta_{M}\Big|_{r=0}^{r=t}\Big)\Big)\Big)\notag\\
	&=\lim_{M\to\infty}\Big(\lim_{\beta\to0}\Big(\lim_{\delta\to0}\Big(\lim_{\varepsilon\to0}\Big(-2I_t^{\text{err}}-2I_t^{\text{meas}}+I_t^{\text{mart}}+I_t^{\text{cut}}+I_t^{\text{ker}}\Big)\Big)\Big)\Big)\notag\\
	&\le\lim_{M\to\infty}\lim_{\beta\to0}\lim_{\delta\to0}\lim_{\varepsilon\to0}I_t^{\text {ker}}.
	\end{align}
	For any $N>0$, define a stopping time $\tau_{N}:=\inf\Big\{t\in[0,T];\int_{0}^{t}\|\nabla\sqrt{\rho^2}\|^2_{L^{2}(\mathbb{T}^{d})}>N\Big\}$. By Chebyshev's inequality and \eqref{eq-2.5}, we have
	\begin{align}\label{eq-4.31}
	\mathbb{P}(\tau_{N}\le T)\rightarrow0,\quad {\rm{as}}\ N\rightarrow \infty.
	\end{align}
	Combining \eqref{eq-4.30} with the definition of the kinetic function, we derive that for every stopping time $0\leq\tau_a\leq\tau_b\leq\tau_N\wedge t$,
	\begin{align}\notag
	&\mathbb{E}\sup_{r\in[\tau_a,\tau_b]}\left\|\rho^{1}(r)-\rho^{2}(r)\right\|_{L^{1}(\mathbb{T}^{d})}-\mathbb{E}\left\|\rho^{1}(\tau_a)-\rho^{2}(\tau_a)\right\|_{L^{1}\left(\mathbb{T}^{d}\right)}\\
	\notag
	\le&C(\|\hat{\rho}^2\|_{L^1(\mathbb{T}^d)},p,d)\mathbb{E}\int_{\tau_a}^{\tau_b}
	\|V(r)\|_{L^{p}(\mathbb{T}^d)}\|\nabla\sqrt{\rho^2}\|^{1+\frac{d}{p}}_{L^2(\mathbb{T}^d)}
	\|\rho^1-\rho^2\|_{L^{1}(\mathbb{T}^d)}\\
	\label{kk-51}
	&+C(\|\hat{\rho}^2\|_{L^1(\mathbb{T}^d)},q,d)\mathbb{E}\int_{\tau_a}^{\tau_b}\|\nabla\cdot V(r)\|_{L^{q}(\mathbb{T}^d)}\|\nabla\sqrt{\rho^2}\|^{\frac{d}{q}}_{L^{2}(\mathbb{T}^{d})}\|\rho^1-\rho^2\|_{L^{1}(\mathbb{T}^d)}.
	\end{align}
	In the following, we aim to apply Lemma \ref{lem-4.1} to (\ref{kk-51}).
	Let $\tau=\tau_{N}\wedge T$, $X=\left\|\rho^{1}-\rho^{2}\right\|_{L^{1}\left(\mathbb{T}^{d}\right)}$, $Y, Z=0$ and
	\begin{align*}
	R:=C(\|\hat{\rho}^2\|_{L^1(\mathbb{T}^d)},q,d)\|\nabla\cdot V(r)\|_{L^{q}(\mathbb{T}^d)}\|\nabla\sqrt{\rho^2}\|^{\frac{d}{q}}_{L^{2}(\mathbb{T}^{d})}+C(\|\hat{\rho}^2\|_{L^1(\mathbb{T}^d)},p,d)\|V(r)\|_{L^{p}(\mathbb{T}^d)}\|\nabla\sqrt{\rho^2}\|^{1+\frac{d}{p}}_{L^2(\mathbb{T}^d)}.
	\end{align*}
Clearly, (\ref{kk-51}) implies \eqref{eq-4.3}. Moreover, it follows that
		\begin{align*}
			\int_{0}^{\tau}R\leq &C(T,\|\hat{\rho}^2\|_{L^1(\mathbb{T}^d)},q,d) \|\nabla\sqrt{\rho^2}\|^{\frac{d}{q}}_{L^2([0,\tau];L^2(\mathbb{T}^d))}\|\nabla\cdot V\|_{L^{q*}([0,T];L^q(\mathbb{T}^d))}\\ &+C(T,\|\hat{\rho}^2\|_{L^1(\mathbb{T}^d)},p,d)
			\|\nabla\sqrt{\rho^2}\|^{\frac{d+p}{p}}_{L^2([0,\tau];L^2(\mathbb{T}^d))}\|V\|_{L^{p*}([0,T];L^p(\mathbb{T}^d))}\\
			\leq& C(N,T,V,\|\hat{\rho}^2\|_{L^1(\mathbb{T}^d)},p,q,d),\ \mathbb{P}-a.s.
		\end{align*}
		which implies \eqref{eq-4.2} holds. Similarly, \eqref{eq-4.1} holds. By employing Lemma \ref{lem-4.1}, we have
	\begin{equation}\label{kk-52}
	\mathbb{E}\sup_{t\in[0,\tau]}\left\|\rho^{1}(\cdot,t)-\rho^{2}(\cdot,t)\right\|_{L^{1}\left(\mathbb{T}^{d}\right)}\le C(N,T,V,\|\hat{\rho}^2\|_{L^1(\mathbb{T}^d)},p,q,d)\mathbb{E}\left\|\hat{\rho}^{1}-\hat{\rho}^{2}\right\|_{L^{1}\left(\mathbb{T}^{d}\right)}.
	\end{equation}
	For any $\delta>0$, we can deduce that

\begin{align}\label{kk-53}
	\mathbb{P}\Big(\sup_{r\in[0,T]}\|\rho^{1}(\cdot,r)-\rho^{2}(\cdot,r)\|_{L^{1}(\mathbb{T}^{d})}>\delta\Big)=0.
	\end{align}
	In fact, it is worth noting that
	\begin{align*}
	\Big\{\omega:\sup_{r\in[0,T]}\|\rho^{1}(\cdot,r)-\rho^{2}(\cdot,r)\|_{L^{1}(\mathbb{T}^{d})}>\delta\Big\}\subset&\Big\{\omega: \sup_{r\in[0,T\wedge\tau_{N}]}\big\|\rho^{1}(\cdot,r)-\rho^{2}(\cdot,r)\big\|_{L^{1}(\mathbb{T}^{d})}>\delta,  \tau_{N}>T\Big\}\\
	&\cup\Big\{\omega: \tau_{N}\le T\Big\}.
	\end{align*}
	 Based on Chebyshev's inequality, (\ref{kk-52}), the property $\hat{\rho}^1=\hat{\rho}^2\ \text{a.e. in }\mathbb{T}^d$, and \eqref{eq-4.31}, we can get (\ref{kk-53}). Since  $\delta$ is arbitrary, it follows that
	\begin{equation*}	\mathbb{P}\Big(\sup_{r\in[0,T]}\|\rho^1(\cdot,r)-\rho^2(\cdot,r)\|_{L^1(\mathbb{T}^d)}=0\Big)=1.
	\end{equation*}
\end{proof}
\begin{remark}\label{rem-4.6}
	We point out that the LPS condition (i.e., Assumption (A1) on $V$) is sufficient to guarantee the integrability of the kernel term $\int_0^t\int_{\mathbb{T}^d}\varphi(x,\rho)\nabla\cdot(\rho V(r)\ast\rho)$ in (\ref{eq-2.6}) and the forthcoming entropy dissipation estimates (see Proposition \ref{pro-5.4} below). However, it is not strong enough to admit the uniqueness. In fact, if we want to avoid imposing any conditions on $\nabla\cdot V$, it requires to handle $\int_{0}^{t}\int_{\mathbb{T}^d}|\rho^2V(r)\ast\nabla(\rho^1-\rho^2)|$ instead of $\int_{0}^{t}\int_{\mathbb{T}^d}\left|\rho^2(\nabla\cdot V(r))\ast(\rho^1-\rho^2)\right|$ in (\ref{eq-4.28}). In this case, it is difficult to control $\nabla(\rho^1-\rho^2)$  by $\|\rho^1-\rho^2\|_{{L^1(\mathbb{T}^d)}}$, which results in the inapplicability of the stochastic Gronwall's inequality. Thus, for technical reasons, we need an additional condition Assumption (A2).
\end{remark}

As a consequence of (\ref{kk-52}) and Chebyshev's inequality, we get the continuity of solutions with respect to the initial data.
\begin{lemma}\label{lem-4.7}
	Let $\{\hat{\rho}^n,\hat{\rho}\}_{n\geq1}\subset\text{{\rm{Ent}}}\left(\mathbb{T}^{d}\right)$ satisfy
	$\lim\limits_{n\to\infty}\|\hat{\rho}^n-\hat{\rho}\|_{L^1(\mathbb{T}^d)}=0$.
	Let $ \rho_n,\rho$ be renormalized kinetic solutions of (\ref{eq-2.2}) in the sense of Definition \ref{def-2.4} with initial values $\rho_n(\cdot ,0)=\hat{\rho}^n,\rho(\cdot, 0)=\hat{\rho}$, respectively. Then for any $\delta>0$, we have
	\begin{equation*}	\lim_{n\to\infty}\mathbb{P}\Big(\sup_{t\in[0,T]}\|\rho_n(t)-\rho(t)\|_{L^1(\mathbb{T}^d)}>\delta\Big)=0.
	\end{equation*}
\end{lemma}

\section{Approximation equation}\label{sec-5}
To demonstrate the existence of renormalized kinetic solutions to (\ref{eq-2.2}), we introduce an approximation equation with regularized coefficients. Specifically, we will introduce smooth sequences that approximate the square root function $\sqrt{\cdot}$ and the kernel $V$, respectively.

According to  \cite[Lemma 5.18]{FG21}, the choice of the approximation of the square root function is stated as follows.
\begin{lemma}\label{rem-5.1}
There exists a sequence $\left\{\sigma_{n}\right\}_{n\in\mathbb{N}}$ that fulfills $\sigma_{n}(\cdot)\to\sqrt{\cdot}$ in $\mathrm{C}_{\text{loc}}^{1}((0,\infty))$ as $n\to\infty$. Furthermore, $\sigma_{n}$ has the following properties.
	 \begin{enumerate}
	 \item[(1)] $\sigma_{n}\in C(\left[0,\infty\right))\cap C^{\infty}((0,\infty))$ with $\sigma_{n}(0)=0$ and $\sigma_{n}'\in C^{\infty}_c(\left[0,\infty\right))$ for every $n\in\mathbb{N}$,
	 \item[(2)] there exists $c\in(0,\infty)$ such that for every $\xi\in[0,\infty)$,
	 \begin{align}\label{kk-5.2}
	 	|\sigma_n(\xi)|\le c\sqrt\xi \text{ uniformly with respect to } n\in\mathbb{N},
	 \end{align}
	 \item[(3)] for every $\delta\in(0,1)$, there exists $c_{\delta}\in(0,\infty)$ such that
	 \begin{align}\label{eq-5.2}		[\sigma_n'(\xi)]^4\mathbf{1}_{\{\xi\ge\delta\}}+|\sigma_n(\xi)\sigma_n'(\xi)|^2\mathbf{1}_{\{\xi\ge\delta\}}\le c_{\delta} \text{ uniformly with respect to } n\in\mathbb{N}.
	\end{align}
	\end{enumerate}
	
\end{lemma}

For the kernel $V$,  let $V_{\gamma}:=((V\wedge(1/\gamma))\vee(-1/\gamma))\ast\eta_{\gamma}$, for every $\gamma>0$ and $(t,x)\in\mathbb{R}\times\mathbb{T}^d$, where $\ast$ denotes spatial convolution, and $\eta_{\gamma}(x)=\frac{1}{\gamma^{d}}\eta(\frac{x}{\gamma})$ is the standard convolution kernel on $\mathbb{T}^d$. Suppose that $V$ satisfies Assumption (A1), then $V_{\gamma}$ satisfies  the following properties.

\begin{lemma}\label{kk-66}
	Let $V$ satisfy Assumption (A1). For $V_{\gamma}$ as defined above, we have
	\begin{enumerate}
		\item[(1)]  for every $\gamma>0$, $V_{\gamma}$ satisfies  Assumption (A1),
		\item[(2)] for every $\gamma>0$, $V_{\gamma}\in L^{\infty}([0,T];L^{\infty}(\mathbb{T}^d;\mathbb{R}^d))$,
		\item[(3)] $V_{\gamma}\to V$ in $L^{p^{*}}\left([0,T];L^p(\mathbb{T}^d;\mathbb{R}^d)\right)$ with  $\frac{d}{p}+\frac{2}{p^{*}}\le1$, $2\le p^*\leq\infty$ and $d<p\le\infty$ as $\gamma\to0$.
	\end{enumerate}
\end{lemma}
\begin{proof}
With the aid of convolutional Young's inequality and dominated convergence theorem, the proof can be easily achieved.
\end{proof}

With the help of $\sigma_{n}$ and $V_{\gamma}$,
we consider the following regularized stochastic PDE
\begin{align}\notag
	\mathrm{d}\rho^{n,\gamma}=&\Delta\rho^{n,\gamma}\mathrm{d}t-\nabla\cdot
\left(\sigma_{n}\left(\rho^{n,\gamma}\right)\mathrm{d}W^F\right)
-\nabla\cdot(\rho^{n,\gamma}V_{\gamma}(t)\ast\rho^{n,\gamma})\mathrm{d}t\\
\label{eq-5.1}
 &+\frac{1}{2}\nabla\cdot\left(F_1\left[\sigma_{n}'\left(\rho^{n,\gamma}\right)\right]^{2}\nabla\rho^{n,\gamma}+\sigma_{n}(\rho^{n,\gamma})\sigma_{n}'(\rho^{n,\gamma}) F_{2}\right)\mathrm{d}t,
\end{align}
with $\rho^{n,\gamma}(\cdot,0)=\hat{\rho}$.


Now, we introduce the definition of weak solutions to \eqref{eq-5.1}.
\begin{definition}\label{def-5.2}
 Let $V$ satisfy Assumption (A1) and $\hat{\rho}\in L^{\infty}\left(\mathbb{T}^{d}\right)$ be nonnegative. For any  $\gamma>0$ and $n\in\mathbb{N}$, a weak solution of (\ref{eq-5.1}) with initial data $\rho^{n,\gamma}(\cdot,0)=\hat{\rho}$ is a nonnegative, $\mathcal{F}_t$-predictable, $L^m(\mathbb{T}^d)$-continuous (for some $m\geq2$) process $\rho^{n,\gamma}$ such that almost surely $\rho^{n,\gamma}\in L^2([0,T];H^1(\mathbb{T}^d))$ and for every $\psi\in C^{\infty}(\mathbb{T}^d)$, almost surely for every $t\in[0,T]$,
	\begin{align}\label{eq-5.3}
	&\int_{\mathbb{T}^d}\rho^{n,\gamma}(x,t)\psi(x)=\int_{\mathbb{T}^d}\hat{\rho}\psi-\int_0^t\int_{\mathbb{T}^d}\nabla\rho^{n,\gamma}\cdot\nabla\psi+\int_0^t\int_{\mathbb{T}^d}(\rho^{n,\gamma} V_{\gamma}(r)\ast\rho^{n,\gamma})\cdot\nabla\psi\notag\\
	&+\int_0^t\int_{\mathbb{T}^d}\sigma_n(\rho^{n,\gamma})\nabla\psi\cdot \mathrm{d}W^F-\frac{1}{2}\int_0^t\int_{\mathbb{T}^d}F_1[\sigma_n'(\rho^{n,\gamma})]^2\nabla\rho^{n,\gamma}\cdot\nabla\psi-\frac{1}{2}\int_0^t\int_{\mathbb{T}^d}\sigma_n(\rho^{n,\gamma})\sigma_n'(\rho^{n,\gamma})F_2\cdot\nabla\psi.
	\end{align}
\end{definition}

For any nonnegative function $\hat{\rho}\in L^{\infty}\left(\mathbb{T}^{d}\right)$, let $\rho^{n,\gamma}$ be a weak solution of (\ref{eq-5.1}) in the sense of Definition \ref{def-5.2} with initial data $\rho^{n,\gamma}(\cdot,0)=\hat{\rho}$.
 By choosing $\psi=1$ in \eqref{eq-5.3} and using the nonnegativity of $\rho^{n,\gamma}$, we deduce that almost surely for every $t\in[0,T]$,
	\begin{align}\label{eq-5.4}
	\| \rho^{n,\gamma}(\cdot,t)\|_{L^{1}(\mathbb{T}^d)}=\|\hat{\rho}\|_{L^{1}(\mathbb{T}^d)}.
	\end{align}
Moreover, we can show that the weak solution $\rho^{n,\gamma}$ of (\ref{eq-5.1}) satisfies
 the following entropy estimate and $L^m(\mathbb{T}^d)-$norm estimate.  As mentioned in the introduction, this $L^m(\mathbb{T}^d)-$estimate is not  uniform with respect to the regularized parameter $\gamma$, which is different from \cite{FG21}.

\subsection{Entropy estimate and $L^m(\mathbb{T}^d)-$norm estimate }\label{subsec-5.1}

%

Let $\Psi:[0,\infty)\to\mathbb{R}$ be the unique function satisfying $\Psi(0)=0$ with $\Psi'(\xi)=\log(\xi)$. Recall that $\text{{\rm{Ent}}}(\mathbb{T}^{d})$ is defined by \eqref{kk-44}.
\begin{proposition}\label{pro-5.4}
	Suppose that $V$ satisfies Assumption (A1). Let $\hat{\rho}\in \text{{\rm{Ent}}}\left(\mathbb{T}^{d}\right)$. For any  $\gamma>0$ and $n\in\mathbb{N}$, let $\rho^{n,\gamma}$ be a weak solution of (\ref{eq-5.1}) in the sense of Definition \ref{def-5.2} with initial data $\rho^{n,\gamma}(\cdot,0)=\hat{\rho}$. Then there exists a constant $c\in(0,\infty)$ depending on $T, d,\|\hat{\rho}\|_{L^1(\mathbb{T}^d)}$ and $\|V\|_{L^{p*}([0,T];L^p(\mathbb{T}^d;\mathbb{R}^d))}$ such that
	\begin{align}\label{eq-5.5}
	&\mathbb{E}\Big[\sup_{t\in[0,T]}\int_{\mathbb{T}^{d}}\Psi(\rho^{n,\gamma}(x,t))\Big]
+\mathbb{E}\Big[\int_{0}^{T}\int_{\mathbb{T}^{d}}|\nabla \sqrt{\rho^{n,\gamma}}|^{2}\Big]\le \int_{\mathbb{T}^{d}}\Psi(\hat{\rho})+c.
	\end{align}
\end{proposition}
\begin{proof}
For the above $\Psi$, we firstly introduce a sequence of smooth approximating functions denoted by $\Psi_{\delta}$ with $\delta\in(0,1)$. Here, we require that $\Psi_{\delta}(0)=0$  and $\Psi_{\delta}'(\xi)=\log (\xi+\delta)$. Applying It\^{o}'s formula \cite{Krylov} and by the nonnegativity of $\rho^{n,\gamma}$, we deduce that almost surely for every $t\in[0,T]$,
	\begin{align}\label{eq-5.6}
	\left.\int_{\mathbb{T}^{d}}\Psi_{\delta}(\rho^{n,\gamma}(x,r))\right|_{r=0}^{r=t}
+\int_{0}^{t}\int_{\mathbb{T}^{d}}\frac{4\rho^{n,\gamma}}{\rho^{n,\gamma}+\delta}\left|\nabla \sqrt{\rho^{n,\gamma}}\right|^{2}=K^{\text{ker}}_{t}+K^{\text{mart}}_{t}+K^{\text{err}}_{t},
	\end{align}
	where
	\begin{align*} K^{\text{ker}}_{t}&=\int_{0}^{t}\int_{\mathbb{T}^{d}}\frac{1}{\rho^{n,\gamma}+\delta}(\rho^{n,\gamma} V_{\gamma}(r)\ast\rho^{n,\gamma})\cdot\nabla\rho^{n,\gamma},\\ K^{\text{mart}}_{t}&=\int_{0}^{t}\int_{\mathbb{T}^{d}}\frac{2\sqrt{\rho^{n,\gamma}}\sigma_n(\rho^{n,\gamma})}{\rho^{n,\gamma}+\delta}\nabla\sqrt{\rho^{n,\gamma}}\cdot\mathrm{d}W^F,\\ K^{\text{err}}_{t}&=\frac{1}{2}\int_{0}^{t}\int_{\mathbb{T}^{d}}\frac{1}{\rho^{n,\gamma}+\delta}\left(\sigma_n(\rho^{n,\gamma})\sigma_n'(\rho^{n,\gamma})F_{2}\cdot\nabla\rho^{n,\gamma}+\sigma_n^2(\rho^{n,\gamma}) F_{3}\right).
	\end{align*}
	We firstly proceed with the term $K^{\text{err}}_{t}$. Thanks to the nonnegativity of $\rho^{n,\gamma}$, \eqref{kk-5.2} and the boundedness of $F_{3}$, by using $\nabla\cdot F_2=0$, combining with the assumption of $\sigma_n$, it follows that there exists a constant $c\in(0,\infty)$ independent of $\delta$ such that almost surely
	\begin{align}\label{eq-5.9}
		\sup_{t\in[0,T]}K^{\text{err}}_{t}\le c(T).
	\end{align}

	Regarding to the term $K^{\text{mart}}_{t}$, it follows from \eqref{kk-5.2}, the Burkholder-Davis-Gundy inequality and H\"{o}lder's inequality that there exists $c\in(0,\infty)$ depending on $F_1$ such that
	\begin{align}
\mathbb{E}\sup_{t\in[0,T]}K^{\text{mart}}_{t}\le&\mathbb{E}\Big[\int_{0}^{T}\int_{\mathbb{T}^{d}}\frac{\rho^{n,\gamma}}{\rho^{n,\gamma}+\delta}\Big|\nabla\sqrt{\rho^{n,\gamma}}\Big|^{2}\Big]+4c^2, \label{eq-5.10}
	\end{align}
	where $c$ is independent of $\delta\in(0,1)$. It remains to consider the term $K^{\text{ker}}_{t}$. By Young's inequality and (\ref{eq-5.4}), we have almost surely that
\begin{align}\label{eq-5.11}
\sup_{t\in[0,T]}K^{\text{ker}}_{t}\leq&\int^T_0\int_{\mathbb{T}^d}\frac{\rho^{n,\gamma}}{\rho^{n,\gamma}+\delta}|\nabla\sqrt{\rho^{n,\gamma}}|^2+\int^T_0\int_{\mathbb{T}^d}\rho^{n,\gamma} (V_{\gamma}\ast\rho^{n,\gamma})^2\notag\\
\leq&\int^T_0\int_{\mathbb{T}^d}\frac{\rho^{n,\gamma}}{\rho^{n,\gamma}+\delta}|\nabla\sqrt{\rho^{n,\gamma}}|^2+C(\|\hat{\rho}\|_{L^1(\mathbb{T}^d)})\int^T_0\|V_{\gamma}\|_{L^{\infty}(\mathbb{T}^d;\mathbb{R}^d)}^2.
\end{align}
Based on (\ref{eq-5.6})-\eqref{eq-5.11} and Lemma \ref{kk-66}, there exists a constant $c\in(0,\infty)$ independent of $\delta\in(0,1)$ such that
	\begin{align}\label{eq-5.12} \mathbb{E}\Big[\sup_{t\in[0,T]}\int_{\mathbb{T}^{d}}\Psi_{\delta}(\rho^{n,\gamma}(x,t))\Big]
+\mathbb{E}\Big[\int_{0}^{T}\int_{\mathbb{T}^{d}}\frac{2\rho^{n,\gamma}}{\rho^{n,\gamma}+\delta}\Big|\nabla\sqrt{\rho^{n,\gamma}}\Big|^{2}\Big]
	\le\int_{\mathbb{T}^{d}} \Psi_{\delta}(\hat{\rho})+c(T, \|\hat{\rho}\|_{L^1(\mathbb{T}^d)},\gamma).
	\end{align}
By the definition of $\Psi_{\delta}$ and (\ref{eq-5.12}), we deduce that there exists $\delta_0>0$ such that
\begin{align}\label{eq-5.13}
\sup_{0<\delta\leq \delta_0}\mathbb{E}\int^T_0\Big\|\sqrt{\frac{2\rho^{n,\gamma}}{\rho^{n,\gamma}+\delta}}\nabla\sqrt{\rho^{n,\gamma}}\Big\|_{L^2(\mathbb{T}^d)}^2
\le\int_{\mathbb{T}^{d}} \Psi_{\delta_0}\left(\hat{\rho}\right)+c(T, \|\hat{\rho}\|_{L^1(\mathbb{T}^d)},\gamma).
\end{align}
Then, there exists a subsequence $\{\delta_k\}_{k\geq1}\subset \{\delta\}_{0<\delta\leq \delta_0}$ and $f^*\in L^2(\Omega\times[0,T];L^2(\mathbb{T}^d))$ such that $\sqrt{\frac{2\rho^{n,\gamma}}{\rho^{n,\gamma}+\delta_k}}\nabla\sqrt{\rho^{n,\gamma}}\rightarrow f^*$ weakly in $L^2(\Omega\times[0,T];L^2(\mathbb{T}^d))$ as $k\rightarrow\infty$.

On the other hand,
for almost every $t\in [0,T]$, and $\varphi\in C^{\infty}(\mathbb{T}^d; \mathbb{R}^d)$, we have
\begin{align*}
&\Big\langle\sqrt{\frac{2\rho^{n,\gamma}}{\rho^{n,\gamma}+\delta}}\nabla\sqrt{\rho^{n,\gamma}}-\sqrt{2}\nabla\sqrt{\rho^{n,\gamma}},\varphi\Big\rangle
=\sqrt{2}\Big\langle \nabla \int^{\sqrt{\rho^{n,\gamma}}}_0\Big(\frac{\xi}{\sqrt{\xi^2+\delta}}-1\Big)\mathrm{d}\xi ,\varphi\Big\rangle\\
=&-\sqrt{2}\Big\langle \int^{\sqrt{\rho^{n,\gamma}}}_0\Big(\frac{\xi}{\sqrt{\xi^2+\delta}}-1\Big)\mathrm{d}\xi ,\nabla\cdot\varphi\Big\rangle
= -\sqrt{2}\Big\langle\Big[\sqrt{\rho^{n,\gamma}+\delta}-\sqrt{\rho^{n,\gamma}}-\sqrt{\delta}\Big],\nabla\cdot\varphi\Big\rangle,
\end{align*}
then, by (\ref{eq-5.4}) and the dominated convergence theorem, it gives
\begin{align}\label{kk-58}
\lim_{\delta\rightarrow 0}\mathbb{E} \int^T_0 \Big\langle\Big(\sqrt{\frac{2\rho^{n,\gamma}}{\rho^{n,\gamma}+\delta}}\nabla\sqrt{\rho^{n,\gamma}}
-\sqrt{2}\nabla\sqrt{\rho^{n,\gamma}}\Big),\varphi \Big\rangle \mathrm{d}r=0.
\end{align}
As a result of (\ref{kk-58}), for any $\varphi\in C^{\infty}(\mathbb{T}^d;\mathbb{R}^d)$, we deduce that
\begin{align}\label{kk-64}
\mathbb{E}\int^T_0\langle\sqrt{2}\nabla\sqrt{\rho^{n,\gamma}},\varphi\rangle \mathrm{d}r=\lim_{k\rightarrow\infty}\mathbb{E}\int^T_0\Big\langle\sqrt{\frac{2\rho^{n,\gamma}}{\rho^{n,\gamma}+\delta_k}}\nabla\sqrt{\rho^{n,\gamma}},\varphi\Big\rangle \mathrm{d}r=\mathbb{E}\int^T_0\langle f^*,\varphi\rangle \mathrm{d}r,
\end{align}
which implies $f^*=\sqrt{2}\nabla\sqrt{\rho^{n,\gamma}}$  almost surely  for almost every $(t,x)\in [0,T]\times\mathbb{T}^d$. Thus, $\sqrt{\frac{2\rho^{n,\gamma}}{\rho^{n,\gamma}+\delta}}\nabla\sqrt{\rho^{n,\gamma}}\rightarrow\sqrt{2}\nabla\sqrt{\rho^{n,\gamma}}$ weakly in $L^2(\Omega\times[0,T];L^2(\mathbb{T}^d))$, as $\delta\rightarrow0$. By the lower semi-continuity of $L^2(\Omega\times[0,T];L^2(\mathbb{T}^d))$-norm, we deduce from (\ref{eq-5.13}) that
\begin{align}\label{kk-59}
\mathbb{E}\int^T_0\|\nabla\sqrt{\rho^{n,\gamma}}\|_{L^2(\mathbb{T}^d)}^2\le\int_{\mathbb{T}^{d}} \Psi_{\delta}\left(\hat{\rho}\right)+c(T, \|\hat{\rho}\|_{L^1(\mathbb{T}^d)},\gamma).
\end{align}
With the aid of (\ref{kk-59}), the kernel term can be reestimated as follows.
	\begin{align}\notag
	\sup_{t\in[0,T]}K^{\text{ker}}_{t}\le&\int_{0}^{T}\|\nabla\rho^{n,\gamma}\cdot V_{\gamma}(r)\ast\rho^{n,\gamma}\|_{L^1(\mathbb{T}^d)}\\ \notag
	\le&\int_{0}^{T}\|V_{\gamma}(r)\|_{L^{p}(\mathbb{T}^d;\mathbb{R}^d)}\|\hat{\rho}\|^{\frac32-\frac{d}{2p}}_{L^{1}(\mathbb{T}^{d})}\|\nabla\sqrt{\rho^{n,\gamma}}\|^{\frac{d}{p}+1}_{L^{2}(\mathbb{T}^{d})}\\
\label{eq-5.14}
	\le& \int_{0}^{T}\int_{\mathbb{T}^d}|\nabla\sqrt{\rho^{n,\gamma}}|^2+c(\|\hat{\rho}\|_{L^1(\mathbb{T}^d)},\|V\|_{L^{p*}([0,T];L^p(\mathbb{T}^d;\mathbb{R}^d))},d,T ),	\end{align}
 where the nonnegativity of $\rho^{n,\gamma}$, \eqref{eq-3.3}, \eqref{eq-3.1} and Lemma \ref{kk-66} are used.
	Combining (\ref{eq-5.6}), \eqref{eq-5.9}, \eqref{eq-5.10} and \eqref{eq-5.14}, we deduce that
	\begin{align}\label{eq-5.15}
	&\mathbb{E}\left[\sup_{t\in[0,T]}\int_{\mathbb{T}^{d}}\Psi_{\delta}(\rho^{n,\gamma}(x,t))\right]
+3\mathbb{E}\left[\int_{0}^{T}\int_{\mathbb{T}^{d}}\frac{\rho^{n,\gamma}}{\rho^{n,\gamma}+\delta}\left|\nabla\sqrt{\rho^{n,\gamma}}\right|^{2}\right]\notag \\
	&\le\int_{\mathbb{T}^{d}} \Psi_{\delta}\left(\hat{\rho}\right)+\int_{0}^{T}\int_{\mathbb{T}^d}|\nabla\sqrt{\rho^{n,\gamma}}|^2+c(T, d,\|\hat{\rho}\|_{L^1(\mathbb{T}^d)},\|V\|_{L^{p*}([0,T];L^p(\mathbb{T}^d;\mathbb{R}^d))}).
	\end{align}
Since $\{\rho^{n,\gamma}=0\}=\left\{\sqrt{\rho^{n,\gamma}}=0\right\}$, by (\ref{kk-59}) and Stampacchia's lemma (see Evans \cite[Chapter 5, Exercise 18]{Eva10} ), we deduce that almost surely
	\begin{align*}
\int_{0}^{T}\int_{\mathbb{T}^{d}}\mathbf{1}_{\{\rho^{n,\gamma}=0\}}\left|\nabla\sqrt{\rho^{n,\gamma}}\right|^{2}=0 .
	\end{align*}
Thus, with the aid of Fatou's lemma, passing to the limit $\delta\to 0$ in (\ref{eq-5.15}), for some $c\in(0,\infty)$,
	\begin{align*}
	&\mathbb{E}\Big[\sup_{t\in[0,T]}\int_{\mathbb{T}^{d}}\Psi(\rho^{n,\gamma}(x,t))\Big]+\mathbb{E}\Big[\int_{0}^{T}\int_{\mathbb{T}^{d}}\Big|\nabla \sqrt{\rho^{n,\gamma}}\Big|^{2}\Big]\le \int_{\mathbb{T}^{d}}\Psi(\hat{\rho})+c(T, d,\|\hat{\rho}\|_{L^1(\mathbb{T}^d)},\|V\|_{L^{p*}([0,T];L^p(\mathbb{T}^d;\mathbb{R}^d))}).
	\end{align*}
\end{proof}


Regarding to the $L^m(\mathbb{T}^d)-$norm estimate, we only need to make estimates of the extra kernel term since the rest of the terms can directly refer to the results of {\color{black}\cite[Proposition 5.7]{FG21}}. For the additional kernel term, it can be estimated by using H\"{o}lder's and convolutional Young's inequalities,  \eqref{eq-5.4} and Gronwall lemma. For the sake of simplicity, we omit the proof.
\begin{proposition}\label{pro-6.2}
	Suppose that $V$ satisfies Assumption (A1). Let $\hat{\rho}\in L^{\infty}(\mathbb{T}^d)$ be a nonnegative function. For any  $\gamma>0$ and $n\in\mathbb{N}$, let $\rho^{n,\gamma}$ be a solution of \eqref{eq-5.1} in the sense of Definition \ref{def-5.2} with initial data $\rho^{n,\gamma}(\cdot,0)=\hat{\rho}$. Then there exists a constant $\lambda\in(0,\infty)$ depending on $m, T,  \gamma,\|\hat{\rho}\|_{L^m(\mathbb{T}^d)}$, and $\|\hat{\rho}\|_{L^1(\mathbb{T}^d)}$ such that
	\begin{align}\label{eq-6.1}	&\sup_{t\in[0,T]}\mathbb{E}\left[\int_{\mathbb{T}^d}(\rho^{n,\gamma})^m(x,t)\right]+\mathbb{E}\left[\int^T_0\int_{\mathbb{T}^d}|\rho^{n,\gamma}|^{(m-2)}|\nabla\rho^{n,\gamma}|^2\right]
\le\lambda.
	\end{align}
\end{proposition}

\begin{remark}
 Note that Lemma \ref{kk-66} implies that the entropy estimate holds uniformly with respect to the regularized parameters $n$ and $\gamma$.
However, from Lemma \ref{kk-66}, it follows that the $L^m(\mathbb{T}^d)-$norm estimate is not uniform with respect to $\gamma$, which results in the inapplicability of the method used by [FG24] to find a limiting kinetic measure.
\end{remark}

\subsection{Existence of renormalized kinetic solutions to the approximation equation}\label{subsec-5.2}
We will firstly show the existence of weak solutions of approximation equation \eqref{eq-5.1}, which are also weak in the probabilistic sense. Then, we introduce the definition of the renormalized kinetic solution to \eqref{eq-5.1}. Finally, we prove that the weak solutions and renormalized kinetic solutions are equivalent.

\begin{theorem}\label{the-6.4}
Suppose that $V$ satisfies Assumption (A1). Let $\hat{\rho}\in L^{\infty}(\mathbb{T}^d)$ be a nonnegative function. Then for any  $\gamma>0$ and $n\in\mathbb{N}$, there exists a stochastic basis $(
\tilde{\Omega},\tilde{\mathcal{F}},\{\tilde{\mathcal{F}}(t)\}_{t\in[0,T]},\tilde{\mathbb{P}})$, a Brownian motion $\tilde{W}^F$, and a process $\tilde{\rho}^{n,\gamma}$, which is a weak solution of \eqref{eq-5.1} in the sense of Definition \ref{def-5.2} with initial data $\tilde{\rho}^{n,\gamma}(\cdot,0)=\hat{\rho}$. Furthermore, $\tilde{\rho}^{n,\gamma}$ satisfies the estimate (\ref{eq-6.1}).
\end{theorem}
\begin{proof}
	 Let $\left\{e_{k}\right\}_{k\in\mathbb{N}}\subset H^{1}\left(\mathbb{T}^{d}\right)$ be an orthonormal basis of  $L^{2}\left(\mathbb{T}^{d}\right)$. Let {\color{black}$0\leq\lambda_1<\lambda_2 <...<\infty$} be the corresponding eigenvalues such that $(-\Delta)e_k=\lambda_ke_k$ for $k\geq 1$.
For every $K\in\mathbb{N}$, define a finite dimensional noise $ W^{F,K}(x,t):=\sum_{k=1}^{K}f_{k}(x)B_{t}^{k}$. For every $M\in\mathbb{N}$, let $\Pi_{M}:L^{2}\left(\mathbb{T}^{d}\times[0,T]\right) \to L^{2}\left(\mathbb{T}^{d}\times[0,T]\right)$ be the projection map defined by
\begin{align*}
  \Pi_{M}g(x,t)=\sum_{k=1}^{M}( g(t),e_{k}) e_{k}(x) \ {\rm{for\ every}}\ g\in L^{2}\left(\mathbb{T}^{d}\times[0,T]\right).
\end{align*}
For any $m>0$, define a smooth  function $S^m:\mathbb{R}\to\mathbb{R}$ satisfying $-m-1\le S^m\le m+1$ and
\begin{align}
	S^m(\xi)=
	\left\{
	\begin{array}{ll}
		\xi, & {\rm{if}}\ -m\le\xi\le m, \\
		m+1, &{\rm{if}}\ \xi\ge m+1,\\
		-m-1, &{\rm{if}}\ \xi\le -m-1,\\
		{\rm{smooth}},  &{\rm{otherwise.}}
	\end{array}
	\right.
\end{align}

We consider the following finite dimensional projected equation
	\begin{align}\label{eq-6.8}
	\mathrm{d}\rho^{n,\gamma,m}_M=&\Pi_{M}\left(\Delta\rho^{n,\gamma,m}_M\mathrm{d}t-\nabla\cdot\left(\sigma_n(\rho^{n,\gamma,m}_M)\mathrm{d} W^{F,K}\right)-\nabla\cdot(S^m(\rho^{n,\gamma,m}_M)V_{\gamma}(t)\ast\rho^{n,\gamma,m}_M)\mathrm{d}t\right)\notag\\
	&+\Pi_{M}\left(\frac{1}{2}\nabla\cdot\left(F_{1}^{K}\left[\sigma_n^{\prime}(\rho^{n,\gamma,m}_M)\right]^{2}\nabla\rho^{n,\gamma,m}_M+\sigma_n(\rho^{n,\gamma,m}_M)\sigma_n^{\prime}(\rho^{n,\gamma,m}_M)F_{2}^{K}\right)\mathrm{d}t\right),
	\end{align}
	where $F_{1}^{K}=\sum_{k=1}^{K}f_{k}^{2}$, $F_{2}^{K}=\sum_{k=1}^{K}f_{k}\nabla f_{k}$.

   Since $\sigma_n$ and $S^m$ are both smooth, bounded and the derivative of $\sigma_n$ is also bounded, equation \eqref{eq-6.8} has a unique probabilistically strong solution $\rho^{n,\gamma,m}_M$ on the time interval $[0,T]$. Through a simple estimation, it can be inferred that the sequence $\{\rho^{n,\gamma,m}_M\}_{M\geq 1}$ is bounded in $L^{\infty}([0,T];L^2(\mathbb{T}^d))\cap L^2([0,T];H^1(\mathbb{T}^d))$.
By a standard tightness argument,
it follows from Prokhorov's theorem and the Skorohod representation theorem that there exists a stochastic basis $(\tilde{\Omega},\tilde{\mathcal{F}},\{\tilde{\mathcal{F}}(t)\}_{t\in[0,T]},\tilde{\mathbb{P}},\tilde{W}^{F,K})$ with expectation $\tilde{\mathbb{E}}$, random variables $\{\tilde{\rho}^{n,\gamma,m}_M\}_{M\geq 1}$ and $\tilde{\rho}^{n,\gamma,m}\in L^{2}(\tilde{\Omega};L^2([0,T];L^{2}(\mathbb{T}^{d})))$ such that
 $\tilde{\rho}^{n,\gamma,m}_M$ has the same law as $\rho^{n,\gamma,m}_M$ and as $M\rightarrow+\infty$,
  \begin{align}\label{qq-r-6}
    \|\tilde{\rho}^{n,\gamma,m}_M-\tilde{\rho}^{n,\gamma,m}\|_{L^2([0,T];L^2(\mathbb{T}^d))}\rightarrow0,\quad \tilde{\mathbb{P}}-a.s.
  \end{align}
  Following the proof idea of \cite[Proposition 5.17]{FG21} and applying It\^{o} formula, we can conclude that $\tilde{\rho}^{n,\gamma,m}$ is nonnegative.
  It implies 
  the $L^1$-norm conservation property. Applying It\^{o} formula, along with the property of the smooth function $S^m$, following a proof analogous to that of Proposition \ref{pro-6.2}, we deduce that $\{\tilde{\rho}^{n,\gamma,m}\}_{m\geq 0}$ is uniformly bounded in $L^{\infty}([0,T];L^2(\mathbb{T}^d))\cap L^2([0,T];H^1(\mathbb{T}^d))$, with respect to the parameters $m$ and $K$.  
	Repeating a process similar to the above,  using a standard tightness argument along with Prokhorov's theorem and the Skorohod representation theorem,  there exists a stochastic basis $(\hat{\Omega},\hat{\mathcal{F}},\{\hat{\mathcal{F}}(t)\}_{t\in[0,T]},\hat{\mathbb{P}},\hat{W}^F)$ with expectation $\hat{\mathbb{E}}$, random variables $\{\hat{\rho}^{n,\gamma,m}\}_{m\geq 0}$ and $\hat{\rho}^{n,\gamma}\in L^{2}(\hat{\Omega};L^2([0,T];L^{2}(\mathbb{T}^{d})))$ such that
	$\hat{\rho}^{n,\gamma,m}$ has the same law as $\tilde{\rho}^{n,\gamma,m}$ and as $K\rightarrow+\infty$, $m\rightarrow+\infty$,
	\begin{align*}
		\|\hat{\rho}^{n,\gamma,m}-\hat{\rho}^{n,\gamma}\|_{L^2([0,T];L^2(\mathbb{T}^d))}\rightarrow0,\quad \hat{\mathbb{P}}-a.s.
	\end{align*}
	Moreover,  $\hat{\rho}^{n,\gamma}$ satisfies \eqref{eq-5.3} with respect to the new stochastic basis $(\hat{\Omega},\hat{\mathcal{F}},\{\hat{\mathcal{F}}(t)\}_{t\in[0,T]},\hat{\mathbb{P}},\hat{W}^F)$. {\color{black}The method employed here is standard (see, for instance, \cite[Proposition 5.4]{DG20}), thus we omit the details.} 
\end{proof}
\begin{remark}
The result of Theorem \ref{the-6.4} is the starting point to prove the existence of the probabilistically strong solution to (\ref{eq-1.1}). Concretely, we will apply Lemma \ref{lem-8.1} to the probabilistically weak solution constructed in Theorem \ref{the-6.4} to find a solution living in the original probability space. Since the conditions of Lemma \ref{lem-8.1} are in the sense of distribution, it is not necessary to emphasize the difference between the original probability space and the new one. Thus, with a little abuse of notations, the probability space, the Brownian motion and the weak solution in Theorem \ref{the-6.4} are still denoted by $(
\Omega,\mathcal{F},\{\mathcal{F}(t)\}_{t\in[0,T]},\mathbb{P})$, $W^F$ and $\rho^{n,\gamma}$, respectively.
\end{remark}


Now, we introduce the other definition of solutions to (\ref{eq-5.1}), which is called a renormalized kinetic solution.
\begin{definition}\label{def-6.3}
Suppose that $V$ satisfies Assumption (A1). Let $\hat{\rho}\in L^{\infty}(\mathbb{T}^d)$ be a nonnegative function. For any $\gamma>0$ and $n\in\mathbb{N}$, a renormalized kinetic solution of (\ref{eq-5.1}) with initial data $\rho^{n,\gamma}(\cdot,0)=\hat{\rho}$ is a nonnegative, $L^m(\mathbb{T}^d)$-continuous (for some $m\geq2$) , $\mathcal{F}_t$-predictable process $\rho^{n,\gamma}$ such that almost surely $\rho^{n,\gamma}\in L^2([0,T];H^1(\mathbb{T}^d))$ and almost surely for every $\psi\in\mathrm{C}_{c}^{\infty}\left(\mathbb{T}^{d}\times{\color{black}(0,\infty)}\right)$ and $t\in[0,T]$,
\begin{align}\label{eq-6.6}
&\int_{\mathbb{T}^{d}}{\color{black}\int_0^{\infty}}\chi^{n,\gamma}(x,\xi,t)\psi(x,\xi)=\int_{\mathbb{T}^{d}}{\color{black}\int_0^{\infty}}\bar{\chi}^{n,\gamma}\left(\hat{\rho}\right)\psi(x,\xi)-\int_{0}^{t}\int_{\mathbb{T}^{d}}\nabla\rho^{n,\gamma}\cdot(\nabla\psi)(x,\rho^{n,\gamma})\notag\\
&-\frac{1}{2}\int_{0}^{t}\int_{\mathbb{T}^{d}}F_{1}(x)\left[\sigma_n'(\rho^{n,\gamma})\right]^{2}\nabla\rho^{n,\gamma}\cdot(\nabla\psi)(x,\rho^{n,\gamma})-\frac{1}{2}\int_{0}^{t}\int_{\mathbb{T}^{d}}\sigma_n(\rho^{n,\gamma})\sigma_n'(\rho^{n,\gamma})F_{2}(x)\cdot(\nabla\psi)(x,\rho^{n,\gamma})\notag\\
&-\int_{0}^{t}\int_{\mathbb{T}^{d}}\left(\partial_{\xi}\psi\right)(x,\rho^{n,\gamma})|\nabla\rho^{n,\gamma}|^{2}+\frac{1}{2}\int_{0}^{t}\int_{\mathbb{T}^{d}}\left(\sigma_n(\rho^{n,\gamma})\sigma_n'(\rho^{n,\gamma})\nabla\rho^{n,\gamma}\cdot F_{2}(x)+F_{3}(x)\sigma_n^{2}(\rho^{n,\gamma})\right)\left(\partial_{\xi}\psi\right)(x,\rho^{n,\gamma})\notag\\
&-\int_{0}^{t}\int_{\mathbb{T}^{d}}\psi(x,\rho^{n,\gamma})\nabla\cdot(\rho^{n,\gamma} V_{\gamma}(r)\ast\rho^{n,\gamma})-\int_{0}^{t}\int_{\mathbb{T}^{d}}\psi(x,\rho^{n,\gamma})\nabla\cdot\left(\sigma_n(\rho^{n,\gamma})\mathrm{d}W^F\right),
\end{align}
where $\chi^{n,\gamma}:\mathbb{T}^d\times\mathbb{R}\times[0,T]\to\{0,1\}$ is the kinetic function given by $\chi^{n,\gamma}(x,\xi,t)=\mathbf{1}_{\{0<\xi<\rho^{n,\gamma}(x,t)\}}$, $\bar{\chi}^{n,\gamma}(\hat{\rho})(x,\xi)=\mathbf{1}_{\{0<\xi<\hat{\rho}(x)\}}$ and the kinetic measure $q^{n,\gamma}=\delta_{0}(\xi-\rho^{n,\gamma})|\nabla\rho^{n,\gamma}|^{2}$.
\end{definition}

Similar to \cite[Proposition 5.21]{FGold}, we can show that
the weak solution is equivalent to the renormalized kinetic solution. It reads as follows.
\begin{proposition}\label{pro-6.6}
	 Suppose that $V$ satisfies Assumption (A1). Let $\hat{\rho}\in L^{\infty}(\mathbb{T}^d)$ be a nonnegative function.  For any $\gamma>0$ and $n\in\mathbb{N}$, let $\rho^{n,\gamma}$ be a weak solution of \eqref{eq-5.1} in the sense of Definition \ref{def-5.2} with initial data $\rho^{n,\gamma}(\cdot,0)=\hat{\rho}$. Then $\rho^{n,\gamma}$ is a renormalized kinetic solution in the sense of Definition \ref{def-6.3}.
\end{proposition}

 \subsection{Tightness of approximating solutions }\label{subsec-5.3}
In this part, we aim to show the $L^1([0,T];L^1(\mathbb{T}^d))-$tightness of the laws of $\left\{\rho^{n,\gamma}\right\}_{n\in\mathbb{N},\gamma\in(0,1)}$  constructed in Theorem \ref{the-6.4}. {\color{black}According to the Aubin--Lions--Simon lemma  (\cite[Corollary 5]{ALS}), any bounded set in $L^1([0,T];W^{1,1}(\mathbb{T}^d)) \cap W^{\beta,1}([0,T];H^{-l}(\mathbb{T}^d))$ for $\beta\in(0,1/2),$
is relatively compact in $L^1([0,T];L^1(\mathbb{T}^d))$.} Hence, an important ingredient is to establish a stable $W^{\beta,1}([0,T];H^{-l}(\mathbb{T}^d))$-estimate. However, we cannot get such an estimation directly for $\left\{\rho^{n,\gamma}\right\}_{n\in\mathbb{N},\gamma\in(0,1)}$
in the presence of the singular term  $\nabla\cdot(F_{1}\rho^{-1}\nabla\rho)$ in \eqref{eq-2.2} which admits a singularity at $\rho=0$. To solve this problem,
we follow the idea used in \cite{FG21} to introduce a smooth function $h_{\delta}$ defined by Definition \ref{def-7.1} below to keep the solution away from zero. As a result, we can establish the $W^{\beta,1}([0,T];H^{-l}(\mathbb{T}^d))$-estimate for $h_{\delta}(\rho)$, which is presented in Proposition \ref{pro-7.3}.

Let $B$ be a Banach space.
 Given $1\le p<\infty$, $0<\sigma<1$, let $W^{\sigma, p}([0,T];B)$ be the fractional Sobolev space defined by \cite{Mar87}. For every $\delta\in(0,1)$, let $\psi_{\delta}\in C^{\infty}([0,\infty))$ be a smooth nondecreasing function satisfying $0\le\psi_{\delta}\le1$ and

%
%

\begin{align*}
 \psi_{\delta}(\xi)=
  \left\{
    \begin{array}{ll}
    1, & {\rm{if}}\ \xi\ge\delta, \\
      0, &{\rm{if}}\ \xi\le\delta/2,\\
      {\rm{smooth}}, &{\rm{otherwise}}.
    \end{array}
  \right.
\end{align*}

Clearly, $\left|\psi_{\delta}'(\xi)\right|\le c/\delta$ for some $c \in(0, \infty)$ independent of $\delta$.
\begin{definition}\label{def-7.1}
	For every $\delta\in(0,1)$, let $h_{\delta}\in\mathrm{C}^{\infty}([0,\infty))$ be defined by
	\begin{align*}
	h_{\delta}(\xi)=\psi_{\delta}(\xi) \xi\ \ \ {\rm{for\ every}}\ \xi\in[0, \infty).
	\end{align*}
\end{definition}
From the definition of $h_{\delta}$, it follows that $h'_{\delta}$ is supported on $[\frac{\delta}{2},\infty)$ and
	\begin{align}\label{eq-7.1}
	h_{\delta}'(\xi)=\psi_{\delta}'(\xi)\xi+\psi_{\delta}(\xi)\le c\mathbf{1}_{\left\{\xi\ge\frac{\delta}{2}\right\}},
	\end{align}
where $c$ is independent of $\delta$.
Moreover, there exists a constant $c\in(0,\infty)$  depending on $\delta$ such that
	\begin{align}\label{eq-7.2}
	h_{\delta}''(\xi)=\psi_{\delta}''(\xi)\xi+2\psi'_{\delta}(\xi)\le c(\delta)\mathbf{1}_{\left\{\frac{\delta}{2}\le\xi\le\delta\right\}}.
	\end{align}


To get tightness of $\rho^{n,\gamma}$ on $L^{1}\left([0,T];L^{1}\left(\mathbb{T}^{d}\right)\right)$, we need to introduce a new metric on $L^1([0,T];L^1(\mathbb{T}^d))$. For every $\delta\in (0,1)$, let $h_{\delta}$ be defined as in Definition \ref{def-7.1}. Let $D: L^1([0,T];L^1(\mathbb{T}^d))\times L^1([0,T];L^1(\mathbb{T}^d))\rightarrow [0,\infty)$ be defined by
\begin{align}
  D(f,g)=\sum^{\infty}_{k=1}2^{-k}
  \Big(\frac{\|h_{1/k}(f)-h_{1/k}(g)\|_{L^1([0,T];L^1(\mathbb{T}^d))}}{1+\|h_{1/k}(f)-h_{1/k}(g)\|_{L^1([0,T];L^1(\mathbb{T}^d))}}\Big).
\end{align}
It is proved by \cite{FG21} that the function $D$ is a metric on $L^1([0,T];L^1(\mathbb{T}^d))$ and the metric topology determined by $D$ is equal to strong norm topology on $L^1([0,T];L^1(\mathbb{T}^d))$.

{\color{black}Following the approach of \cite{FG21}, to obtain the $L^{1}\!\big([0,T];L^{1}(\mathbb{T}^d)\big)$-tightness of the laws of $\{\rho^{n,\gamma}\}_{n\in\mathbb{N},\,\gamma\in(0,1)}$, it suffices to show that, for every $k\in \mathbb{N}$, the laws $\{h_{1/k}(\rho^{n,\gamma})\}_{n\geq 1, \gamma>0}$ are tight on $L^{1}\!\big([0,T];L^{1}(\mathbb{T}^d)\big)$ in the strong topology. Thus, based on the Aubin-Lions-Simon lemma, we need to make estimates of
 $L^1([0,T];W^{1,1}(\mathbb{T}^d))$ and $W^{\beta,1}([0,T];H^{-l}(\mathbb{T}^d))$ norms  for $h_{\delta}(\rho^{n,\gamma})$.} The following result can be easily proved by using \eqref{eq-7.1}, Lemma \ref{lem-2}, H\"{o}lder's inequality and (\ref{eq-5.5}).
\begin{lemma}\label{lem-7.2}
	 Suppose that $V$ satisfies Assumption (A1). Let $\hat{\rho}\in L^{\infty}(\mathbb{T}^d)$ be a nonnegative function.  For any  $\gamma>0$ and $n\in\mathbb{N}$, let $\rho^{n,\gamma}$ be a weak solution of (\ref{eq-5.1}) in the sense of Definition \ref{def-5.2} with initial data $\rho^{n,\gamma}(\cdot,0)=\hat{\rho}$. Then, there exists a constant $c\in(0,\infty)$ independent of $\delta$ such that
	\begin{align*}
	\mathbb{E}\left[\|h_{\delta}(\rho^{n,\gamma})\|_{L^{1}\left([0,T];W^{1,1}\left(\mathbb{T}^{d}\right)\right)}\right]\le c(T, d,\|\hat{\rho}\|_{L^{1}(\mathbb{T}^{d})},\|V\|_{L^{p*}([0,T];L^p(\mathbb{T}^d;\mathbb{R}^d))}).
	\end{align*}
\end{lemma}

Moreover, we need the following $W^{\beta,1}([0,T];H^{-l}(\mathbb{T}^d))$-estimate for $h_{\delta}(\rho)$.
\begin{proposition}\label{pro-7.3}
		Suppose that $V$ satisfies Assumption (A1).  Let $\hat{\rho}\in L^{\infty}(\mathbb{T}^d)$ be a nonnegative function.  For any $\gamma>0$ and $n\in\mathbb{N}$, let $\rho^{n,\gamma}$ be a weak solution of (\ref{eq-5.1}) in the sense of Definition \ref{def-5.2} with initial data $\rho^{n,\gamma}(\cdot,0)=\hat{\rho}$. Then, for every $\beta\in(0,1/2)$ and $l>\frac{d}{2}+1$, there exists $c\in(0,\infty)$ depending on $\delta,\beta,T, d,l,\|V\|_{L^{p*}([0,T];L^p(\mathbb{T}^d;\mathbb{R}^d))}$, and $\|\hat{\rho}\|_{L^1(\mathbb{T}^d)}$ such that
	\begin{align}
	\mathbb{E}\left[\|h_{\delta}(\rho^{n,\gamma})\|_{W^{\beta,1}\left([0,T];H^{-l}\left(\mathbb{T}^{d}\right)\right)}\right]\le c(\delta,\beta,,T, d,l,\|V\|_{L^{p*}([0,T];L^p(\mathbb{T}^d;\mathbb{R}^d))},\|\hat{\rho}\|_{L^1(\mathbb{T}^d)}).\notag
	\end{align}
\end{proposition}

\begin{proof}
	Applying It\^{o} formula, for every $t\in[0,T]$, for every $\delta\in(0,1)$, as distribution on $\mathbb{T}^{d}$, we have almost surely that $h_{\delta}(\rho^{n,\gamma}(x,t))=h_{\delta}\left(\hat{\rho}\right)+J_{t}^{\mathrm{f}.\mathrm{v}}+J_{t}^{\mathrm{ker}}+J_{t}^{\mathrm{mart}}$. The terms $J_{t}^{\mathrm{f}.\mathrm{v}}$ and $J_{t}^{\mathrm{mart}}$ correspond respectively to the terms $J_{t}^{\mathrm{f}.\mathrm{v}}$ and $J_{t}^{\mathrm{mart}}$ in \cite[Proposition 5.14]{FGold}. Here we mainly focus on estimating the kernel term $J_{t}^{\mathrm{ker}}$ of the following form.

	\begin{align*}
	J_{t}^{\mathrm{ker}}=&-\int_{0}^{t}h_{\delta}'(\rho^{n,\gamma})\nabla\cdot(\rho^{n,\gamma} V_{\gamma}(r)\ast\rho^{n,\gamma}).
	\end{align*}
By integration by parts formula, we get
	\begin{align}\notag &\mathbb{E}\Big[\|J_{\cdot}^{\mathrm{ker}}\|_{W^{1,1}([0,T];H^{-l}(\mathbb{T}^{d}))}\Big]
=\mathbb{E}\int_{0}^{T}\Big(\|J_{t}^{\mathrm{ker}}\|_{H^{-l}(\mathbb{T}^{d})}
+\|\partial_tJ_{t}^{\mathrm{ker}}\|_{H^{-l}(\mathbb{T}^{d})}\Big)\\ \notag
	\le&(T+1)\mathbb{E}\int_{0}^{T}\sup_{\|\varphi\|_{H^l(\mathbb{T}^d)}=1}\Big|\int_{\mathbb{T}^d}h_{\delta}''(\rho^{n,\gamma})\varphi(x)\nabla\rho^{n,\gamma}\cdot(\rho^{n,\gamma} V_{\gamma}(t)\ast\rho^{n,\gamma})\Big|\\ \notag
	&+(T+1)\mathbb{E}\int_{0}^{T}\sup_{\|\varphi\|_{H^l(\mathbb{T}^d)}=1}\Big|\int_{\mathbb{T}^d}h_{\delta}'(\rho^{n,\gamma})\nabla\varphi(x)\cdot(\rho^{n,\gamma} V_{\gamma}(t)\ast\rho^{n,\gamma})\Big|\\
\label{eq-7.9}
=:&I_1+I_2.
	\end{align}
Note that when $l>\frac{d}{2}+1$, we have $\|f\|_{L^{\infty}(\mathbb{T}^d)}\le c\|f\|_{H^l(\mathbb{T}^d)}$ and $\|\nabla f\|_{L^{\infty}(\mathbb{T}^d)}\le c\|f\|_{H^l(\mathbb{T}^d)}$. By H\"{o}lder's and convolutional Young's inequalities, Lemma \ref{lem-2}, \ref{kk-66} \eqref{eq-7.2}, and Proposition \ref{pro-5.4}, we have
	\begin{align}\notag
	I_1\le&c(\delta,l)\mathbb{E}\left[\int_{0}^{T}\|\nabla\sqrt{\rho^{n,\gamma}}\cdot V_{\gamma}(t)\ast\rho^{n,\gamma}\|_{L^1(\mathbb{T}^d)}\right]\\ \notag
\le&c(\delta,l)\|\hat{\rho}\|_{L^1(\mathbb{T}^d)}\Big(\int_{0}^{T}\|V_{\gamma}(t)\|^2_{L^2(\mathbb{T}^d)}\Big)^{\frac12}\Big(\mathbb{E}\int_{0}^{T}\|\nabla\sqrt{\rho^{n,\gamma}}\|^{2}_{L^{2}(\mathbb{T}^{d})}\Big)^{\frac12}\\
\label{eq-7.10}
\le & c\left(\delta,l,T, d,\|\hat{\rho}\|_{L^{1}(\mathbb{T}^{d})},\|V\|_{L^{p*}([0,T];L^p(\mathbb{T}^d;\mathbb{R}^d))}\right).
	\end{align}
  Applying H\"{o}lder's inequality to $\frac{1}{p}+\frac{1}{q}=1$, by using convolutional Young's inequality, \eqref{eq-7.1}, Lemma \ref{kk-66}, Proposition \ref{pro-5.4} and Gagliardo-Nirenberg interpolation inequality, there exists a constant $c\in(0,\infty)$ depending on $l$ such that
	\begin{align}\notag
	I_2
\le & c(l)\|\hat{\rho}\|_{L^1(\mathbb{T}^d)}\mathbb{E}\int_{0}^{T}\|\rho^{n,\gamma}\|_{L^q(\mathbb{T}^d)} \|V_{\gamma}(t)\|_{L^p(\mathbb{T}^d)}\\ \notag
\le & c(l,d)\|\hat{\rho}\|^{2-\frac{d}{2p}}_{L^1(\mathbb{T}^d)}\mathbb{E}\int_{0}^{T}\|V_{\gamma}(t)\|_{L^p(\mathbb{T}^d)}\|\nabla\sqrt{\rho^{n,\gamma}}\|^{\frac{d}{p}}_{L^{2}(\mathbb{T}^{d})}\notag\\
\label{eq-7.11}
	\le&c\left(l,T, d,\|\hat{\rho}\|_{L^{1}(\mathbb{T}^{d})},\|V\|_{L^{p*}([0,T];L^p(\mathbb{T}^d;\mathbb{R}^d))}\right).
	\end{align}
	Based on \eqref{eq-7.9}-\eqref{eq-7.11}, we conclude that there exists a constant $c\in(0,\infty)$ such that
    \begin{align}\label{eq-7.12}
    \mathbb{E}\left[\|J_{\cdot}^{\mathrm{ker}}\|_{W^{1,1}([0,T];H^{-l}(\mathbb{T}^{d}))}\right]\le c\left(\delta,l,T, d,\|\hat{\rho}\|_{L^{1}(\mathbb{T}^{d})},\|V\|_{L^{p*}([0,T];L^p(\mathbb{T}^d;\mathbb{R}^d))}\right).
    \end{align}
	By the embeddings $W^{\beta,2}(\mathbb{T}^d), W^{1,1}(\mathbb{T}^d)\hookrightarrow W^{\beta,1}(\mathbb{T}^d)$ for every $\beta\in(0,\frac12)$, we complete the proof.
\end{proof}

{\color{black}With the help of Lemma \ref{lem-7.2}, Proposition \ref{pro-7.3} and the Aubin-Lions-Simon lemma, we can establish the tightness of the laws of $h_\delta(\rho^{n,\gamma})$ on $L^{1}\left([0,T];L^{1}\left(\mathbb{T}^{d}\right)\right)$ in the strong topology. Then, following the approach used in \cite{FG21}, we obtain the following result.}

\begin{proposition}\label{pro-7.8}
Suppose that $V$ satisfies Assumption (A1). Let $\hat{\rho}\in L^{\infty}(\mathbb{T}^d)$ be a nonnegative function. For any $\gamma>0$ and $n\in\mathbb{N}$, let $\rho^{n,\gamma}$ be the renormalized kinetic solution of (\ref{eq-5.1}) with initial data $\rho^{n,\gamma}(\cdot,0)=\hat{\rho}$, then the laws of $\left\{\rho^{n,\gamma}\right\}_{n\in\mathbb{N},\gamma\in(0,1)}$ are tight on $L^{1}\left([0,T];L^{1}(\mathbb{T}^{d}\right))$ in the strong norm topology.
\end{proposition}

\begin{proposition}\label{pro-7.9}
Suppose that $V$ satisfies Assumption (A1). Let $\hat{\rho}\in L^{\infty}(\mathbb{T}^d)$ be a nonnegative function.
For any $\gamma>0$, $n\in\mathbb{N}$ and $\psi\in\mathrm{C}_{c}^{\infty}\left(\mathbb{T}^{d}\times(0,\infty)\right)$, let $\rho^{n,\gamma}$ be the renormalized kinetic solution of (\ref{eq-5.1})  with initial data $\rho^{n,\gamma}(\cdot,0)=\hat{\rho}$. Let
 	\begin{align}\label{kk-63}	M_{t}^{n,\gamma,\psi}:=\int_{0}^{t}\int_{\mathbb{T}^{d}}\psi\left(x,\rho^{n,\gamma}\right)\nabla\cdot\left(\sigma_{n}\left(\rho^{n,\gamma}\right)\mathrm{d}W^F\right).
	\end{align}
 Then for every $\beta\in(0,1 /2)$, the laws of the martingales $\{M^{n,\gamma,\psi}\}_{n\in\mathbb{N},\gamma\in(0,1)}$ are tight on $\mathrm{C}^{\beta}([0,T];\mathbb{R})$. \end{proposition}

\section{Existence of renormalized kinetic solutions to Dean-Kawasaki equation}\label{sec-6}
In this section, we aim to prove the existence of renormalized kinetic solutions (strong in the probabilistic sense) to the Dean-Kawasaki equation
\begin{align}\label{eq-8.1}
\mathrm{d}\rho=\Delta\rho \mathrm{d}t-\nabla\cdot(\rho V(t)\ast\rho)\mathrm{d}t-\nabla\cdot(\sqrt\rho \mathrm{d}W^F(t))+\frac{1}{8}\nabla\cdot(F_1\rho^{-1}\nabla\rho+2F_2)\mathrm{d}t.
\end{align}
Our proof relies on pathwise uniqueness and the  following method from Gy\"{o}ngy and Krylov \cite[Lemma 1.1]{gyongy1996existence}, which is very close to the celebrate result of Yamada and Watanabe. 
		\begin{lemma}\label{lem-8.1}
			Let $(\Omega,\mathcal{F},\mathbb{P})$ be a probability space and $\left\{X_{n}:\Omega\to\bar{X}\right\}_{n\geq 1}$ be a sequence of random variables, where $\bar{X}$ is a complete separable metric space. Then,
			$X_{n}$ converges in probability as $n\to\infty$, if and only if for any sequences $\left\{\left(n_{k},m_{k}\right)\right\}_{k=1}^{\infty}$ satisfying $n_{k}, m_{k}\to\infty$ as $k\to\infty$, there exists a further subsequence $\left\{\left(n_{k^{\prime}},m_{k^{\prime}}\right)\right\}_{k^{\prime}=1}^{\infty}$ fulfilling $n_{k^{\prime}},m_{k^{\prime}}\to\infty$ as $k^{\prime}\to\infty$ such that the joint laws of $\left(X_{n_{k^{\prime}}}, X_{m_{k^{\prime}}}\right)_{k^{\prime} \in \mathbb{N}}$ converge weakly to a probability measure $\mu$ on $\bar{X} \times \bar{X}$ satisfying $\mu(\{(x, y)\in\bar{X}\times\bar{X}: x=y\})=1$, as $k^{\prime}\to\infty$.
		\end{lemma}
		\begin{remark}\label{remark-6.2}
In the above lemma, the state space $\bar{X}$ is  required to be a Polish space. In fact, according to  \cite[Theorem 1.1]{holden2022global}, this condition can be relaxed to a Hilbert space endowed with the weak topology. We mention that the weak conditional version will be utilized to establish the existence of probabilistically strong solution during the proof of the following Theorem \ref{the-existence}.
		\end{remark}


	

\begin{theorem}\label{the-existence}
	Let $\hat{\rho}\in\text{{\rm{Ent}}}\left(\mathbb{T}^{d}\right)$. We have the following two results.
	
	\item{(i)} Suppose that Assumption (A1) holds. Then there exists a stochastic basis $(\tilde{\Omega},\tilde{\mathcal{F}},\{\tilde{\mathcal{F}}(t)\}_{t\in[0,T]},\tilde{\mathbb{P}})$, {\color{black} a trace-class Brownian motion $\tilde{W}^{F}$ on  $L^2(\mathbb{T}^{d})$,} and a process $\tilde{\rho}\in L^1(\tilde{\Omega};L^1([0,T]\times \mathbb{T}^d))$, which is a renormalized kinetic solution of \eqref{eq-8.1} in the sense of Definition \ref{def-2.4} with the equation (\ref{eq-2.6}) holds for almost every $t\in [0,T]$ and $\tilde{\rho}(\cdot,0)=\hat{\rho}$.
	\item{(ii)} Suppose that Assumptions (A1) and (A2) holds. Then there exists a probabilistically strong renormalized kinetic solution $\rho$ of \eqref{eq-8.1} in the sense of Definition \ref{def-2.4} with initial data $\rho(\cdot,0)=\hat{\rho}$.
\end{theorem}
\begin{proof}
The proof mainly refers to the method of \cite[Theorem 5.25]{FG21}, but it should be noted that our kinetic measure of the approximation equation is not a finite measure on $[0,T] \times \mathbb{T}^d \times \mathbb{R}$. This is weaker than the situation in \cite{FG21}. In addition, some efforts are required to pass to the limit of the extra kernel term. We start with the proof of (i), which involves the following six steps.


		
		\noindent \textbf{The proof of (i).}\	$\mathbf{Step\ 1.\ Jakubowski-Skorokhod \ representation\ theorem.}$ Let $l>\frac{d}{2}+1$ be some integer and fix a countable sequence $\left\{\psi_{j}\right\}_{j\in\mathbb{N}}$ which is dense in $\mathrm{C}_{c}^{\infty}\left(\mathbb{T}^{d} \times(0, \infty)\right)$ in the strong $H^{l}\left(\mathbb{T}^{d} \times(0, \infty)\right)$-topology. Using the tightness estimates from Proposition \ref{pro-7.8} and Proposition \ref{pro-7.9}, along with Jakubowski-Skorokhod representation theorem \cite{jakubowski1997almost}, there exists a probability space $(\tilde{\Omega},\tilde{\mathcal{F}},\tilde{\mathbb{P}})$ and {\color{black}a nonnegative process $\tilde{\rho}^{k}\in L^{1}\Big(\tilde{\Omega};L^{1}([0,T]\times\mathbb{T}^{d})\Big)$, $\nabla\sqrt{\tilde{\rho}^{k}}\in L^{2}\Big(\tilde{\Omega};L^{2}([0,T]\times\mathbb{T}^{d};\mathbb{R}^{d})\Big)$,
		 $(\tilde{M}^{k,\psi_{j}})_{j\in\mathbb{N}}\in L^{2}(\tilde{\Omega};\mathrm{C}([0,T])^{\mathbb{N}})$, and a kinetic measure $\tilde{q}^k$,} such that for any $k\in\mathbb{N}$,
	\begin{align}\notag	
{\color{black}\tilde{M}_t^{k,\psi_{j}}}
=&-\left.{\color{black}\int_0^{\infty}}\int_{\mathbb{T}^{d}}\tilde{\chi}^{k}(x,\xi,r)\psi_j(x,\xi)\right|_{r=0}^{r=t} -\int_{0}^{t}\int_{\mathbb{T}^{d}}\nabla\tilde{\rho}^{k}\cdot(\nabla\psi_j)\left(x,\tilde{\rho}^{k}\right)\\
\notag		
    &-\frac{1}{2}\int_{0}^{t}\int_{\mathbb{T}^{d}}F_1\left[\sigma_{n_k}^{\prime}\left(\tilde{\rho}^{k}\right)\right]^{2}\nabla\tilde{\rho}^{k}\cdot(\nabla\psi_j)\left(x,\tilde{\rho}^{k}\right)
	-\frac{1}{2}\int_{0}^{t}\int_{\mathbb{T}^{d}}\sigma_{n_k}\left(\tilde{\rho}^{k}\right)\sigma_{n_k}^{\prime}\left(\tilde{\rho}^{k}\right)F_{2}\cdot(\nabla\psi_j)\left(x,\tilde{\rho}^{k}\right)
\\ \notag			 &+\frac{1}{2}\int_{0}^{t}\int_{\mathbb{T}^{d}}\left(\partial_{\xi}\psi_j\right)\left(x,\tilde{\rho}^{k}\right)\sigma_{n_k}\left(\tilde{\rho}^{k}\right)\sigma'_{n_k}\left(\tilde{\rho}^{k}\right)\nabla\tilde{\rho}^{k}\cdot F_{2}+\frac{1}{2}\int_{0}^{t}\int_{\mathbb{T}^{d}}F_{3}\sigma_{n_k}^{2}\left(\tilde{\rho}^{k}\right)\left(\partial_{\xi}\psi_j\right)\left(x,\tilde{\rho}^{k}\right)\\
	\label{kk-68} &-\int_{0}^{t}{\color{black}\int_0^{\infty}}\int_{\mathbb{T}^{d}}\partial_{\xi}\psi_j(x,\xi)\mathrm{d}\tilde{q}^{k}
	-\int_{0}^{t}\int_{\mathbb{T}^{d}}\psi_j(x,\tilde{\rho}^{k})\nabla\cdot(\tilde{\rho}^{k}V_{\gamma_k}(r)\ast\tilde{\rho}^{k}).
	\end{align}
where
{\black
	\begin{align}\label{eq-8.4}	\tilde{q}^k(\tilde{\omega})\geq\delta_0(\xi-\tilde{\rho}^k(\tilde{\omega}))|\nabla\tilde{\rho}^k(\tilde{\omega})|^2 \ \ \text{for every }\tilde{\omega}\in\tilde{\Omega}.
\end{align}
}
 Moreover, there exists $\tilde{\rho}\in L^1(\tilde{\Omega};L^1([0,T]\times\mathbb{T}^d))$, $\nabla\sqrt{\tilde{\rho}}\in L^{2}(\tilde{\Omega};L^{2}([0,T]\times\mathbb{T}^{d};\mathbb{R}^{d}))$
	and $(\tilde{M}^{\psi_{j}})_{j\in\mathbb{N}}\in L^{1}(\tilde{\Omega};\mathrm{C}([0,T])^{\mathbb{N}})$ such that
	as $k\to\infty$,
	\begin{align}\label{eq-8.19}
		\tilde{\rho}^{k}\to \tilde{\rho} \quad {\rm{strongly\ in\ }} L^1([0,T];L^1(\mathbb{T}^d)),\ \tilde{\mathbb{P}}-a.s.,
	\end{align}
	\begin{align}\label{kk-67}
		\nabla\sqrt{\tilde{\rho}^k}\rightharpoonup \nabla\sqrt{\tilde{\rho}} \quad {\rm{weakly\ in\ }} L^{2}(\tilde{\Omega};L^{2}([0,T]\times\mathbb{T}^{d};\mathbb{R}^{d})),
	\end{align}
	and
	\begin{align}\label{kk-70}
		{\color{black}\tilde{M}_t^{k,\psi_{j}}}\rightarrow \tilde{M}^{\psi_{j}}_t,\ \tilde{\mathbb{P}}-\text{almost surely for any}\ t\in [0,T], j\in \mathbb{N},
	\end{align}
	where $\tilde{\mathbb{P}}$-almost surely for every $j\in\mathbb{N}$ and $t\in[0,T]$,
	\begin{align}\label{kk-69}
		\tilde{M}^{\psi_{j}}_t=\int^t_0\int_{\mathbb{T}^d}\psi_j(x,\tilde{\rho})\nabla\cdot(\sqrt{\tilde{\rho}}\mathrm{d}\tilde{W}^F)
	\end{align}
	and $\tilde{W}^F$ is defined analogously to \eqref{kk-43} by the Brownian motion $\tilde{\beta}$ on $\tilde{\Omega}$. For the proof of the result outlined above, we refer readers to the proof of \cite[Theorem 5.25]{FG21}, which will not be repeated here.

\noindent	$\mathbf{Step\ 2.\ Existence\ of\ a\ limiting\ kinetic\ measure.}$
	For any $M>0$, set
	\begin{align*}
	\theta_M(\xi):=\mathbf{1}_{[0,M]}(\xi),\quad \Theta_M(\xi):=\int^{\xi}_{0}\int^{r}_{0}\theta_M(s)\mathrm{d}s\mathrm{d}r.
	\end{align*}
With the aid of an approximation argument, we first apply It\^{o} formula to $\Theta_M(\rho^{n_k,\gamma_{k}}(x,T))$. Then, according to Definition \ref{def-6.3}, the definition of $q^{n_k,\gamma_k}$, and the properties that $\tilde{q}^k$ has the same law as $q^{n_k,\gamma_k}$ on $\mathcal{M}\left(\mathbb{T}^{d}\times\mathbb{R}\times[0,T]\right)$ and  $\rho^{n_k,\gamma_k}$ has the same law as $\tilde{\rho}^k$,  it follows almost surely that for every  $M>0$ and $k\in\mathbb{N}$,
	\begin{align}\label{eq-8.06}
	\tilde{\mathbb{E}}\Big|\int^T_0{\color{black}\int_0^{\infty}}\int_{\mathbb{T}^d}\theta_M(\xi)
\mathrm{d}\tilde{q}^k\Big|^2
\lesssim& \tilde{\mathbb{E}}\Big|\int_{\mathbb{T}^d}
\Theta_M(\hat{\rho}^{n_k})\Big|^2+\tilde{\mathbb{E}}\Big|\int_{\mathbb{T}^d}\Theta_M(\tilde{\rho}^k(x,T))\Big|^2\notag\\
&+\tilde{\mathbb{E}}\Big|\int^T_0\int_{\mathbb{T}^d}\Theta'_M(\tilde{\rho}^k)\nabla\cdot
\Big(\tilde{\rho}^kV_{\gamma_{k}}\ast\tilde{\rho}^k\Big)\Big|^2\notag\\
&+\tilde{\mathbb{E}}\Big|\int^T_0\int_{\mathbb{T}^d}\theta_M(\tilde{\rho}^k)\sigma_{n_k}(\tilde{\rho}^k)\nabla\tilde{\rho}^k\cdot\mathrm{d}\tilde{W}^F\Big|^2\notag\\
&+\tilde{\mathbb{E}}\Big|\int^T_0\int_{\mathbb{T}^d}\theta_M(\tilde{\rho}^k)\Big(\sigma_{n_k}(\tilde{\rho}^k)\sigma'_{n_k}(\tilde{\rho}^k)\nabla\tilde{\rho}^k\cdot F_2+F_3\sigma^2_{n_k}(\tilde{\rho}^k)\Big)\Big|^2.
	\end{align}
	 According to the property $\Theta_M(\xi)\leq M(M+|\xi|)$, it follows that
	\begin{align*}
	\tilde{\mathbb{E}}\Big|\int_{\mathbb{T}^d}\Theta_M(\hat{\rho}^{n_k})\Big|^2
+\tilde{\mathbb{E}}\Big|\int_{\mathbb{T}^d}\Theta_M(\tilde{\rho}^k(x,T))\Big|^2
\le C(M)\|\hat{\rho}\|^2_{L^1(\mathbb{T}^d)}.
	\end{align*}
	For the martingale term on the righthand side of \eqref{eq-8.06}, by  It\^{o} isometry,  Assumption (A1), \eqref{kk-5.2} and Proposition \ref{pro-5.4}, we deduce that there exists a  constant $c\in(0,\infty)$ independent of $k$ such that
	\begin{align*} \tilde{\mathbb{E}}\Big|\int^T_0\int_{\mathbb{T}^d}\theta_M(\tilde{\rho}^k)\sigma_{n_k}(\tilde{\rho}^k)
\nabla\tilde{\rho}^k\cdot\mathrm{d}\tilde{W}^F\Big|^2
\le&\tilde{\mathbb{E}}\int^T_0\int_{\mathbb{T}^d}F_1(x)\mathbf{1}_{\{0\le\tilde{\rho}^k\le M\}}|\sigma_{n_k}(\tilde{\rho}^k)|^2|\nabla\tilde{\rho}^k|^2\\
	\le& c(M,T, d,\|\hat{\rho}\|_{L^1(\mathbb{T}^d)},\|V\|_{L^{p*}([0,T];L^p(\mathbb{T}^d;\mathbb{R}^d))}).
	\end{align*}
Note that
	\begin{align*}
	\mathbf{1}_{\{0\le\tilde{\rho}^k\le M\}}\sigma_{n_k}(\tilde{\rho}^k)\sigma'_{n_k}(\tilde{\rho}^k)\nabla\tilde{\rho}^k
=\frac12\nabla(\sigma^2_{n_k}(\tilde{\rho}^k\wedge M)),
	\end{align*}
 by integration by parts formula, we deduce that
	\begin{align*}
	\int^T_0\int_{\mathbb{T}^d}\theta_M(\tilde{\rho}^k)\sigma_{n_k}(\tilde{\rho}^k)\sigma'_{n_k}(\tilde{\rho}^k)\nabla\tilde{\rho}^k\cdot F_2=-\frac12\int^T_0\int_{\mathbb{T}^d}\sigma^2_{n_k}(\tilde{\rho}^k\wedge M)\nabla\cdot F_2=0.
	\end{align*}
	Moreover, there exists a constant $c\in(0,\infty)$ independent of $k$ such that
	\begin{align*}
	\int^T_0\int_{\mathbb{T}^d}\theta_M(\tilde{\rho}^k)F_3\sigma^2_{n_k}(\tilde{\rho}^k)\le c\int^T_0\int_{\mathbb{T}^d}\mathbf{1}_{\{0\le\tilde{\rho}^k\le M\}}F_3\tilde{\rho}^k\le c(M,T ).
	\end{align*}
	It remains to make estimate of the kernel term. By integration by parts formula,
H\"{o}lder's  and convolutional Young's inequalities, and Proposition \ref{pro-5.4}, we deduce that there exists a constant $c\in(0,\infty)$ independent of $k$ such that
	\begin{align*}
		&\tilde{\mathbb{E}}\Big|\int^T_0\int_{\mathbb{T}^d}\Theta'_M(\tilde{\rho}^k)\nabla\cdot
(\tilde{\rho}^kV_{\gamma_{k}}\ast\tilde{\rho}^k)\Big|^2
=\tilde{\mathbb{E}}\Big|\int^T_0\int_{\mathbb{T}^d}\mathbf{1}_{\{0\le\tilde{\rho}^k\le M\}}\nabla\tilde{\rho}^k\cdot(\tilde{\rho}^kV_{\gamma_{k}}\ast\tilde{\rho}^k)\Big|^2\\
		\le&4M^3\tilde{\mathbb{E}}\Big(\int^T_0\|\nabla\sqrt{\tilde{\rho}^k}\cdot V_{\gamma_{k}}\ast\tilde{\rho}^k\|_{L^1(\mathbb{T}^d)}\Big)^2\le4M^3\|\hat{\rho}\|^2_{L^1(\mathbb{T}^d)}\|V\|^2_{L^2([0,T];L^2(\mathbb{T}^d;\mathbb{R}^d))}\tilde{\mathbb{E}}\int^T_0\|\nabla\sqrt{\tilde{\rho}^k}\|^2_{L^2(\mathbb{T}^d)}\\
		\le&c(M,T, d,\|\hat{\rho}\|_{L^1(\mathbb{T}^d)},\|V\|_{L^{p*}([0,T];L^p(\mathbb{T}^d;\mathbb{R}^d))}).
	\end{align*}
	Combining all the above estimates, we conclude that there exists a constant $\Lambda$ such that
	\begin{align}\label{eq-8.266}
	&\sup_{k\geq1}\tilde{\mathbb{E}}(\tilde{q}^k([0,T]\times\mathbb{T}^d\times[0,M]))^2\le\Lambda(M,T, d,\|\hat{\rho}\|_{L^1(\mathbb{T}^d)},\|V\|_{L^{p*}([0,T];L^p(\mathbb{T}^d;\mathbb{R}^d))}).
	\end{align}
	
	For every $r\in \mathbb{N}$, define $K_r:=\mathbb{T}^d\times [0,T]\times [0,r]$. Let $\mathcal{M}_r$ be the space of bounded Borel measures over $K_r$ (with norm given by the total variation of measures). Clearly, $\mathcal{M}_r$ is the topological dual of $C(K_r)$, which is the set of continuous functions on $K_r$.
	By \eqref{eq-8.266}, the sequence $\{\tilde{q}^{k}=\delta_0(\cdot-\tilde{\rho}^k)|\nabla\tilde{\rho}^k|^2\}_{k\geq1}$ is uniformly bounded in $L^2(\tilde{\Omega}; \mathcal{M}_r)$. By the Banach-Alaoglu theorem, there exists $\tilde{q}_r\in L^2(\tilde{\Omega}; \mathcal{M}_r)$ and
	a subsequence still denoted by $\{\tilde{q}^{k}\}_{k\in \mathbb{N}}$ such that $\tilde{q}^{k}\rightharpoonup\tilde{q}_r $ in $L^2(\tilde{\Omega}; \mathcal{M}_r)-$weak star as $k\to\infty$. By a diagonal process, we extract a subsequence (not relabeled) and a Radon measure $\tilde{q}$  on $\mathbb{T}^d\times [0,T]\times [0,\infty)$ such that $\tilde{q}^{k}\rightharpoonup\tilde{q}-$weak
	star  in $L^2(\tilde{\Omega}; \mathcal{M}_r)$ as $k\to\infty$ for every $r\in \mathbb{N}$.
	
    The limiting measure $\tilde{q}$ is a kinetic measure
in the sense of	Definition \ref{def-2.2}, since the predictable property is stable with respect to weak limits. In addition, we claim that $\tilde{q}$ fulfills (\ref{eq-2.66}). That is,
	\begin{align}\label{eq-8.21}
	\tilde{q}(t,x,\xi)\geq 4\delta_0(\xi-\tilde{\rho}) \tilde{\rho}|\nabla\sqrt{\tilde{\rho}}|^2\quad \tilde{\mathbb{P}}-a.s.
	\end{align}
%
Indeed,
we can choose a subsequence (still denoted by $\tilde{\rho}^k$) such that for every $A\in\tilde{\mathcal{F}}$ and for all nonnegative $\phi\in \mathrm{C}_{c}^{\infty}\left(\mathbb{T}^{d}\times[0,T]\times(0,\infty)\right)$,
	\begin{align*}
	\nabla\sqrt{\tilde{\rho}^k}\sqrt{\phi(x,t,\tilde{\rho}^k)\tilde{\rho}^k}I_A\rightharpoonup\nabla\sqrt{\tilde{\rho}}\sqrt{\phi(x,t,\tilde{\rho})\tilde{\rho}}I_A,
	\end{align*}
	weakly in $L^2(\tilde{\Omega}\times[0,T];L^2(\mathbb{T}^d))$, as $k\to\infty$. By the lower semi-continuity of the Sobolev norm, we have
	\begin{align*}
	4\tilde{\mathbb{E}}\Big(\int^T_0\int_{\mathbb{T}^d}|\nabla\sqrt{\tilde{\rho}}|^2\tilde{\rho}
	\phi(x,t,\tilde{\rho})I_A\Big)
	\leq 4\liminf_{k\to\infty}\tilde{\mathbb{E}}\Big(\int^T_0\int_{\mathbb{T}^d}|\nabla\sqrt{\tilde{\rho}^{k}}|^2\tilde{\rho}^k
	\phi(x,t,\tilde{\rho}^k)I_A\Big)=\tilde{\mathbb{E}}\Big(\tilde{q}(\phi)I_A\Big).
	\end{align*}
	
\noindent	$\mathbf{Step\ 3.\ The\ entropy\ estimate.}$ Since $\tilde{\rho}^k$ satisfies the entropy estimates \eqref{eq-5.5} uniformly on $k$, by (\ref{kk-67}) and the weak lower semi-continuity of the Sobolev norm, we deduce that
	\begin{align*}
	\tilde{\mathbb{E}}\Big[\int_{0}^{T}\int_{\mathbb{T}^{d}}|\nabla \sqrt{\tilde{\rho}}|^{2}\Big]\le \int_{\mathbb{T}^{d}}\Psi(\hat{\rho})+1+c(T, d,\|\hat{\rho}\|_{L^1(\mathbb{T}^d)},\|V\|_{L^{p*}([0,T];L^p(\mathbb{T}^d;\mathbb{R}^d))}),
	\end{align*}
where we have used $\underset{\xi>0}{{\rm{\min}}}\Psi(\xi)=\Psi(1)=-1$.
Thus, $\tilde{\rho}$ satisfies the condition \eqref{eq-2.5} in Definition \ref{def-2.4}.
	
\noindent	$\mathbf{Step\ 4.\ Passing\ to\ the\ limits.}$ From now on, we aim to show that $(\tilde{\rho},\tilde{q}, \tilde{\beta})$
	is a kinetic solution to \eqref{eq-8.1} in the sense of Definition \ref{def-2.4}. To facilitate the proof of convergence, we denote $\sigma(\xi)=\sqrt\xi$ for every $\xi\in[0,\infty)$.
	Clearly, (1) and (2) in Definition \ref{def-2.4} hold. To achieve the result, we need to verify that for every $j\in\mathbb{N}$ and $t\in[0,T]$, the kinetic function $\tilde{\chi}$ of $\tilde{\rho}$ satisfies
	\begin{align}\label{eq-8.23}
	&\int^t_0\int_{\mathbb{T}^d}\psi_j(x,\tilde{\rho})
	\nabla\cdot(\sigma(\tilde{\rho})\mathrm{d}\tilde{W}^F)
	=-\left.{\color{black}\int_0^{\infty}}\int_{\mathbb{T}^{d}}\tilde{\chi}(x,\xi,r)\psi_{j}(x,\xi)\right|_{r=0}^{r=t}
	-\int_{0}^{t}\int_{\mathbb{T}^{d}}\nabla\tilde{\rho}\cdot(\nabla\psi_{j})\left(x,\tilde{\rho}\right)\notag\\ &-\frac{1}{2}\int_{0}^{t}\int_{\mathbb{T}^{d}}F_1\left[\sigma^{\prime}\left(\tilde{\rho}\right)\right]^{2}\nabla\tilde{\rho}\cdot(\nabla\psi_{j})\left(x,\tilde{\rho}\right)-\frac{1}{2}\int_{0}^{t}\int_{\mathbb{T}^{d}}\sigma\left(\tilde{\rho}\right)\sigma^{\prime}\left(\tilde{\rho}\right)F_{2}\cdot(\nabla\psi_{j})\left(x,\tilde{\rho}\right)\notag\\
	&+\frac{1}{2}\int_{0}^{t}\int_{\mathbb{T}^{d}}\left(\partial_{\xi}\psi_{j}\right)\left(x,\tilde{\rho}\right)\sigma\left(\tilde{\rho}\right)\sigma^{\prime}\left(\tilde{\rho}\right)\nabla\tilde{\rho}\cdot F_{2}+\frac{1}{2}\int_{0}^{t}\int_{\mathbb{T}^{d}}F_{3}\sigma^{2}\left(\tilde{\rho}\right)\left(\partial_{\xi}\psi_{j}\right)\left(x,\tilde{\rho}\right)\notag\\
	&-\int_{0}^{t}{\color{black}\int_0^{\infty}}\int_{\mathbb{T}^{d}}\partial_{\xi}\psi_{j}(x,\xi)\mathrm{d}\tilde{q}-\int_{0}^{t}\int_{\mathbb{T}^{d}}\psi_{j}(x,\tilde{\rho})\nabla\cdot(\tilde{\rho}V(r)\ast\tilde{\rho}).
	\end{align}
	
By the definition of $\tilde{\chi}^{k}$, (\ref{eq-8.19}) and the dominated convergence theorem, for every $j\in\mathbb{N}$ and $r\in[0,T]$,
	\begin{align*}
	\lim_{k\to\infty}\tilde{\mathbb{E}}\Big|{\color{black}\int_0^{\infty}}\int_{\mathbb{T}^{d}}\tilde{\chi}^{k}(x,\xi,r)\psi_{j}(x,\xi)-{\color{black}\int_0^{\infty}}\int_{\mathbb{T}^{d}}\tilde{\chi}(x,\xi,r)\psi_{j}(x,\xi)\Big|=0.
	\end{align*}
Since $\psi_{j}\in\mathrm{C}_{c}^{\infty}\left(\mathbb{T}^{d}\times(0,\infty)\right)$, we can choose $0<\delta<M<\infty$ such that $[\delta,M]$ is the compact support of $\psi_{j}$ with respect to $\xi$. Then,
by Lemma \ref{lem-2}, (\ref{kk-67}), Proposition \ref{pro-5.4}, the boundedness of $F_1$ and \eqref{eq-5.2}, it follows that
for every $j\in\mathbb{N}$ and $t\in[0,T]$,
	\begin{align}\notag
	\lim_{k\to\infty}\tilde{\mathbb{E}}\int_{0}^{t}\int_{\mathbb{T}^{d}}
F_1\left[\sigma_{n_k}^{\prime}\left(\tilde{\rho}^k\right)\right]^{2}
\nabla\tilde{\rho}^k\cdot(\nabla\psi_{j})\left(x,\tilde{\rho}^k\right)
=\tilde{\mathbb{E}}\int_{0}^{t}\int_{\mathbb{T}^{d}}F_1\left[\sigma^{\prime}
\left(\tilde{\rho}\right)\right]^{2}\nabla\tilde{\rho}\cdot(\nabla\psi_{j})
\left(x,\tilde{\rho}\right).
	\end{align}
	Since for every $r\in \mathbb{N}$, $\tilde{q}^{k}\rightharpoonup\tilde{q}$ in $L^2(\tilde{\Omega}; \mathcal{M}_r)-$weak star as $k\to\infty$, it follows from the property $\left\{\psi_{j}\right\}_{j\in\mathbb{N}}\in\mathrm{C}_{c}^{\infty}\left(\mathbb{T}^{d}\times(0,\infty)\right)$ that for every $j\in\mathbb{N}$ and $t\in[0,T]$,
	\begin{align*}
	\lim_{k\to\infty}\tilde{\mathbb{E}}\int_{0}^{t}{\color{black}\int_0^{\infty}}\int_{\mathbb{T}^{d}}\partial_{\xi}\psi_{j}(x,\xi)\mathrm{d}\tilde{q}^k=\tilde{\mathbb{E}}\int_{0}^{t}{\color{black}\int_0^{\infty}}\int_{\mathbb{T}^{d}}\partial_{\xi}\psi_{j}(x,\xi)\mathrm{d}\tilde{q}.
	\end{align*}

Similar to the above, by the boundedness of $F_2$ and $F_3$, \eqref{eq-5.2}, Lemma \ref{rem-5.1}, (\ref{eq-8.19}) and (\ref{kk-67}), we get the other terms of (\ref{eq-8.23}) except the last kernel term.
For the kernel term, by integration by parts formula, we have
	\begin{align}\label{eq-8.27}	&\tilde{\mathbb{E}}\left|\int_{0}^{t}\int_{\mathbb{T}^{d}}\psi_{j}(x,\tilde{\rho}^k)\nabla\cdot(\tilde{\rho}^kV_{\gamma_k}(r)\ast\tilde{\rho}^k)-\psi_{j}(x,\tilde{\rho})\nabla\cdot(\tilde{\rho}V(r)\ast\tilde{\rho})\right|\le b^{k}_1+b^{k}_2,
\end{align}
where
\begin{align*}
b^{k}_1=&\tilde{\mathbb{E}}\Big|\int_{0}^{t}\int_{\mathbb{T}^{d}}(\partial_{\xi}\psi_{j})(x,\tilde{\rho}^k)\nabla\tilde{\rho}^k\cdot\tilde{\rho}^kV_{\gamma_k}(r)\ast\tilde{\rho}^k-(\partial_{\xi}\psi_{j})(x,\tilde{\rho})\nabla\tilde{\rho}\cdot\tilde{\rho}V(r)\ast\tilde{\rho}\Big|,\notag\\	b^{k}_2=&\tilde{\mathbb{E}}\Big|\int_{0}^{t}\int_{\mathbb{T}^{d}}(\nabla\psi_{j})(x,\tilde{\rho}^k)\cdot\tilde{\rho}^k V_{\gamma_k}(r)\ast\tilde{\rho}^k-(\nabla\psi_{j})(x,\tilde{\rho})\cdot\tilde{\rho}V(r)\ast\tilde{\rho}\Big|.
\end{align*}
	We firstly proceed with the term $b^{k}_1$. By Lemma \ref{lem-2}, we have
	\begin{align}\label{eq-8.28}
	b^{k}_1\le b^{k}_{11}+b^{k}_{12}+b^{k}_{13},
\end{align}
where
\begin{align*}	b^{k}_{11}:=&2\tilde{\mathbb{E}}\left|\int_{0}^{t}\int_{\mathbb{T}^{d}}
(\partial_{\xi}\psi_{j})(x,\tilde{\rho}^k)(\tilde{\rho}^k)^{3/2}(V_{\gamma_k}(r)-V(r))
\ast\tilde{\rho}^k\cdot\nabla\sqrt{\tilde{\rho}^k}\right|,\\ b^{k}_{12}:=&2\tilde{\mathbb{E}}\left|\int_{0}^{t}\int_{\mathbb{T}^{d}}
\left((\partial_{\xi}\psi_{j})(x,\tilde{\rho}^k)(\tilde{\rho}^k)^{3/2}V(r)\ast
\tilde{\rho}^k-(\partial_{\xi}\psi_{j})(x,\tilde{\rho})(\tilde{\rho})^{3/2}V(r)
\ast\tilde{\rho}\right)\cdot\nabla\sqrt{\tilde{\rho}^k}\right|,\\ b^{k}_{13}:=&2\tilde{\mathbb{E}}\left|\int_{0}^{t}\int_{\mathbb{T}^{d}}
(\partial_{\xi}\psi_{j})(x,\tilde{\rho})(\tilde{\rho})^{3/2}V(r)\ast
\tilde{\rho}\cdot(\nabla\sqrt{\tilde{\rho}^k}-\nabla\sqrt{\tilde{\rho}})\right|.
	\end{align*}
	Recall that  $[\delta,M]$ is the compact support of $\psi_{j}$ with respect to $\xi$. For the term $b^{k}_{11}$, it follows from H\"{o}lder's inequality, Proposition \ref{pro-5.4}, the property of $\psi_{j}$ and convolutional Young's inequality that there exists a constant $c\in(0,\infty)$ independent of $k$ such that
	\begin{align}\label{eq-8.29}
	b^{k}_{11}\le&2\left(\tilde{\mathbb{E}}\int_{0}^{t}\int_{\mathbb{T}^{d}}\left|(\partial_{\xi}\psi_{j})(x,\tilde{\rho}^k)(\tilde{\rho}^k)^{3/2}(V_{\gamma_k}(r)-V(r))\ast\tilde{\rho}^k\right|^2\right)^{\frac12}\left(\tilde{\mathbb{E}}\int_{0}^{t}\int_{\mathbb{T}^{d}}\left|\nabla\sqrt{\tilde{\rho}^k}\right|^2\right)^{\frac12}\notag\\
	\le&cM^{\frac32}\|\partial_{\xi}\psi_{j}\|_{L^{\infty}(\mathbb{R})}\|\hat{\rho}\|_{L^1(\mathbb{T}^d)}\|V_{\gamma_k}-V\|_{L^2([0,T];L^2(\mathbb{T}^d;\mathbb{R}^d))}\rightarrow 0,\ {\rm{as}}\ k\to\infty.
	\end{align}
According to the properties of $V$ and $\{\psi_{j}\}_{j\in\mathbb{N}}$, by \eqref{kk-67}, it holds that for every $j\in\mathbb{N}$ and $t\in[0,T]$,
	\begin{align}\label{eq-8.31}
	\lim_{k\to\infty}b^{k}_{13}=0.
		\end{align}
	For the term $b^{k}_{12}$, it follows from H\"{o}lder's inequality and Proposition \ref{pro-5.4} that there exists a constant $c\in(0,\infty)$ independent of $k$ such that
	\begin{align}\notag
	b^{k}_{12}\le&2\left(\tilde{\mathbb{E}}\int_{0}^{t}\int_{\mathbb{T}^{d}}
\left|(\partial_{\xi}\psi_{j})(x,\tilde{\rho}^k)(\tilde{\rho}^k)^{3/2}
V(r)\ast\tilde{\rho}^k-(\partial_{\xi}\psi_{j})(x,\tilde{\rho})(\tilde{\rho})^{3/2}
V(r)\ast\tilde{\rho}\right|^2\right)^{\frac12}\left(\tilde{\mathbb{E}}
\int_{0}^{t}\int_{\mathbb{T}^{d}}|\nabla\sqrt{\tilde{\rho}^k}|^2\right)^{\frac12}\\
\notag
	\le&c\left(\tilde{\mathbb{E}}\int_{0}^{t}\int_{\mathbb{T}^{d}}
\left|(\partial_{\xi}\psi_{j})(x,\tilde{\rho}^k)(\tilde{\rho}^k)^{3/2}
V(r)\ast\tilde{\rho}^k-(\partial_{\xi}\psi_{j})(x,\tilde{\rho})(\tilde{\rho})^{3/2}
V(r)\ast\tilde{\rho}\right|^2\right)^{\frac12}\\
\label{eq-8.32}
\lesssim&\left(b^{k}_{121}+b^{k}_{122}\right)^{\frac12},
	\end{align}
where
\begin{align*}
b^{k}_{121}:=&2\tilde{\mathbb{E}}\int_{0}^{t}\int_{\mathbb{T}^{d}}\left|(\partial_{\xi}\psi_{j})(x,\tilde{\rho}^k)(\tilde{\rho}^k)^{3/2}V(r)\ast(\tilde{\rho}^k-\tilde{\rho})\right|^2,\\
b^{k}_{122}:=&2\tilde{\mathbb{E}}\int_{0}^{t}\int_{\mathbb{T}^{d}}\left|\left((\partial_{\xi}\psi_{j})(x,\tilde{\rho}^k)(\tilde{\rho}^k)^{3/2}-(\partial_{\xi}\psi_{j})(x,\tilde{\rho})(\tilde{\rho})^{3/2}\right)V(r)\ast\tilde{\rho}\right|^2.
	\end{align*}
	For the term $b^{k}_{121}$, by using H\"{o}lder's inequality and convolutional Young's inequality, we deduce that
	\begin{align*}
	b^{k}_{121}\le& 2M^3\|\partial_{\xi}\psi_{j}\|^2_{L^{\infty}(\mathbb{R})}
\tilde{\mathbb{E}}\int_{0}^{T}\|V(r)\ast(\tilde{\rho}^k-\tilde{\rho})\|_{L^{2}(\mathbb{T}^d)}^2
\\
\le &2M^3\|\partial_{\xi}\psi_{j}\|^2_{L^{\infty}(\mathbb{R})}\tilde{\mathbb{E}}\int_{0}^{T}\|V(r)\|^2_{L^{2}(\mathbb{T}^d;\mathbb{R}^d)}\|\tilde{\rho}^k-\tilde{\rho}\|^2_{L^{1}(\mathbb{T}^d)}\\
	\le&2M^3\|\partial_{\xi}\psi_{j}\|^2_{L^{\infty}(\mathbb{R})}\left(\int_{0}^{T}\|V\|^{p^{*}}_{L^{2}(\mathbb{T}^d;\mathbb{R}^d)}\right)^{\frac{2}{p^{*}}}\left(\tilde{\mathbb{E}}\int_{0}^{T}\|\tilde{\rho}^k-\tilde{\rho}\|^{\frac{2p^{*}-4}{p^{*}}}_{L^{1}(\mathbb{T}^d)}\right)^{\frac{p^{*}}{p^{*}-2}}.
	\end{align*}
Since $\|\tilde{\rho}^k-\tilde{\rho}\|_{L^{1}(\mathbb{T}^d)}\le2\|\hat{\rho}\|_{L^{1}(\mathbb{T}^d)}$, by the dominated convergence theorem and  (\ref{eq-8.19}), we have
	\begin{align}\label{eq-8.36}
	\lim_{k\to\infty}b^{k}_{121}=0.
	\end{align}
Note that
	\begin{align*}	
\left|\left((\partial_{\xi}\psi_{j})(x,\tilde{\rho}^k)
(\tilde{\rho}^k)^{3/2}-(\partial_{\xi}\psi_{j})
(x,\tilde{\rho})(\tilde{\rho})^{3/2}\right)V\ast\tilde{\rho}\right|^2
\le4M^3\|\partial_{\xi}\psi_{j}\|_{L^{\infty}(\mathbb{R})}|V\ast\tilde{\rho}|^2,
	\end{align*}
	and
	\begin{align*} \tilde{\mathbb{E}}\int_{0}^{T}\int_{\mathbb{T}^{d}}|V(r)\ast\tilde{\rho}|^2\le\|\hat{\rho}\|^2_{L^{1}(\mathbb{T}^d)}\int_{0}^{T}\|V(r)\|^2_{L^{2}(\mathbb{T}^d;\mathbb{R}^d)}<\infty.
	\end{align*}
By (\ref{eq-8.19}) and applying the dominated convergence theorem, for every $j\in\mathbb{N}$ and $t\in[0,T]$, it yields that up to a subsequence
	\begin{align}\label{eq-8.37}
	\lim_{k\to\infty}b^{k}_{122}=0.
	\end{align}
	Based on \eqref{eq-8.32}-\eqref{eq-8.37}, we deduce that for every $j\in\mathbb{N}$ and $t\in[0,T]$,
	\begin{align}\label{eq-8.38}
	\lim_{k\to\infty}b^{k}_{12}=0.
	\end{align}
	Combining \eqref{eq-8.28}, \eqref{eq-8.29}, \eqref{eq-8.31} and \eqref{eq-8.38}, we get that for every $j\in\mathbb{N}$ and $t\in[0,T]$,
	\begin{align}\label{eq-8.39}
	\lim_{k\to\infty}b^{k}_{1}\le\lim_{k\to\infty}(b^{k}_{11}+b^{k}_{12}+b^{k}_{13})=0.
	\end{align}
	It remains to handle the term $b^{k}_2$. Clearly, we have
	\begin{align}\notag
b^{k}_2\leq &\tilde{\mathbb{E}}\left|\int_{0}^{t}\int_{\mathbb{T}^{d}}
(\nabla\psi_{j})(x,\tilde{\rho}^k)\cdot\tilde{\rho}^k(V_{\gamma_k}(r)-V(r))
\ast\tilde{\rho}^k\right|
+\tilde{\mathbb{E}}\left|\int_{0}^{t}\int_{\mathbb{T}^{d}}(\nabla\psi_{j})(x,\tilde{\rho}^k)\cdot\tilde{\rho}^k V(r)\ast(\tilde{\rho}^k-\tilde{\rho})\right|\\  \notag
+&\tilde{\mathbb{E}}\left|\int_{0}^{t}\int_{\mathbb{T}^{d}}\left((\nabla\psi_{j})(x,\tilde{\rho}^k)\tilde{\rho}^k-(\nabla\psi_{j})(x,\tilde{\rho})\tilde{\rho}\right)\cdot V(r)\ast\tilde{\rho}\right|\\
\label{eq-8.40}
=:&b^{k}_{21}+b^{k}_{22}+b^{k}_{23}.
\end{align}
For the term $b^{k}_{21}$, by  H\"{o}lder's inequality, convolutional Young's inequality and $\lim\limits_{k\to\infty}\| V_{\gamma_k}-V\|_{L^1([0,T];L^{p}(\mathbb{T}^d;\mathbb{R}^d))}=0$, for every $j\in\mathbb{N}$ and $t\in[0,T]$, as $k\rightarrow\infty$,
	\begin{align}\label{eq-8.41}
b^{k}_{21}\le M\|\nabla\psi_{j}\|_{L^{\infty}(\mathbb{R})}\|\hat{\rho}\|_{L^{1}(\mathbb{T}^d)}\| V_{\gamma_k}-V\|_{L^1([0,T];L^{p}(\mathbb{T}^d;\mathbb{R}^d))}\rightarrow0.
	\end{align}
	For the term $b^{k}_{22}$, in view of the property of $\psi_{j}$, H\"{o}lder's inequality and (\ref{eq-8.19}), as $k\rightarrow \infty$, we get
	\begin{align}\notag
	b^{k}_{22}\le&\|\nabla\psi_{j}\|_{L^{\infty}(\mathbb{R})}M\tilde{\mathbb{E}}\int_{0}^{T}\| V(r)\ast(\tilde{\rho}^k-\tilde{\rho})\|_{L^{p}(\mathbb{T}^d)}\\ \notag
\le&\|\nabla\psi_{j}\|_{L^{\infty}(\mathbb{R})}M\Big(\int_{0}^{T}\| V(r)\|^{p^{*}}_{L^{p}(\mathbb{T}^d;\mathbb{R}^d)}\Big)^{\frac{1}{p^{*}}}
\Big(\tilde{\mathbb{E}}\int_{0}^{T}
\|\tilde{\rho}^k-\tilde{\rho}\|^{\frac{p^{*}}{p^{*}-1}}_{L^{1}(\mathbb{T}^d)}\Big)
^{\frac{p^{*}-1}{p^{*}}}\\ \label{eq-8.42}
	\le&\|\nabla\psi_{j}\|_{L^{\infty}(\mathbb{R})}Mc(\|\hat{\rho}\|_{L^1(\mathbb{T}^d)})\| V\|_{L^{p^{*}}\left([0,T];L^p(\mathbb{T}^d;\mathbb{R}^d)\right)}\Big(\tilde{\mathbb{E}}\int_{0}^{T}\|\tilde{\rho}^k-\tilde{\rho}\|_{L^{1}(\mathbb{T}^d)}\Big)^{\frac{p^{*}-1}{p^{*}}}
\rightarrow  0.
	\end{align}
	For the last term $b^{k}_{23}$, by H\"{o}lder's and convolutional Young's inequalities, as $k\rightarrow \infty$, it holds that
	\begin{align}
	\notag \left|\left((\nabla\psi_{j})(x,\tilde{\rho}^k)\tilde{\rho}^k-(\nabla\psi_{j})(x,\tilde{\rho})\tilde{\rho}\right)\cdot V(r)\ast\tilde{\rho}\right|\le2\|\nabla\psi_{j}\|_{L^{\infty}(\mathbb{R})}M|V\ast\tilde{\rho}|.
	\end{align}
	Applying the dominated convergence theorem, for every $j\in\mathbb{N}$ and $t\in[0,T]$, we have that up to a subsequence
	\begin{align}\label{eq-8.43}
	\lim_{k\to\infty}b^{k}_{23}=0.
	\end{align}
 Based on \eqref{eq-8.40}-\eqref{eq-8.43}, we deduce that for every $j\in\mathbb{N}$ and $t\in[0,T]$,
	\begin{align}\label{eq-8.44}
	\lim_{k\to\infty}b^{k}_{2}=0.
	\end{align}
	According to \eqref{eq-8.27}, \eqref{eq-8.39} and \eqref{eq-8.44}, we have that for every $j\in\mathbb{N}$ and $t\in[0,T]$,
	\begin{align}
	\lim_{k\to\infty}\tilde{\mathbb{E}}\left|\int_{0}^{t}\int_{\mathbb{T}^{d}}\psi_{j}(x,\tilde{\rho}^k)\nabla\cdot(\tilde{\rho}^kV(r)\ast\tilde{\rho}^k)-\psi_{j}(x,\tilde{\rho})\nabla\cdot(\tilde{\rho}V(r)\ast\tilde{\rho})\right|=0.\notag
	\end{align}
Taking $k\rightarrow \infty$ on both sides of (\ref{kk-68}), with the aid of the  above estimates and (\ref{kk-70}), we conclude that
for every $t\in[0,T]$ and $j\in\mathbb{N}$, the kinetic function $\tilde{\chi}$ of $\tilde{\rho}$ satisfies \eqref{eq-8.23}.

\noindent $\mathbf{Step\ 5.\ Properties\ of\ \tilde{q}.}$ We will prove that $\tilde{q}$ satisfies \eqref{eq-2.7} in Definition \ref{def-2.4}.
 For any $M>0$, set
\begin{align*}
\tilde{\theta}_M(\xi):=\mathbf{1}_{[M,M+1]}(\xi),\quad \tilde{\Theta}_M(\xi):=\int^{\xi}_{0}\int^{r}_{0}\tilde{\theta}_M(s)\mathrm{d}s\mathrm{d}r.
\end{align*}
Following a method similar to that in  \textbf{Step 4}, we first use smooth approximation and apply  It\^{o} formula to $\tilde{\Theta}_M(\rho^{n_k,\gamma_k})$. Since $\tilde{q}^k$ and $q^{n_k,\gamma_k}$ have the same law on $\mathcal{M}(\mathbb{T}^{d}\times\mathbb{R}\times[0,T])$, and $\tilde{\rho}^k$ and $\rho^{n_k,\gamma_k}$ also have same distribution, it follows that as  $k\rightarrow \infty$, for every $M>0$,
\begin{align*}
&\tilde{\mathbb{E}}\int^T_0{\color{black}\int_0^{\infty}}\int_{\mathbb{T}^d}\tilde{\theta}_M(\xi)
\mathrm{d}\tilde{q}\\
=&\tilde{\mathbb{E}}{\color{black}\int_0^{\infty}}\int_{\mathbb{T}^d}\bar{\chi}(\hat{\rho})\tilde{\Theta}_M'(\xi)-\tilde{\mathbb{E}}{\color{black}\int_0^{\infty}}\int_{\mathbb{T}^d}\tilde{\chi}(x,\xi,T)\tilde{\Theta}_M'(\xi)+\tilde{\mathbb{E}}\int^T_0\int_{\mathbb{T}^d}\tilde{\theta}_M(\tilde{\rho})\left(\tilde{\rho}V\ast\tilde{\rho}\right)\cdot\nabla\tilde{\rho}+\frac12\tilde{\mathbb{E}}\int^T_0\int_{\mathbb{T}^d}\tilde{\theta}_M(\tilde{\rho})F_3\tilde{\rho}\\
\le&\tilde{\mathbb{E}}{\color{black}\int_0^{\infty}}\int_{\mathbb{T}^d}\mathbf{1}_{[M,M+1]}(\hat{\rho})+\tilde{\mathbb{E}}{\color{black}\int_0^{\infty}}\int_{\mathbb{T}^d}\mathbf{1}_{[M,M+1]}(\tilde{\rho}(T))
+2(M+1)^{3/2}\tilde{\mathbb{E}}\int^T_0\int_{\mathbb{T}^d}\mathbf{1}_{[M,M+1]}(\tilde{\rho})\nabla\sqrt{\tilde{\rho}}\cdot V(r)\ast\tilde{\rho}\\
&+\frac{M+1}{2}\tilde{\mathbb{E}}\int_{0}^{T}\int_{\mathbb{T}^{d}}F_{3}\mathbf{1}_{[M,M+1]}(\tilde{\rho}).
\end{align*}
By convolutional Young's inequality and Proposition \ref{pro-5.4}, it gives that  $\nabla\sqrt{\tilde{\rho}}\cdot V\ast\tilde{\rho}$ is integrable in $L^1(\tilde{\Omega}\times [0,T]\times \mathbb{T}^d)$. Then, by \cite[Lemma 7]{FG23},
it follows that
\begin{align}\label{kineticprp}
\liminf_{M\to \infty}\tilde{\mathbb{E}}\tilde{q}\Big([0,T]\times \mathbb{T}^d\times [M,M+1]\Big)=0.
\end{align}

\noindent $\mathbf{Step\ 6.\ Recovering\ the\ equation\ almost\ everywhere\ in\ time}.$
Denote by $\mathcal{A}\subset [0,T]$ the atoms of the kinetic measure $\tilde{q}$ in time,
\begin{align*}
  \mathcal{A}=\{t\in [0,T]: \tilde{q}(\{t\}\times \mathbb{T}^d\times \mathbb{R})\neq0\}.
\end{align*}
Since $\tilde{q}$ is almost surely a Radon measure on $\mathbb{T}^d\times [0,T]\times [0,\infty)$ satisfying  (\ref{kineticprp}), the set $\mathcal{A}$
is almost surely at most countable. It then follows from the above \textbf{Steps 1-5} that there almost surely exists a
random set of full measure $\mathcal{C}\subset [0,T]\setminus\mathcal{A}$ such that, for every $t\in \mathcal{C}$ and $\psi\in C^{\infty}_c(\mathbb{T}^d\times (0,\infty))$,
\begin{align}\label{r-2}
&\left.{\color{black}\int_0^{\infty}}\int_{\mathbb{T}^{d}}\tilde{\chi}(x,\xi,r)\psi(x,\xi)\right|_{r=0}^{r=t}
=		-\int_{0}^{t}\int_{\mathbb{T}^{d}}\nabla\tilde{\rho}\cdot(\nabla\psi)\left(x,\tilde{\rho}\right)-\frac{1}{2}\int_{0}^{t}\int_{\mathbb{T}^{d}}F_1\left[\sigma^{\prime}\left(\tilde{\rho}\right)\right]^{2}\nabla\tilde{\rho}\cdot(\nabla\psi)\left(x,\tilde{\rho}\right)\notag\\ &-\frac{1}{2}\int_{0}^{t}\int_{\mathbb{T}^{d}}\sigma\left(\tilde{\rho}\right)\sigma^{\prime}\left(\tilde{\rho}\right)F_{2}\cdot(\nabla\psi_{j})\left(x,\tilde{\rho}\right)+\frac{1}{2}\int_{0}^{t}\int_{\mathbb{T}^{d}}\left(\partial_{\xi}\psi\right)\left(x,\tilde{\rho}\right)\sigma\left(\tilde{\rho}\right)\sigma^{\prime}\left(\tilde{\rho}\right)\nabla\tilde{\rho}\cdot F_{2}\notag\\
	&+\frac{1}{2}\int_{0}^{t}\int_{\mathbb{T}^{d}}F_{3}\sigma^{2}\left(\tilde{\rho}\right)\left(\partial_{\xi}\psi\right)\left(x,\tilde{\rho}\right)
-\int^t_0\int_{\mathbb{T}^d}\psi(x,\tilde{\rho})
	\nabla\cdot(\sigma(\tilde{\rho})\mathrm{d}\tilde{W}^F)\notag\\
	&-\int_{0}^{t}{\color{black}\int_0^{\infty}}\int_{\mathbb{T}^{d}}\partial_{\xi}\psi(x,\xi)\mathrm{d}\tilde{q}-\int_{0}^{t}\int_{\mathbb{T}^{d}}\psi(x,\tilde{\rho})\nabla\cdot(\tilde{\rho}V(r)\ast\tilde{\rho}).
\end{align}
In this position, we conclude that there exists a stochastic basis $(\tilde{\Omega},\tilde{\mathcal{F}},\{\tilde{\mathcal{F}}(t)\}_{t\in[0,T]},\tilde{\mathbb{P}})$, a trace-class Brownian motion $\tilde{W}^{F}$ on  $L^2(\mathbb{T}^{d})$, and a process $\tilde{\rho}\in L^1(\tilde{\Omega};L^1([0,T]\times \mathbb{T}^d))$, which is a renormalized kinetic solution of \eqref{eq-8.1} in the sense of Definition \ref{def-2.4} with the equation (\ref{eq-2.6}) holds for almost every $t\in [0,T]$ and $\tilde{\rho}(\cdot,0)=\hat{\rho}$.
 We complete the proof of (i).

\noindent \textbf{The proof of (ii).}\ Under further Assumption (A2), it follows that Theorem \ref{the-uniq} holds. Then we  can
 follow the method of \cite[Theorem 5.25]{FG21} to deduce that the above $\tilde{\rho}$ admits an almost surely representative in $\mathrm{C}\left([0,T];L^{1}(\mathbb{T}^{d})\right)$. As a result,
the kinetic measure $\tilde{q}$ has no atoms in time, thereby $\tilde{\rho}$ satisfies the equality (\ref{eq-2.6}) for every $t\in [0,T]$. To complete the proof of (ii),
it remains to prove the existence of a renormalized solution in the probabilistic strong sense.
 Thanks to the pathwise uniqueness obtained by Theorem \ref{the-uniq}, by adopting a methodology analogous to that in \cite[Theorem 5.25, Conclusion]{FG21}, we obtain that the limiting joint distribution $\mu$ is supported on the diagonal, i.e., $\mu({(x,y)\in\bar{X}\times\bar{X}: x=y}) = 1$. According to Remark \ref{remark-6.2}, the conditions of Lemma \ref{lem-8.1} are met. Thus, for the original solutions $\left\{\rho^{n,\gamma}\right\}_{ n\in\mathbb{N},\gamma \in(0,1)}$ on the original probability space $(\Omega, \mathcal{F}, \mathbb{P})$, after passing to a subsequence $\gamma_{k}\to0, n_{k}\to\infty$, there exists a random variable $\rho\in L^{1}\left(\Omega\times[0,T];L^{1}\left(\mathbb{T}^{d}\right)\right)$ such that $\left\{\rho^{n_{k},\gamma_{k}}\right\}_{k\in\mathbb{N}}$ converge to $\rho$ in probability. A further subsequence still denoted by $\{n_{k},\gamma_{k}\}_{k\in\mathbb{N}}$ yields $\rho^{n_{k},\gamma_{k}}\to\rho$ almost surely. A simplified variant of the previous argument confirms that $\rho$ satisfies the criteria of a renormalized kinetic solution of \eqref{eq-8.1} in the sense of Definition \ref{def-2.4} on $(\Omega, \mathcal{F}, \mathbb{P})$, hence $\rho$ is a probabilistically strong solution to \eqref{eq-8.1}. The proof is completed with the assistance of Theorem \ref{the-uniq}.
\end{proof}

\section{Well-posedness of the fluctuating Ising-Kac-Kawasaki equation}\label{sec-7}
In this section, we consider the fluctuating Ising-Kac-Kawasaki equation in any dimension $d\ge1$, which is of the following form
\begin{equation}\label{mainspde2}
\mathrm{d}\rho=\Delta\rho \mathrm{d}t-\nabla\cdot[(1-\rho^2) \nabla J\ast\rho]\mathrm{d}t-\nabla\cdot(\sqrt{1-\rho^2}\circ \mathrm{d}W^F),
\end{equation}
where $W^F$ is defined by (\ref{kk-43}) and $J$ denotes the Kac potential satisfying $\nabla J\in C^{\infty}(\mathbb{T}^d;\mathbb{R}^d)$. As discussed in the Introduction part, (\ref{mainspde2}) exhibits the same nonlinear fluctuations as Kawasaki dynamics for the Ising-Kac model. Thus, the well-posedness of (\ref{mainspde2}) plays a fundamental role in studying nonlinear fluctuations of Kawasaki dynamics.

Due to the structural resemblance between \eqref{mainspde2} and the Dean-Kawasaki equation \eqref{mainspde}, we  can employ a similar proof strategy as used in the previous sections to establish the well-posedness of the renormalized kinetic solution to \eqref{mainspde2}. However,  there are still some challenges that need to be handled carefully, for instance, the lack of $L^1(\mathbb{T}^d)$-norm preservation and two singularities appear in the derivative of the diffusion coefficient $\sqrt{1-\xi^2}$.

To maintain the clarity, we only present the proof of entropy estimate and the tightness (as seen in Theorem \ref{the-existence}), which are two essential ingredients in proving the well-posedness of \eqref{mainspde2}.

\subsection{Renormalized\ kinetic\ solution\ to\ \eqref{mainspde2}}
 The Stratonovich equation \eqref{mainspde2} is formally equivalent to the  It\^o equation
\begin{equation}\label{eqq1}
\mathrm{d}\rho=\Delta\rho \mathrm{d}t-\nabla\cdot[(1-\rho^2)\nabla J\ast\rho)\mathrm{d}t-\nabla\cdot(\sqrt{1-\rho^2} \mathrm{d}W^F)+\frac{1}{2}\nabla\cdot\Big(F_1\frac{\rho^2}{1-\rho^2}\nabla\rho-\rho F_2\Big)\mathrm{d}t.
\end{equation}

For a given bounded solution $\rho$ of \eqref{eqq1}, we define the kinetic function $\chi: \mathbb{T}^{d}\times\mathbb{R}\times[0,T]\to\{-1,1\}$ of $\rho$ as
\begin{align*}
\chi(x,\xi,t):=\mathbf{1}_{\{0<\xi<\rho(x,t)\}}-{\color{black}\mathbf{1}_{\{\rho(x,t)<\xi<0\}}}.
\end{align*}
Formally, we have the following the distributional equalities
\begin{align}
\nabla\chi=2\delta_{0}(\xi-\rho)\nabla\rho, \quad \partial_{\xi}\chi=2\delta_{0}(\xi)-2\delta_{0}(\xi-\rho)\quad {\rm{and}}\ \rho=\frac12\int_{\mathbb{R}}\chi \mathrm{d}\xi.
\end{align}
Then an informal application of It\^o's formula suggests that the kinetic function $\chi$ of $\rho$ satisfies the equation
\begin{align}\label{eqq2}
\partial_t\chi=&2\nabla\cdot\left(\delta_{0}(\xi-\rho)\nabla\rho\right)+\nabla\cdot\left(\delta_{0}(\xi-\rho)\left(F_1\frac{\xi^2}{1-\xi^2}\nabla\rho-\xi F_2\right)\right)-2\delta_{0}(\xi-\rho)\nabla\cdot[(1-\rho^2) \nabla J\ast\rho]\notag\\
&+2\partial_{\xi}q-\partial_{\xi}\left(\delta_{0}(\xi-\rho)\left(-\xi\nabla\rho\cdot F_2+(1-\xi^2) F_3\right)\right)-2\delta_{0}(\xi-\rho)\nabla\cdot(\sqrt{1-\rho^2} \dot{W}^F),
\end{align}
where\  $q=\delta_0(\xi-\rho)|\nabla\rho|^2$ is the parabolic defect measure.

Similar to the Dean-Kawasaki equation, before defining a renormalized kinetic solution to \eqref{mainspde2}, we have to specify the kinetic measure.
\begin{definition}\label{def-7.2}
	Let $(\Omega,\mathcal{F},\mathbb{P})$ be a probability space with a filtration $(\mathcal{F}_{t})_{t\in[0,\infty)}.$ A kinetic measure is a measurable map $q$ from $\Omega$ to the space of nonnegative, finite Radon measures on $\mathbb{T}^{d}\times{[-1,1]}\times[0,T]$ that satisfies the property that the process
	\begin{align*}
		(\omega,t)\in \Omega\times[0,T]\to
		\int^t_0{\color{black}\int_{-1}^1}\int_{\mathbb{T}^d}\psi(x,\xi)\mathrm{d}q(x,\xi,r)
	\end{align*}
	is $\mathcal{F}_t-$predictable, for every $\psi\in C^{\infty}_c(\mathbb{T}^d\times [-1,1])$.
\end{definition}

We will prove the well-posedness of \eqref{mainspde2} for initial data with finite mathematical entropy.
 Let
\begin{align}\label{psi}
	\psi(\xi)=\frac{1}{2}\log\frac{1+\xi}{1-\xi},
\end{align}
and let $\Psi:[-1,1]\rightarrow\mathbb{R}$ be a function defined by
\begin{align}\label{eq-psi}
\Psi(\xi)= \frac{1}{2}\Big[(\xi+1)\log(\xi+1)-(\xi+1)+(1-\xi)\log(1-\xi)-(1-\xi)\Big].
\end{align}
A direct computation shows that $\Psi'(\xi)=\psi(\xi)$.
Define
\begin{align}\label{eqq3}
\overline{\text{{\rm{Ent}}}}(\mathbb{T}^{d})=\Big\{\rho: -1\leq\rho\leq1\ a.e.,\ {\rm{and}}\ \int_{\mathbb{T}^{d}}\Psi(\rho(x))\mathrm{d}x<\infty\Big\}.
\end{align}

\begin{definition}(Renormalized kinetic solution)\label{def-7.3}
	Let $\rho_0\in  \overline{\text{{\rm{Ent}}}}(\mathbb{T}^{d})$. A renormalized kinetic solution of (\ref{mainspde2}) with initial datum $\rho(\cdot,0)=\rho_0$ is an almost surely continuous $L^2(\mathbb{T}^d;[-1,1])$-valued $\mathcal{F}_t$-predictable function $\rho\in L^2\left(\Omega\times[0,T];L^2(\mathbb{T}^d;[-1,1])\right)$ that satisfies the following properties.
	\begin{enumerate}
		\item Essentially bounded: almost surely for every $t\in[0,T]$,
		\begin{equation}\label{eqq4}
		\rho(\cdot,t)\in[-1,1],\ a.e.
		\end{equation}
		\item Regularity of $\sqrt{1-\rho^2}$: there exists a constant $c\in(0,\infty)$ depending on $T,\rho_0, J$ and $d$ such that
		\begin{equation}\label{eqq5}
		\mathbb{E}\int^T_0\int_{\mathbb{T}^d}\Big[|\nabla\sqrt{1-\rho^2}|^2+|\nabla \rho|^2\Big]\mathrm{d}x\mathrm{d}s\le c(T,d,\rho_0,J).
		\end{equation}
		Furthermore, there exists a finite nonnegative kinetic measure $q$ satisfying the following properties.
		\item Regularity: almost surely
		\begin{align}\label{eqq6}
		\delta_{0}(\xi-\rho)|\nabla \rho|^{2}\le q\quad {\rm{on}}\ \mathbb{T}^d\times[-1,1]\times[0,T].
		\end{align}
\item Optimal regularity: the measure $\mu$ defined by
	\begin{align}\label{mu}
		\mathrm{d}\mu=\left(1-\xi^2\right)^{-1}\mathrm{d}q\ \text{is finite on}\ \mathbb{T}^d\times(-1,1)\times[0,T].
	\end{align}

		\item The equation: for every $\varphi\in\mathrm{C}_{c}^{\infty}\left(\mathbb{T}^{d}\times(-1,1)\right)$, almost surely for every $t\in[0,T]$,
		\begin{align}\label{eqq8}
		&{\color{black}\int_{-1}^1}\int_{\mathbb{T}^d}\chi(x,\xi,t)\varphi(x,\xi)={\color{black}\int_{-1}^1}\int_{\mathbb{T}^d}\bar{\chi}(\rho_0)\varphi(x,\xi)-2\int_0^t\int_{\mathbb{T}^d}\nabla\rho\cdot(\nabla\varphi)(x,\rho)\notag\\
		&-\int_0^t\int_{\mathbb{T}^d}F_1(x)\frac{\rho^2}{1-\rho^2}\nabla\rho\cdot(\nabla\varphi)(x,\rho)+\int_0^t\int_{\mathbb{T}^d}\rho F_2(x)\cdot(\nabla\varphi)(x,\rho)\notag\\
		&-2\int_0^t{\color{black}\int_{-1}^1}\int_{\mathbb{T}^d}\partial_{\xi}\varphi(x,\xi)dq+\int_0^t\int_{\mathbb{T}^d}\left(-\rho\nabla\rho\cdot F_2(x)+(1-\rho^2)F_3(x)\right)(\partial_{\xi}\varphi)(x,\rho)\notag\\
		&-2\int_0^t\int_{\mathbb{T}^d}\varphi(x,\rho)\nabla\cdot((1-\rho^2)\nabla J\ast\rho)-2\int_0^t\int_{\mathbb{T}^d}\varphi(x,\rho)\nabla\cdot(\sqrt{1-\rho^2} \mathrm{d}W^F(r)),
		\end{align}
		where $\bar{\chi}(\rho_0)(x,\xi):={\color{black}\mathbf{1}_{\{0<\xi<\rho_0(x)\}}-\mathbf{1}_{\{\rho_0(x)<\xi<0\}}}$.
	\end{enumerate}	
\end{definition}

In order to prove the existence of renormalized kinetic solution to \eqref{eqq1}, we also need to introduce approximation equations with regularized coefficients. Similarly to \cite[Proposition 4.1]{DFG} and  \cite[Lemma 5.18]{FG21}, we propose a smooth sequence to approximate $\sqrt{1-\xi^2}$.

\begin{lemma}
	Let $s^{\frac12}:\mathbb{R}\rightarrow[0,1]$ be defined by
	\begin{align*}
		s^{\frac12}(\xi)=\sqrt{1-\xi^2}\  \text{if}\  \xi\in[-1,1]\ \text{and}\ s^{\frac12}(\xi)=0\ \text{if}\ \xi\notin [-1,1].
	\end{align*}
	Then there exists a sequence of smooth, compactly supported approximations
	$\{s^{\frac12,\eta}\}_{\eta\in(0,1)}$ satisfying
	\begin{align*}
		\lim_{\eta\rightarrow 0}\|s^{\frac{1}{2},\eta}-s^{\frac{1}{2}}\|_{L^{\infty}(\mathbb{R})}=0.
	\end{align*}
	Furthermore, $s^{\frac12,\eta}$ has the following properties.
	\begin{enumerate}
		\item[(1)] $s^{\frac12,\eta}\in C(\mathbb{R})\cap C^{\infty}(\mathbb{R})$ with $s^{\frac12,\eta}(1)=s^{\frac12,\eta}(-1)=0$ and $\left(s^{\frac12,\eta}\right)^{\prime}\in C^{\infty}_c((-1,1))$,
		\item[(2)] there exists $c\in(0,\infty)$ such that for every $\xi\in[-1,1]$,
		\begin{align}\label{eqq20}
			|s^{\frac12,\eta}(\xi)|\le c\sqrt{1-\xi^2} \text{ uniformly with respect to } \eta\in(0,1),
		\end{align}
		\item[(3)] for every $\delta\in(0,1)$, there exists $c_{\delta}\in(0,\infty)$ such that uniformly with respect to $\eta\in(0,1)$,
		\begin{align}\label{eqq21}		\left[\left(s^{\frac12,\eta}\right)^{\prime}(\xi)\right]^4\mathbf{1}_{\{-1+\delta\le\xi\le1-\delta\}}+\left|s^{\frac12,\eta}(\xi)\left(s^{\frac12,\eta}\right)^{\prime}(\xi)\right|^2\mathbf{1}_{\{-1+\delta\le\xi\le1-\delta\}}\le c_{\delta}.
		\end{align}
	\end{enumerate}
	
\end{lemma}


Now, for every $\eta\in(0,1)$, we consider an approximating equation of \eqref{eqq1},
\begin{align}\label{eq-5}
\partial_t\rho^{\eta}=&\Delta\rho^{\eta}-\nabla\cdot[(1-\left(\rho^{\eta}\right)^2)\nabla J\ast\rho^{\eta}]-\nabla\cdot\Big(s^{\frac{1}{2},\eta}(\rho^{\eta})\dot{W}^F\Big)\notag\\
&+\frac{1}{2}\nabla\cdot\Big(F_1\Big|(s^{\frac{1}{2},\eta})'(\rho^{\eta})\Big|^2\nabla\rho^{\eta}-s^{\frac{1}{2},\eta}(\rho^{\eta})(s^{\frac{1}{2},\eta})'(\rho^{\eta}) F_2\Big),
\end{align}
with $\rho^{\eta}(0)=\rho_{0}$. The weak solution of \eqref{eq-5} can be defined analogously to Definition \ref{def-5.2}.
\begin{remark}
Proceed similarly as in Theorem \ref{the-6.4}, we can show that if the initial value $\rho_{0}\in  [-1,1]$, then $\rho^{\eta}\in [-1,1]$ almost surely.
\end{remark}

\subsection{Entropy estimate}
 Recall that $\Psi(\cdot)$ and $\overline{\text{{\rm{Ent}}}}(\mathbb{T}^{d})$ are defined by  \eqref{eq-psi} and \eqref{eqq3}, respectively. We provide the following entropy dissipation estimate.
\begin{proposition}\label{pro-7.4}
Let $\rho_0\in \overline{\text{{\rm{Ent}}}}(\mathbb{T}^{d})$. For any $\eta\in(0,1)$, let $\rho^{\eta}$ be a weak solution of (\ref{eq-5}) 
with initial data $\rho^{\eta}(\cdot,0)=\rho_0$. Then there exists a constant $c\in(0,\infty)$ depending on  $T$ and $\|\nabla J\|_{L^1(\mathbb{T}^d;\mathbb{R}^d)}$ such that

	\begin{align}\label{entropy}
		\mathbb{E}\left[\sup _{t \in[0, T]} \int_{\mathbb{T}^d} \Psi(\rho^{\eta}(x, t))\right]+\mathbb{E}\int_0^T\int_{\mathbb{T}^d}  \frac{1}{1-\left(\rho^{\eta}\right)^2}|\nabla \rho^{\eta}|^2\leqslant \int_{\mathbb{T}^d} \Psi\left(\rho_0\right)+c.
	\end{align}

\end{proposition}

\begin{proof}
For the above $\Psi$, we introduce a sequence of smooth approximating functions denoted by $\Psi_\delta$ with $\delta \in(0,1)$. Here, we require that $\Psi_\delta(0)=\Psi(0)$ and $\Psi_\delta^{\prime}(\xi)=\frac{1}{2(1+\delta)} \log \left(\frac{1+\xi+\delta}{1-\xi+\delta}\right)$. Clearly, $\Psi_\delta^{''}(\xi)=\frac{1}{(1+\delta)^2-\xi^2}$. Applying It\^{o} formula, we deduce that almost surely for every $t \in[0, T]$,
\begin{align}\label{eqq13}
\left.\int_{\mathbb{T}^d} \Psi_{\delta}(\rho^{\eta}(x, r))\right|_{\mathrm{r}=0} ^{\mathrm{r}=t}+\int_0^t \int_{\mathbb{T}^d} \frac{1}{(1+\delta)^2-\left(\rho^{\eta}\right)^2}|\nabla \rho^{\eta}|^2=K_{t}^{\mathrm{ker}}+K_{t}^{\text {mart }}+K_{t}^{\mathrm{err}},
\end{align}
where  $K_{t}^{\mathrm{mart}}$ represents the term related to the martingale, and $K_{t}^{\text {err }}$ represents the term associated with the transformation of the It\^{o}-Stratonovich integral and the quadratic variation term from the  It\^{o} formula.  The specific forms and estimates of these two terms are similar to \cite{FG21}, thus we omit them. In the following, we only focus on the kernel term $K^{err}_t$ which is of the form
\begin{align*}
	K_{t}^{\mathrm{ker}}=\int_0^t \int_{\mathbb{T}^d} \frac{1}{(1+\delta)^2-\left(\rho^{\eta}\right)^2}((1-(\rho^{\eta})^2) \nabla J\ast \rho^{\eta}) \cdot \nabla \rho^{\eta}.
\end{align*}
Since $|\rho^{\eta}|\le1$ almost surely, it follows from the Young inequality that there exists a constant $c\in(0,\infty)$ such that
 \begin{align}\label{eqq12}
\mathbb{E}\sup _{t \in[0, T]}K_{t}^{\mathrm{ker}}\le& \mathbb{E}\int_0^T \int_{\mathbb{T}^d} \left|\frac{1}{(1+\delta)^2-\left(\rho^{\eta}\right)^2} \nabla \rho^{\eta} \cdot((1-(\rho^{\eta})^2) ) \nabla J\ast \rho^{\eta})\right|\notag \\
	\leq&\frac{1}{2} \mathbb{E} \int_0^T \int_{\mathbb{T}^d} \frac{1}{(1+\delta)^2-\left(\rho^{\eta}\right)^2}|\nabla \rho^{\eta}|^2+\frac12 \mathbb{E} \int_0^T \int_{\mathbb{T}^d} \frac{1}{(1+\delta)^2-\left(\rho^{\eta}\right)^2}\left|1-(\rho^{\eta})^2\right|^2|\nabla J\ast \rho^{\eta}|^2 \notag\\
	\leq &\frac{1}{2} \mathbb{E} \int_0^T \int_{\mathbb{T}^d} \frac{1}{(1+\delta)^2-\left(\rho^{\eta}\right)^2}|\nabla \rho^{\eta}|^2+c(\|\nabla J\|_{L^1(\mathbb{T}^d;\mathbb{R}^d)}).
\end{align}
Therefore,  there exists a constant $c\in(0,\infty)$  independent of $\delta \in(0,1)$ such that

	\begin{align*}
		&\mathbb{E}\left[\sup _{t \in[0, T]} \int_{\mathbb{T}^d} \Psi_\delta(\rho^{\eta}(x, t))\right]+\mathbb{E}\left[\int_0^T\int_{\mathbb{T}^d} \frac{1}{(1+\delta)^2-\left(\rho^{\eta}\right)^2}|\nabla \rho^{\eta}|^2\right] \leqslant \int_{\mathbb{T}^d} \Psi_\delta\left(\rho_0\right)+c(T,\|\nabla J\|_{L^1(\mathbb{T}^d;\mathbb{R}^d)}).
	\end{align*}
	The remaining proof can be done by a similar method as Proposition \ref{pro-5.4}, thus we get the result \eqref{entropy}.
\end{proof}

\begin{remark}
	Similarly to Lemma \ref{lem-2}, the weak derivative
	\begin{align*}
		\frac{1}{1-\left(\rho^{\eta}\right)^2}|\nabla \rho^{\eta}|^2&=\left|\nabla \sqrt{1-\left(\rho^{\eta}\right)^2}\right|^2+|\nabla \rho^{\eta}|^2
	\end{align*}
	holds for almost every $(x,t)\in\mathbb{T}^d\times[0,T]$. Therefore,  the entropy estimate given in \eqref{entropy} not only implies the regularity of $\sqrt{1-\rho^2}$, but also indicates that the kinetic measure is finite.
\end{remark}

\subsection{Tightness\ of\ approximating\ solutions}\label{sec-r1}
 {\black We aim to establish the tightness of the laws of the family $\{\rho^{\eta}\}_{\eta \in (0,1)}$ in the space $L^{2}\big([0,T];L^{2}(\mathbb{T}^{d};[-1,1])\big)$. Compared to the Dean-Kawasaki equation, the situation is similar in that the equation (\ref{eqq1}) also contains a singular term
$\nabla\cdot(F_1\frac{\rho^2}{1-\rho^2}\nabla\rho)$, but the difference lies in that this singular term has two singularities at $\rho=-1$ and $\rho=+1$. In this case, it seems that the approach in Section \ref{subsec-5.3} based on equivalent metric is difficult to be applied to the equation (\ref{eqq1}). Even if we modify the truncation functions to force the solution away from $-1$ and $+1$, there are difficulties in constructing an equivalent metric to the strong norm topology on $L^2([0,T];L^2(\mathbb{T}^d;[-1,1]))$.
Fortunately, this obstacle can be solved via
 a transformation that $v^{\eta}=1-(\rho^{\eta})^2$, where $\rho^{\eta}$ is the solution of approximating equation. A simple calculation shows that the singular term related to $v=:v^{\eta}$ is changed to $-\frac{1}{2}\nabla\cdot(F_1\frac{1-v}{v}\nabla v)-\frac{1}{4}F_1\frac{|\nabla v|^2}{v}$. Clearly, it has only one singularity at $v=0$, thus we can establish the tightness of $v^{\eta}=1-(\rho^{\eta})^2$ on $L^1([0,T];L^1(\mathbb{T}^d;[0,1]))$ by using the same method as in Section \ref{subsec-5.3}. With the aid of this result and the weak convergence of $\{\rho^{\eta}\}_{\eta \in (0,1)}$  in $L^2([0,T];L^2(\mathbb{T}^d;[-1,1]))$, we can conclude that $\{\rho^{\eta}\}_{\eta \in (0,1)}$ is tight on $L^2([0,T];L^2(\mathbb{T}^d;[-1,1]))$.

In accordance with Lemma \ref{lem-7.2} and Proposition \ref{pro-7.3}, we need to make estimates of $L^2([0, T] ; H^{1}(\mathbb{T}^d))$ and $W^{\beta, 1}\left([0, T] ; H^{-l}\left(\mathbb{T}^d\right)\right)$ norms for $v^{\eta}=1-(\rho^{\eta})^2$. Firstly, by
utilizing the entropy dissipation estimates established in Proposition \ref{pro-7.4}, together with the chain rule and the boundedness of $\rho^{\eta}$, we reach the following result.

 \begin{lemma}\label{lem-7.6}
	Let $\rho_0\in \overline{\text{{\rm{Ent}}}}(\mathbb{T}^{d})$. For any $\eta\in(0,1)$, let $\rho^{\eta}$ be a weak solution of (\ref{eq-5}) 
	with initial data $\rho^{\eta}(\cdot,0)=\rho_0$.  Then there exists a constant $c\in(0, \infty)$ depending on $ \delta,T,d, \|\nabla J\|_{L^1\left(T^d;\mathbb{R}^d\right)}$, and $\rho_0$ such that
	\begin{align*}
		\mathbb{E}\left[\left\|h_\delta(1-(\rho^{\eta})^2)\right\|_{L^2\left([0, T] ; H^{1}\left(\mathbb{T}^d\right)\right)}\right] \leqslant \lambda\left( \delta,T, d, \|\nabla J\|_{L^1\left(T^d;\mathbb{R}^d\right)}, \rho_0\right).
	\end{align*}
\end{lemma}
Now, let us focus on
the estimation of $W^{\beta,1}\big([0,T];H^{-l}(\mathbb{T}^d)\big)$ norm for $v^{\eta}=1-(\rho^{\eta})^2$. For every test function $\phi \in C^{\infty}(\mathbb{T}^d)$, applying It\^{o}'s formula to the quantity $\int_{\mathbb{T}^d} h_{\delta}\big(1 - (\rho^{\eta})^2\big) \phi$, we get the following identity in the sense of distribution:
 \begin{align}\notag
 &dh_{\delta}(1-(\rho^{\eta})^2)\\  \notag
 =&-2\nabla\cdot(h'_{\delta}(1-(\rho^{\eta})^2)\rho^{\eta}\nabla\rho^{\eta})dt-2[2h''_{\delta}(1-(\rho^{\eta})^2)(\rho^{\eta})^2-h'_{\delta}(1-(\rho^{\eta})^2)]|\nabla\rho^{\eta}|^2dt\\\notag
 &+2h'_{\delta}(1-(\rho^{\eta})^2)\rho^{\eta}\nabla\cdot[(1-(\rho^{\eta})^2)\nabla J\ast\rho^{\eta}]dt-\nabla\cdot\left[h'_{\delta}(1-(\rho^{\eta})^2)\rho^{\eta}|(s^{\frac{1}{2},\eta})'(\rho^{\eta}))|^2\nabla\rho^{\eta}F_1\right] dt\\ \notag
 &
 -[2h''_{\delta}(1-(\rho^{\eta})^2)(\rho^{\eta})^2-h'_{\delta}(1-(\rho^{\eta})^2)]|(s^{\frac{1}{2},\eta})'(\rho^{\eta})|^2|\nabla\rho^{\eta}|^2F_1dt\\ \notag
 &+h'_{\delta}(1-(\rho^{\eta})^2)\rho^{\eta}\nabla\cdot(s^{\frac{1}{2},\eta}(\rho^{\eta})(s^{\frac{1}{2},\eta})'(\rho^{\eta})F_2)dt\\ \notag
 &+[2h''_{\delta}(1-(\rho^{\eta})^2)(\rho^{\eta})^2-h'_{\delta}(1-(\rho^{\eta})^2)]\left[|\nabla s^{\frac{1}{2},\eta}(\rho^{\eta})|^2F_1+|s^{\frac{1}{2},\eta}(\rho^{\eta})|^2F_3+2\nabla s^{\frac{1}{2},\eta}(\rho^{\eta})s^{\frac{1}{2},\eta}(\rho^{\eta})F_2\right]dt\\
 \label{transform-eq}
 &+2h'_{\delta}(1-(\rho^{\eta})^2)\rho^{\eta}\nabla\cdot(s^{\frac{1}{2},\eta}(\rho^{\eta})dW^F).
 \end{align}
 By applying the method presented in Lemma \ref{lem-7.2} to equation (\ref{transform-eq}), and combining it with the estimates established in Lemma \ref{lem-7.6}, we obtain the following time-regularity estimates.
 \begin{lemma}\label{pro-7.7}
 	Let $\rho_0\in \overline{\text{{\rm{Ent}}}}(\mathbb{T}^{d})$. For  any $\eta\in(0,1)$, let $\rho^{\eta}$ be a weak solution of (\ref{eq-5}) 
 	with initial data $\rho^{\eta}(\cdot,0)=\rho_0$.  Then, for every $\beta \in(0,1 / 2)$ and $l>\frac{d}{2}+1$, there exists $\lambda \in(0, \infty)$ depending on $\delta, \beta, T,d, l,\|\nabla J\|_{L^1\left(T^d;\mathbb{R}^d\right)}$, and $\rho_0$ such that
 	\begin{align*}
 	\mathbb{E}\left[\left\|h_\delta(1-(\rho^{\eta})^2)\right\|_{W^{\beta, 1}\left([0, T] ; H^{-l}\left(\mathbb{T}^d\right)\right)}\right] \leqslant \lambda\left(\delta, \beta, T, d, l,\|\nabla J\|_{L^1\left(T^d;\mathbb{R}^d\right)}, \rho_0\right).
 	\end{align*}
 \end{lemma}

Based on Lemmas \ref{lem-7.6} and \ref{pro-7.7}, using the same method as Section \ref{subsec-5.3}, we get the tightness of the laws of $1-(\rho^{\eta})^2$ on $L^{1}\big([0,T];L^{1}(\mathbb{T}^{d};[0,1])\big)$. This, together with the weak convergence of $\{\rho^{\eta}\}_{\eta \in (0,1)}$  in $L^2([0,T];L^2(\mathbb{T}^d;[-1,1]))$, we conclude that $\{\rho^{\eta}\}_{\eta \in (0,1)}$ is tight on $L^2([0,T];L^2(\mathbb{T}^d;[-1,1]))$. It reads as follows.
\begin{proposition}\label{pro-tight}
	Let $\rho_0\in \overline{\text{{\rm{Ent}}}}(\mathbb{T}^{d})$. For  any $\eta\in(0,1)$, let $\rho^{\eta}$ be a weak solution of (\ref{eq-5}) 
	with initial data $\rho^{\eta}(\cdot,0)=\rho_0$.  Then the laws of $\left\{\rho^{\eta}\right\}_{\eta\in(0,1)}$ are tight on $L^{2}\left([0,T];L^{2}(\mathbb{T}^{d};[-1,1]\right))$.
\end{proposition}
}
	
\subsection{Well-posedness of  \eqref{mainspde2}}	
 Leveraging the previous entropy estimate and tightness of the approximation equation, we can prove the well-posedness of the renormalized kinetic solution for the fluctuating Ising-Kac-Kawasaki equation \eqref{mainspde2}. The proof is achieved through an analogous approach to the well-posedness of  the Dean-Kawasaki equation \eqref{mainspde}.
\begin{theorem}\label{the-7.11}
	For any spatial dimension $d\ge1$, suppose that $\nabla J\in C^{\infty}(\mathbb{T}^d;\mathbb{R}^d)$. Let $\rho_0\in\overline{{\rm{Ent}}}\left(\mathbb{T}^{d}\right)$. Then there exists {\color{black}a unique probabilistically strong  renormalized kinetic solution} to (\ref{mainspde2}) in the sense of Definition \ref{def-7.3} with initial data $\rho_0$.
	\end{theorem}
\begin{proof}
For the uniqueness of renormalized kinetic solutions to (\ref{mainspde2}), since (\ref{mainspde2}) has two singularities at $\pm1$, we need to introduce a new cutoff function in the velocity variable compared with the Dean-Kawasaki equation \eqref{mainspde}. Concretely, for each $\beta\in(0,\frac14)$, define a smooth function $\zeta_{\beta}:\mathbb{R}\to[0,1]$ which satisfies $\zeta_{\beta}=0$ if $\xi\le-1+\beta$ or $\xi\ge1-\beta$, $\zeta_{\beta}=1$ if $-1+2\beta\le\xi\le1-2\beta$, and the condition $|\zeta_{\beta}'|\le c/\beta$ for some constant $c\in(0,\infty)$ independent of $\beta$.  Then, with the help of $\zeta_{\beta}$ and \eqref{mu},  we can apply a method akin to that used in \cite[Theorem 4.6]{FG21}, \cite[Theorem 2.6]{DFG} and Theorem \ref{the-uniq} to establish the uniqueness. Since the entropy estimate from Proposition \ref{pro-7.4} provides the regularity properties of the solution, the proof of the existence of renormalized kinetic solutions to (\ref{mainspde2}) can closely parallel that of \cite[Theorem 5.25]{FG21}, \cite[Theorem 2.15]{DFG} and Theorem \ref{the-existence}.
\end{proof}
	\begin{remark}
Different from the Dean-Kawasaki equation \eqref{mainspde}, the well-posedness of (\ref{mainspde2}) holds for all spatial dimensions $d\geq 1$. The reason is that we have assumed the condition $\nabla J\in C^{\infty}(\mathbb{T}^d;\mathbb{R}^d)$, which can guarantee the integrability of the kernel term instead of applying the interpolation inequalities.
	\end{remark}

\noindent{\bf  Acknowledgements}\quad
The first author is supported by National Key R\&D Program of China (No. 2020YFA0712700), National Natural Science Foundation of China (No. 12090010, 12090014, 12471138) and Key Laboratory of Random Complex Structures and Data Science, Academy of Mathematics and Systems Science, Chinese Academy of Sciences (No. 2008DP173182). The second author is supported by the Deutsche Forschungsgemeinschaft (DFG, German Research Foundation) via IRTG 2235 - Project Number 282638148. The third author is supported by Open Foundation of the State Key Laboratory of Mathematical Sciences (Grant No. 09), Beijing Institute of Technology Research Fund Program for Young Scholars and MIIT Key Laboratory of Mathematical Theory and Computation in Information Security.
 
 The authors are grateful to the anonymous reviewers for their valuable suggestions and insightful feedback that have improved the clarity and quality of this work.

\bibliographystyle{alpha}
\bibliography{wang-wu-zhang.bib}

\def\refname{References}

\end{document}